\documentclass[11pt,a4paper,reqno]{amsart}
\usepackage[all]{xy}
\usepackage{amsmath,amsfonts,amssymb,amscd}
\usepackage{soul,color}
\theoremstyle{plain}
\newtheorem{theorem}{Theorem}[section]
\newtheorem{mytheorem}{Theorem}[subsection]
\newtheorem{corollary}[mytheorem]{Corollary}
\newtheorem{Corollary}[theorem]{Corollary}
\newtheorem{lemma}[mytheorem]{Lemma}
\newtheorem{proposition}[mytheorem]{Proposition}
\newtheorem{definition}[mytheorem]{Definition}
\newtheorem{remark}[mytheorem]{Remark}

\newtheorem{example}[mytheorem]{Example}

\font\cyr wncyr10 at 11pt \def\Sha{\hbox{\cyr X}}

\def\tF{\textsc{F}}
\def\tV{\textsc{V}}

\def\AA{{\mathbb A}}
\def\CC{{\mathbb C}}

\def\FF{{\mathbb F}}

\def\NN{{\mathbb N}}

\def\QQ{{\mathbb Q}}
\def\RR{{\mathbb R}}
\def\ZZ{{\mathbb Z}}

\def\A{{\mathcal A}}
\def\B{{\mathcal B}}
\def\C{{\mathcal C}}

\def\E{{\mathcal E}}
\def\F{{\mathcal F}}
\def\cG{{\mathcal G}}
\def\cH{{\mathcal H}}

\def\K{{\mathcal K}}
\def\cL{{\mathcal L}}
\def\M{{\mathcal M}}
\def\N{{\mathcal N}}
\def\cO{{\mathcal O}}
\def\cP{{\mathcal P}}

\def\cS{{\mathcal S}}

\def\V{{\mathcal V}}

\newcommand{\fA}{\mathfrak{A}}

\newcommand{\fC}{\mathfrak{C}}

\newcommand{\fE}{\mathfrak{E}}

\newcommand{\fL}{\mathfrak{L}}
\newcommand{\fM}{\mathfrak{M}}

\newcommand{\fV}{\mathfrak{V}}
\newcommand{\fW}{\mathfrak{W}}

\newcommand{\fa}{\mathfrak{a}}
\newcommand{\fb}{\mathfrak{b}}
\newcommand{\fc}{\mathfrak{c}}
\newcommand{\fd}{\mathfrak{d}}
\newcommand{\fe}{\mathfrak{e}}
\newcommand{\ff}{\mathfrak{f}}

\newcommand{\fk}{\mathfrak{k}}

\newcommand{\fm}{\mathfrak{m}}
\newcommand{\fn}{\mathfrak{n}}

\newcommand{\fp}{\mathfrak{p}}

\newcommand{\fr}{\mathfrak{r}}
\newcommand{\fs}{\mathfrak{s}}

\newcommand{\kpinf}{{k}_{\infty}^{(p)}}
\newcommand{\ilim}{\varprojlim_n}
\newcommand{\Kpinf}{{K}_{\infty}^{(p)}}
\newcommand{\cal}{\mathcal}
\newcommand{\La}{\operatorname{\Lambda}}
\newcommand{\lra}{\longrightarrow}

\newcommand{\bo}{\mathbf}

\newcommand{\lr}{{\longrightarrow\;}}
\newcommand{\mbb}{\mathbb}

\newcommand{\dlim}{\varinjlim_n}

\def\+{{\dagger}}
\DeclareMathOperator{\Ker}{Ker}
\DeclareMathOperator{\Coker}{Coker}

\DeclareMathOperator{\Aut}{Aut}

\DeclareMathOperator{\coh}{H}
\DeclareMathOperator{\Frob}{Frob}
\DeclareMathOperator{\Fr}{Fr}

\DeclareMathOperator{\Gal}{Gal}

\DeclareMathOperator{\Ext}{Ext}
\DeclareMathOperator{\cExt}{\mathcal{E}xt}
\DeclareMathOperator{\G_m}{\mathbb{G}_m}
\DeclareMathOperator{\GL}{GL}

\DeclareMathOperator{\res}{res}

\DeclareMathOperator{\image}{Im}
\DeclareMathOperator{\Sel}{Sel}

\DeclareMathOperator{\Hom}{Hom}
\DeclareMathOperator{\Nm}{N}
\DeclareMathOperator{\Tor}{Tor}

\DeclareMathOperator{\Spec}{Spec}

\DeclareMathOperator{\rank}{rank}

\definecolor{purple}{rgb}{0.7,0,1}

\begin{document}
\title[ ] {On the Iwasawa Main conjecture of \\ abelian varieties  over function fields }

\author[Lai] {King Fai Lai}
\address{School of Mathematical Sciences\\
Capital Normal University\\
Beijing 100048, China}
\email{kinglaihonkon@gmail.com}

\author[Longhi] {Ignazio Longhi}
\address{Department of Mathematics\\
National Taiwan University\\
Taipei 10764, Taiwan}
\email{longhi@math.ntu.edu.tw}

\author[Tan]{Ki-Seng Tan}
\address{Department of Mathematics\\
National Taiwan University\\
Taipei 10764, Taiwan}
\email{tan@math.ntu.edu.tw}

\author[Trihan]{Fabien Trihan}
\address{Graduate School of Mathematics\\
Graduate School of Mathematics\\
Nagoya University\\
Chikusa-ku, Nagoya, 464-8602\\
Japan}
\email{trihan@math.nagoya-u.ac.jp}



\begin{abstract} We study a geometric analogue of the Iwasawa Main Conjecture for abelian varieties introduced by Mazur in the two following cases:
\begin{enumerate}
\item Constant ordinary abelian varieties over $\ZZ_p^d$-extensions of function fields  ramifying at a finite set of places.
\item Semistable abelian varieties over the arithmetic $\ZZ_p$-extension of a function field.
\end{enumerate}
One of the tools we use in our proof is a pseudo-isomorphism relating the duals of the Selmer groups of $A$ and its dual abelian variety $A^t$. This holds as well over number fields and is a consequence of a quite general algebraic functional equation.
\end{abstract}

\maketitle

\tableofcontents

\section{Introduction}  \label{sec:intro}

\subsection{Our main results}
We prove in this paper some cases of Iwasawa Main Conjecture for abelian varieties over
function fields and we also prove an algebraic functional equation for abelian varieties over global fields.
We fix a prime number $p$ ($p=2$ is allowed). Let $K$ be a global field of characteristic $0$ or $p$ and let $A$ be an abelian variety over $K$ of dimension $g$.
Let $L$ be a $\ZZ_p^d$-extension of $K$ ($d\geq 1$), unramified outside a finite set of places, with $\Gamma:=\Gal(L/K)$,
and write $Q(\Lambda)$ for the fraction field of the Iwasawa algebra $\La :=\ZZ_p[[\Gamma]]$.
Let $X_p(A/L)$ denote the Pontryagin dual of the Selmer group $\Sel_{p^\infty}(A/L)$ (the definition of the latter will be recalled in \S \ref{sss:sel} below).\\

We first assume that $K$ is of characteristic $p$ and consider the cases where either
\vskip5pt
\noindent (a) $A$ is a constant ordinary abelian variety,

or
\vskip5pt
\noindent (b) $A$ is an abelian variety with at worst semistable reduction, and $L$ is the arithmetic $\ZZ_p$-extension (which we shall denote by $\Kpinf$).
\vskip5pt

In both cases $X_p(A/L)$ is finitely generated over $\La$, hence we can define the characteristic ideal $\chi\big(X_p(A/L)\big)$ (\S\ref{sss:krul} below will recall the definition of characteristic ideals).
It is a principal ideal of $\Lambda$, and we let $c_{A/L}\in\Lambda$ denote a generator (called a characteristic element associated with $X_p(A/L)$), which is of course unique up to elements in $\Lambda^{\times}$.
Note that $c_{A/L}\not=0$ if and only if $X_p(A/L)$ is torsion.

For $\omega$ a continuous character of $\Gamma$ and $T$ a finite set of places of $K$, let $L_T(A,\omega,s)$ denote the twisted Hasse-Weil $L$-function of $A$ with the local factors at $T$ taken away. In case (a), $T$ will be the ramification locus of $L/K$ and in case (b), it will consist of those places where $A$ has bad reduction (see \S \ref{s:interpol} and \S\ref{ss:twistedHW} for more detail). Note that if the image of $\omega$ is contained in $\cO$, the ring of integers of a finite extension of $\QQ_p$, then $\omega$ can be uniquely extended to a continuous $\cO$-algebra homomorphism $\cO[[\Gamma]]\rightarrow \cO$. Our first main theorem is the following:

\begin{theorem} \label{t:summary} In both cases {\em{(a)}} and {\em{(b)}} above there exists a ``$p$-adic $L$-function'' $\cL_{A/L}\in Q(\La)$ such that for any continuous character
$\omega\colon\Gamma\rightarrow\CC^\times$, $\omega(\cL_{A/L})$ is defined and
\begin{equation}\label{e:imc2interpolate}
\omega(\cL_{A/L})=*_{A,T,\omega}\cdot L_T(A,\omega,1)
\end{equation}
for an explicit fudge factor $*_{A,T,\omega}$. Furthermore,
\begin{equation}\label{e:imc3characteristic}
\mathcal{L}_{A/L}\equiv \star_{A,L}\cdot c_{A/L}\mod\La^\times
\end{equation}
for a precise $\star_{A,L}\in Q(\La)^\times/\La^\times$.
\end{theorem}

Theorem \ref{t:summary} summarizes the main results of this paper: the interpolation formula \eqref{e:imc2interpolate} is proven in Theorems \ref{t:interpo} and  \ref{GIMC2} (case (a) and (b) respectively), while \eqref{e:imc3characteristic} is the content of Theorems \ref{t:imc3constant} and \ref{GIMC3}. For the precise expression of $*_{A,T,\omega}$ and $\star_{A,L}$ we refer to these theorems: here we just note that one has $\star_{A,L}=1$ in case (a) and $*_{A,T,\omega}=1$ in case (b).\\

Now consider the more general case:
\vskip5pt
\noindent
(c) $A$ has {\em{potentially ordinary}} reduction at each ramified place of $L/K$ while $K$ is of characteristic $0$ or $p$.
\vskip5pt

Obviously, case (c) contains both cases (a) and (b). It is known that in case (c) the module $X_p(A/L)$ is finitely generated over $\La$ (see \S\ref{se:se}), so the characteristic ideal $\chi\big(X_p(A/L)\big)$ and the characteristic element $c_{A/L}\in\Lambda$ are defined. Let $\cdot^{\sharp}\colon\La\rightarrow \La$, $\lambda\mapsto\lambda^{\sharp}$, denote the isomorphism induced by the inversion $\Gamma\rightarrow\Gamma$, $\gamma\mapsto \gamma^{-1}$ (see $\S$ \ref{ss:iwH1}). Let $A^t$ denote the dual abelian variety of $A$.

\begin{theorem}\label{t:xaat}
Suppose $A$ has potentially ordinary reduction at each ramified place of $L/K$. Then
$$\chi(X_p(A/L))=\chi( X_p(A/L))^{\sharp}=\chi( X_p(A^t/L)).$$
\end{theorem}

\noindent If $X_p(A/L)$ is not a torsion $\La$-module, the statement is obvious (by definition the characteristic ideal of a non-torsion finitely generated $\La$-module is $0$ and the last equality follows from the trivial Lemma \ref{l:isogAtA}). So the real content of Theorem \ref{t:xaat} is that if $X_p(A/L)$ is torsion we have
\begin{equation} \label{e:xaat} X_p(A/L)\sim X_p(A/L)^{\sharp}\sim X_p(A^t/L),\end{equation}
(where $\sim$ denotes pseudo-isomorphism of $\La$-modules). In the number field case, results in this direction were obtained in \cite[\S7]{mazur72}, \cite[\S8]{gr89} and \cite{p03} (see also \cite{za10} for a non-commutative generalization).\\

A key ingredient in the proof of Theorem \ref{t:xaat} will be a more general duality result, Theorem \ref{t:al}, which will also play a crucial role in case (a) of Theorem \ref{t:summary}.

\subsubsection{Motivation: elliptic curves over $\QQ$} To put our results in perspective let us review the first and most studied case in the Iwasawa theory of abelian varieties over global fields, i.e., when $K=\QQ$ and $A=E$ is an elliptic curve with good ordinary reduction at $p>2$. The only $\ZZ_p$-extension of $\QQ$ is the cyclotomic $\ZZ_p$-extension $\QQ_\infty$,
which can be viewed as the analogue of $\Kpinf$. We have  the following statements:
\begin{itemize}
\item[({\bf IMC1})] $X_p(E/\QQ_\infty)$ is a finitely generated torsion $\La$-module;
\item[({\bf IMC2})] there exists an element $\cL_{E/\QQ_\infty}\in\La $ such that for any finite character $\omega$ of $\Gamma$ we have
$$ \omega (\cL_{E/\QQ_\infty})=*_{E,p,\omega}\cdot L(E,\bar{\omega},1) $$
for an explicit fudge factor $*_{E,p,\omega}$;
\item[({\bf IMC3})] we have an equality of ideals of $\La$
$$(\cL_{E/\QQ_\infty})=(c_{E/\QQ_\infty}).$$
\end{itemize}
Here ({\bf IMC1}) follows from the works of Mazur, Rubin and Kato (\cite{mazur72,rubin91,Ka04}), and ({\bf IMC2}) was proved by Mazur and Swinnerton-Dyer (\cite{MS74}).\footnote{Under the assumption of modularity, but now we know that all elliptic curves defined over $\QQ$ are modular.}
As for ({\bf IMC3}) (usually referred to as the Iwasawa Main Conjecture), it was first proved by Rubin (\cite{rubin91}) when $E$ has complex multiplication.
Later Kato (\cite{Ka04}) proved that $c_{E/L}$ divides $\cL_E$ in $\La [\frac{1}{p}]$. In 2002, Skinner and Urban announced that $\cL_E$ divides $c_{E/L}$ assuming a conjecture on the existence of the Galois representation for automorphic forms on $U(2,2)$ and under certain conditions on $E$ (see \cite[Corollary 3.6.10]{su11}).

Note that for general $A$ and $K$ the analogue of ({\bf IMC1}) is not always true: it can happen that $c_{A/L}=0$, for example if $K$ is a quadratic extension of a number field $k$ and $L/k$ is dihedral, as may be deduced from \cite{MR07}. In the appendix, we give a family of similar examples of vanishing $c_{A/L}$ in characteristic $p$.
\footnote{As in the number field case, this happens in a situation where Heegner points appear; it might however be interesting to note that in the function field case the tower of extensions generated by Heegner points is quite bigger: its Galois group contains $\ZZ_p^\infty$! See e.g. \cite{Br04}.} However, $X_p (A/L)$ is a torsion $\La$-module in ``most'' cases by \cite[Theorem 7]{tan10b} (to which we refer for the precise meaning of ``most'').

\subsection{A closer look at our results}
In the rest of this introduction we survey the main ideas used in the paper and provide more precise statements of our results.
First we remark that the methods of our proof of Theorem \ref{t:summary} in cases (a) and (b), given in Parts \ref{part:constantA} and \ref{part:semistable} respectively, are extremely different. One should remember the difference between Iwasawa theory for CM and not CM elliptic curves in the number field case. Actually, in the function field setting, the CM abelian varieties are essentially those defined over the constant field and the existence of the Frobenius endomorphism will provide a key element in our proof in case (a); as mentioned above, we shall also need the algebraic functional equation of Section \ref{s:con}. As for case (b), our methods will be geometric, based on cohomological techniques.
The two cases have an intersection when $A$ is constant and $L$ is the arithmetic extension: in Subsection \ref{su:2dreams} we will check that the results coincide.

\subsubsection{Part \ref{part:algebraic}: the algebraic functional equation} Here we are in the setting of case (c): in particular, $K$ is allowed to be a number field.
In order to prove Theorem \ref{t:xaat} we introduce the notion of $\Gamma$-system.
Since this formalism can be applied in a wide range of contexts, we develop it axiomatically, as a means to deal with dualities in Iwasawa towers over any global field.
Basically, a $\Gamma$-system consists of two projective systems of finite abelian $p$-groups, $\{\fa_n\}$ and $\{\fb_n\}$, endowed with an action of $\Gamma$ and such that $\fa_n$ can be identified with the Pontryagin dual of $\fb_n$ for all $n$: see \S3.1 for the precise axioms. We also ask that the projective limits $\fa$ and $\fb$ are finitely generated torsion $\La$-modules.
\footnote{As we learned only after this paper had been essentially completed, our definition of $\Gamma$-system is very similar to the notion of ``normic system'' introduced in \cite[D\'efinition 2.1]{vau09}. The main difference is that Vauclair does not include a duality in his definition.}
Then the natural question is if there is any relation between the characteristic ideals of $\fa$ and $\fb^\sharp$: in Theorem \ref{t:al} we will prove that the two modules are actually pseudo-isomorphic under certain assumptions (one - that is, being pseudo-controlled - quite natural, the rest more technical).

In Section 5, we prove Theorem \ref{t:xaat} by reviewing some necessary background about the Cassels-Tate duality and then applying the $\Gamma$-system machinery to compare the $p$-Selmer groups of an abelian variety and its dual, when $X_p(A/L)$ is $\Lambda$-torsion. Subsection \ref{ss:idempotents} enounces some more precise results that can be obtained when the endomorphism algebra of $A$ is split at $p$.

\subsubsection{Part \ref{part:constantA}: the constant ordinary case} Here we deal with case (a). In \S\ref{ss:stickelb} we shall define the Stickelberger element $\theta_L\in\La$ as the value at $u=1$ of a certain power series $\Theta_L(u)\in\La[[u]]$: by construction $\theta_L$ interpolates Dirichlet $L$-functions $L(\omega,s)$ for all continuous characters $\omega\colon\Gamma\longrightarrow\CC^\times$. Following \cite{mazur72}, we associate with $A$ the twist matrix $\bo u$: let $\{\alpha_i\}$ be its set of eigenvalues (see \S\ref{ss:twist} below). It turns out that in case (a) of Theorem \ref{t:summary} the $p$-adic $L$-function $\mathcal{L}_{A/L}$ equals
$$\theta_{A,L}:=\prod \Theta_L(\alpha_i^{-1}) \Theta_L(\alpha_i^{-1})^{\sharp}\,.$$
As for the fudge factor $*_{A,T,\omega}$ in the interpolation formula \eqref{e:imc2interpolate}, it is somewhat complicated and will be defined only in section \ref{s:interpol}. Now we just note that here $T$ is the ramification locus of $L/K$ and $*_{A,T,\omega}$ comes mostly from the functional equation for $L(\omega,s)$ (necessary in order to deal with $\Theta_L(u)^{\sharp}$): see Theorem \ref{t:interpo}, which gives our analogue of ({\bf IMC2}) in this setting. (As for ({\bf IMC1}), $X_p(A/L)$ will be torsion, and hence $c_{A/L}\neq0$, when $L$ contains $\Kpinf$, by either \cite[Theorem 2]{tan10b} or \cite[Theorems 1.8 and 1.9]{ot09}).

Our analogue of the Iwasawa Main Conjecture ({\bf IMC3}) in case (a) is the following:

\begin{theorem} \label{t:imc3constant} Let $A$ be a constant ordinary abelian variety over the function field $K$. For any $\ZZ_p^d$-extension $L/K$ unramified outside a finite set of places, we have
\begin{equation} \label{e:imc3} \chi(X_p(A/L))=(\theta_{A,L}). \end{equation}
\end{theorem}

An immediate consequence is the following criterion, which might turn out useful for computationally verifying when $X_p(A/L)$ is $\La$-torsion.

\begin{Corollary} The Selmer group $\Sel_{p^\infty}(A/L)$ has positive $\La$-corank if and only if $\Theta(u)$ vanishes at some $\alpha_i^{-1}$.
\end{Corollary}

\noindent The main ingredient of our proof consists in splitting $X_p(A/L)$ in a Frobenius and Verschiebung part. When $L$ contains the arithmetic extension, the Frobenius part is related to the Iwasawa module of the trivial motive: by a theorem of Crew, the characteristic ideal of the latter is generated by the Stickelberger element $\theta_L$. This will lead us to "one half" of the conjecture; the second half is then deduced by means of Theorem \ref{t:al}, which in this setting yields a functional equation relating the Verschiebung part of $X_p(A/L)$ with the Frobenius part of $X_p(A^t/L)$ (Proposition \ref{p:sfft}). Finally, the case when $L$ does not contain the arithmetic extension is obtained by means of the main theorem of \cite{tan10b}.

\subsubsection{Part \ref{part:semistable}: the arithmetic extension case}\label{s:arithmetic}
Here we deal with case (b): $L$ is $\Kpinf$, the arithmetic $\ZZ_p$-extension (that is, the one obtained by extending only the constant field of the base field). In order to explain our results, we need first to introduce some cohomology groups and operators between them (these definitions will be repeated, in more detail, in section \ref{su:cohom}).

We write $C/\FF$ for  the smooth proper geometrically connected curve  which is the model of the function field $K$ over its field of constants $\FF$. Let $C_\infty:=C\times_{\FF} \kpinf$ (where $\kpinf$ denotes the $\ZZ_p$-extension of $\FF$) and $\pi\colon C_\infty\to C$ be the \'etale covering with Galois group $\Gal(\Kpinf/K)$.

Let $\cal A$ denote the N\'eron model of $A$ over $C$. Let  $Z$ the finite  set of points where $A$ has bad reduction. Denote by  $Lie(\A)$ the Lie algebra of $\A$. Let $L^i_\infty$ be the $i$th cohomology group of
$$\RR\Gamma\big(C_{\infty},\pi^*Lie(\A(-Z))\big)\otimes^{\mbb{L}}\QQ_p/\ZZ_p\,.$$
Let $D$ be the covariant log Dieudonn\'e crystal  associated with $A/K$ as constructed in \cite[IV]{KT03}. The syntomic complex ${\cal S}_D$ is the mapping fibre of ``$1-$Frobenius''  in the derived category of complexes of sheaves over $C_{\acute{e}t}$ (\cite[\S 5.8]{KT03}). Let $N^i_\infty$ be the $i$th cohomology group of
$$\RR\Gamma(C_\infty,\pi^*\cS_D)\otimes^{\mbb{L}}\QQ_p/\ZZ_p\,.$$

\begin{theorem}\label{GIMC1} For $i=0,1,2$ and $j=0,1$, the Pontryagin duals of $N^i_\infty$ and $L^j_\infty$ are finitely generated torsion $\La$-modules.
\end{theorem}

\noindent The proof shall be given in Corollaries \ref{M,L} and \ref{1case}. Note that, by \cite{KT03}, $X_p(A/\Kpinf)$ is a submodule of the dual of $N^1_\infty$, so Theorem \ref{GIMC1} provides a proof of the analogue of ({\bf IMC1}) in case (b). This was already known by \cite{ot09}, whose argument is simplified in the present paper.

Formula \eqref{e:charelement} in section \ref{su:charelem} will define $f_{A/\Kpinf}\in Q(\La)/\La^\times$ as the alternating product of the characteristic elements associated with the duals of $N^i_\infty$ and $L^j_\infty$ (in Part \ref{part:semistable} we prefer to work with characteristic elements rather than characteristic ideals mostly because some of these factors come with negative exponent). Because of the relation between $X_p(A/\Kpinf)$ and $N^1_\infty$, we can write
$$f_{A/K_{\infty}^{(p)}}=\star_{A,K_{\infty}^{(p)}}\cdot c_{A/K_{\infty}^{(p)}}\,$$
where $\star_{A,K_{\infty}^{(p)}}$ consists of terms whose arithmetic meaning is explained in Proposition \ref{p:summarize}.

We define the $p$-adic $L$-function $\cL_{A/\Kpinf}$ as the alternating product of determinants of the action of ``$1-$Frobenius'' on the log crystalline cohomology of $D(-Z)$ (see \S\ref{su:padicL} for the precise expression).
In the spirit of ({\bf IMC2}), Theorem \ref{GIMC2} proves that $\cL_{A/\Kpinf}$ satisifies the interpolation formula \eqref{e:imc2interpolate}, with $T=Z$ and $*_{A,\omega}=1$.

Finally, we can state our analogue of the Iwasawa Main Conjecture ({\bf IMC3}) in case (b).

\begin{theorem} \label{GIMC3} Let $A$ be an abelian variety with at worst semistable reduction relative to the arithmetic extension $\Kpinf/K$.
We have the following equality in $Q(\La)^\times/\La^\times${\em :}
$$\cL_{A/\Kpinf}=f_{A/\Kpinf}\,.$$
\end{theorem}

\noindent The proof is based on a generalization of a lemma of $\sigma$-linear algebra that was used to prove the cohomological formula of the Birch and Swinnerton-Dyer conjecture (see \cite[Lemma 3.6]{KT03}).

Subsection \ref{su:Eulerchar} investigates the consequences of Theorem \ref{GIMC3} in the direction of a $p$-adic Birch and Swinnerton-Dyer conjecture. The following result can be seen as a geometric analogue of the conjecture of Mazur-Tate-Teitelbaum (\cite{mtt86}):

\begin{theorem} Assume that $A/K$ has semistable reduction. Then
$$ord\big(\cL_{A/\Kpinf}\big)=ord_{s=1}\big(L_Z(A,s)\big)\geq\rank_\ZZ A(K)\,.$$
If moreover $A/K$ verifies the Birch and Swinnerton-Dyer Conjecture, the inequality above becomes an equality and
$$|L(\cL_{A/\Kpinf})|_p^{-1}\equiv  c_{BSD}\cdot |(N_\infty^2)_\Gamma| \mod \ZZ_p^\times \,,$$
where $c_{BSD}$ is the leading coefficient at $s=1$ of $L_Z(A,s)$.
\end{theorem}

\noindent Here $ord$ denotes the ``analytic rank'' and $L$ the ``leading coefficient'' of power series in $\La$ (see \S\ref{ss:twistEuler} for precise definitions). The proof will be provided in Theorems \ref{t:ranks} and \ref{euchar}. Proposition \ref{p:N2trivial} will prove that very often the error term $|(N_\infty^2)_\Gamma|$ is just 1.

We end Part \ref{part:semistable} by generalizing the constructions of chapter 9 and 10 to arbitrary log Dieudonn\'e crystals and by considering several open questions.

\begin{subsubsection}*{Acknowledgements} The fourth author has been supported by EPSRC. He would like also to express his gratitude to Takeshi Saito for his hospitality at the University of Tokyo where part of this work has been written. Authors 2, 3 and 4 thank Centre de Recerca Matem\`atica for hospitality while working on part of this paper.
Authors 1, 2 and 3 have been partially supported by the National Science Council of Taiwan, grants NSC98-2115-M-110-008-MY2, NSC100-2811-M-002-079 and NSC99-2115-M-002-002-MY3 respectively. Finally, it is our pleasure to thank NCTS/TPE for supporting a number of meetings of the authors in National Taiwan University. \end{subsubsection}

\begin{section}{Settings} \label{s:setting}

In this section we set notations for later use and recall some well-known facts about Iwasawa algebras and modules.\\

Let $K$ be a global field. In Parts \ref{part:constantA} and \ref{part:semistable} we will specify $K$ to be a characteristic $p$ function field, but for now we do not impose any restriction, since
results in Part \ref{part:algebraic} hold also in the number field case. As usual, for $v$ a place of $K$, we will write $K_v$ for the completion of $K$ at $v$.

We fix a $\ZZ_p^d$-extension $L/K$, $d\geq1$, with Galois group $\Gamma:=\Gal(L/K)$.
We recall that one important difference between number fields and function fields is that in the latter case there is room for many more $\ZZ_p$-extensions: actually, if $char(K)=p$ there is no bound on the $\ZZ_p$-rank of an abelian extension of $K$. As mentioned in the introduction, a distinguished one is the arithmetic $\ZZ_p$-extension $\Kpinf/K$, that is, the one obtained by extending only the constant field of the base field, which is unramified everywhere. All other $\ZZ_p$-extensions will ramify at some places (and some can even ramify at infinitely many places).

Let $S$ denote the ramification locus of $L/K$: a basic assumption throughout this paper is that $S$ is a finite set. Put
$\Gamma_n:=\Gamma/\Gamma^{p^n}\simeq (\ZZ_p/p^n\ZZ_p)^d$. Let $K_n$ denote the $n$th layer of $L/K$, so that $\Gamma_n=\Gal(K_n/K)$ and $\Gal(L/K_n)=\Gamma^{p^n}=:\Gamma^{(n)}$.

We shall be concerned with the complete group ring $\La:=\ZZ_p[[\Gamma]]$. The choice of a $\ZZ_p$-basis
$\{ \gamma_i\}$
for $\Gamma=\bigoplus_{i=1}^d\gamma_i^{\ZZ_p}\simeq\ZZ_p^d$ yields an isomorphism $\La\simeq\ZZ_p[[T_1,...,T_d]]$ where $T_i:=\gamma_i-1$. It follows that the Iwasawa algebra $\Lambda$ is indeed a unique factorization domain.

\begin{subsection}{Iwasawa algebras}\label{su:iw}
Even if our main interest lies in $\Lambda$, we are going to need Iwasawa algebras $\La(\Gamma'):=\ZZ_p[[\Gamma']]$ for other topologically finitely generated abelian $p$-adic Lie groups $\Gamma'$. Even more generally, assume that $\cO$ is the ring of integers of some finite extension of $\QQ_p$, and let $\La_\cO(\Gamma')$ denote the complete group ring $\cO[[\Gamma']]$. It has the usual topology as inverse limit
$$\varprojlim_{\Delta}\cO[\Gamma'/\Delta]$$
where $\Delta$ runs through all finite index subgroups and as a topological space $\cO[\Gamma'/\Delta]$ is just a finite product of copies of $\cO$. In the following, by $\La$- or $\La_\cO(\Gamma')$-module we will always mean a topological one, with continuous action of the scalar ring.

\subsubsection{} \label{ss:iwH1} Any group homomorphism $\phi\colon\Gamma'\rightarrow\La_\cO(\Gamma')^{\times}$ gives rise by linearity to a ring homomorphism $\cO[\Gamma']\rightarrow\La_\cO(\Gamma')$; if moreover $\phi$ is continuous, the latter map extends to $\phi\colon\La_\cO(\Gamma')\rightarrow\La_\cO(\Gamma')$. The most important occurrences in our paper will be the following.
\begin{enumerate}
\item[(H1)] The inversion $\Gamma'\rightarrow \Gamma'$, $\gamma\mapsto \gamma^{-1}$, gives rise to the isomorphism
$$\cdot^\sharp\colon\La_\cO(\Gamma')\,\lr \La_\cO(\Gamma')\,,$$
sending an element $\lambda$ to $\lambda^\sharp$.
\item[(H2)] Suppose $\phi\colon{\Gamma'}\rightarrow \cO^{\times}$ is a continuous homomorphism. Let
$$\phi^*\colon\La_\cO(\Gamma')\,\longrightarrow\La_\cO(\Gamma')$$
be the ring homomorphism determined by $\phi^*(\gamma):=\phi(\gamma)^{-1}\cdot\gamma$ for $\gamma\in\Gamma'$. Since on $\Gamma'$ the composition $\phi^*\circ (\phi^{-1})^*$ is the identity map, we see that $\phi^*$ is an isomorphism on $\La_\cO(\Gamma')$.
\end{enumerate}
The particular importance of the map $\cdot^\sharp$ for us stems from the fact that if $\langle\;,\;\rangle$ is a $\Gamma$-invariant pairing between $\La$-modules then
\begin{equation} \label{e:twistpairing} \langle \lambda\cdot a,b\rangle=\langle a,\lambda^\sharp\cdot b\rangle
\end{equation}
for any $\lambda\in\La$.
\end{subsection}


\begin{subsection}{Iwasawa modules}\label{su:ba} Assume that $\Gamma'$ is isomorphic to $\ZZ_p^n$ for some $n$, so that $\La_{\cO}(\Gamma')\simeq\cO[[T_1,...,T_n]]$. By definition a $\La_{\cO}(\Gamma')$-module $M$ is pseudo-null if and only if no height one prime ideal contains its annihilator (i.e., if for any height one prime $\fp$ the localization $M_\fp=0$ is trivial). A comprehensive reference is \cite[\S4]{bou65}.

\begin{lemma}\label{l:psn}
A finitely generated $\La_{\cO}(\Gamma')$-module $M$ is pseudo-null if and only if there exist relatively prime $f_1,...,f_k\in\La_{\cO}$, $k\geq2$, so that $f_iM=0$ for every $i$.
\end{lemma}

\begin{proof}
Since $\La_{\cO}(\Gamma')$ is a unique factorization domain, all height one prime ideals are principal and the claim follows.
\end{proof}

A pseudo-isomorphism is a homomorphism with pseudo-null kernel and cokernel. We will write $M\sim N$ to mean that there exists a pseudo-isomorphism from $M$ to $N$.

\begin{lemma} \label{l:pseudoisom} A composition of pseudo-injections {\em{(}}resp.\! pseudo-surjections,  resp.\! pseudo-isomorphisms{\em{)}} is a pseudo-injection {\em{(}}resp.\! pseudo-surjection,  resp.\! pseudo-isomorphism{\em{)}}. Pseudo-isomorphism is an equivalence relation in the category of finitely generated torsion $\La_{\cO}(\Gamma')$-modules.
\end{lemma}

\begin{proof}Let $\alpha\colon M\to N$ and $\beta\colon N\to P$ be two morphisms of $\La_{\cO}(\Gamma')$-modules. The first claim follows observing that there are exact sequences
$$0\longrightarrow \Ker(\alpha)\lr \Ker(\beta\circ\alpha)\lr \image(\alpha)\cap \Ker(\beta)\longrightarrow 0$$
and
$$ \Coker(\alpha)\lr \Coker(\beta\circ\alpha)\lr \Coker(\beta)\longrightarrow 0.$$
For the second statement, the only thing left to prove is symmetry. Let $\alpha\colon M\to N$ be a pseudo-isomorphism. Let $T$ be the set of height one primes containing $Ann_{\La_{\cO}(\Gamma')}(M)$ and put $S:=(\La_{\cO}(\Gamma')-\cup_{\fp\in T}\fp)$. The map $\alpha\otimes1\colon S^{-1}M\to S^{-1}N$ is an isomorphism of $S^{-1}\La_{\cO}(\Gamma')$-modules: let $\beta$ be its inverse. Then
$$\Hom_{S^{-1}\La_{\cO}(\Gamma')}(S^{-1}N,S^{-1}M)\simeq S^{-1}\Hom_{\La_{\cO}(\Gamma')}(N,M)$$
implies $s\beta\in\Hom_{\La_{\cO}(\Gamma')}(N,M)$ for some $s\in S$ and $s\beta$ is the required pseudo-isomorphism. For more details, the reader is referred to the proof of \cite[\S4, no.\,4, Th.\,5]{bou65}.
\end{proof}

\begin{lemma} \label{l:pseudoinj}
Let $\alpha\colon M\to N$ and $\beta\colon N\to M$ be two pseudo-injections of finitely generated torsion $\La_{\cO}(\Gamma')$-modules. Then $\beta\circ\alpha$ is a pseudo-isomorphism.
\end{lemma}

\begin{proof} By hypothesis, the localized maps $\alpha_\fp$ and $\beta_\fp$ are injective for all $\fp$ of height one. Besides, $M_\fp$ and $N_\fp$ are finitely generated torsion $\La_{\cO}(\Gamma')_\fp$-modules. We show that $\beta_\fp\circ\alpha_\fp$ is surjective.

Let $r$ be such that $\fp^r M_\fp = 0$. Also, let $\kappa(\fp)$ be the residue field of the discrete valuation ring $\La_{\cO}(\Gamma')_\fp$. We use induction on $r$. If $r =1$, then the composite map $\beta_\fp \circ \alpha_\fp$ is an injection of the finite dimensional $\kappa(\fp)$-vector space $M_\fp$ into itself: hence it has to be surjective. For $r> 1$, the endomorphism induced by $\beta_\fp \circ \alpha_\fp$ on $m_\fp:= M_\fp/M_\fp[\fp]$ is injective and $\fp^{r-1} m_\fp = 0$, so the map is surjective on $m_\fp$ (induction hypothesis) and on $M_\fp[\fp]$ (by the first step): hence it is surjective also on $M_\fp$.
\end{proof}

\subsubsection{}\label{sss:krul} We are still assuming $\Gamma'\simeq\ZZ_p^n$. Suppose $M$ is a finitely generated torsion $\La_{\cO}(\Gamma')$-module. By the general theory of modules over a Krull domain, there is a pseudo-isomorphism
\begin{equation} \label{e:structureM} \Phi\colon\bigoplus_{i=1}^m \La_{\cO}(\Gamma')/\xi_i^{r_i}\La_{\cO}(\Gamma')\longrightarrow M,\end{equation}
where each $\xi_i$ is irreducible. We shall denote $[M]:=\bigoplus_{i=1}^m \La_{\cO}(\Gamma')/\xi_i^{r_i}\La_{\cO}(\Gamma')$.
Since a non-zero element in $[M]$ cannot be simultaneously annihilated by relatively prime elements of $\La_{\cO}(\Gamma')$, there is no non-trivial pseudo-null submodule of $[M]$, and hence $\Phi$ is an embedding. We know that $[M]$ is uniquely determined by $M$ up to isomorphism, but $\Phi$ is not uniquely determined by $M$. However, we shall fix one such $\Phi$ and view $[M]$ as a submodule of $M$. \footnote{The notation $[\,\cdot\,]$ is intentionally reminiscent of the one denoting the integral part of a number.}

For a finitely generated $\La$-module $M$, let $\chi_{\cO,\Gamma'}(M)\subset\La_{\cO}$ denote its characteristic ideal: in the notations of \eqref{e:structureM} it is
$$\chi_{\cO,\Gamma'}(M):=\prod_{i=1}^m (\xi_i^{r_i})$$
if $M$ is torsion and $0$ if not. If $\cO=\ZZ_p$ and $\Gamma'=\Gamma$, we write $\chi:=\chi_{\cO,\Gamma'}$.

\subsection{Twists} \label{su:twistmodule}
Let $M$ be a $\La_{\cO}(\Gamma')$-module. Any endomorphism $\alpha\colon\La_{\cO}(\Gamma')\to\La_{\cO}(\Gamma')$ defines a twisted $\La_{\cO}(\Gamma')$-module
$\La_{\cO}(\Gamma')\,{}_\alpha\!\otimes_{\La_{\cO}(\Gamma')}M\,,$ where the action on the copy of $\La_{\cO}(\Gamma')$ on the left is via $\alpha$ (i.e., we have $(\alpha(\lambda)\mu)\otimes m=\mu\otimes\lambda m$ for $\lambda$, $\mu\in\La_{\cO}(\Gamma')$ and $m\in M$) and the module structure is given by
\begin{equation} \label{e:twistaction} \lambda\cdot(\mu\otimes m) := (\lambda\mu)\otimes m \end{equation}
(where $\lambda\mu$ is the product in $\La_{\cO}(\Gamma')$ ).
If moreover $\alpha$ is an isomorphism, $\La_{\cO}(\Gamma')\,{}_\alpha\!\otimes_{\La_{\cO}(\Gamma')}M$ can be identified with $M$ with the $\La_\cO(\Gamma')$-action twisted by $\alpha^{-1}$, since in this case \eqref{e:twistaction} becomes
\begin{equation} \label{e:twistaction2} \lambda\cdot(1\otimes m)=1\otimes\alpha^{-1}(\lambda)m \,.\end{equation}

\begin{lemma}\label{l:phi[]chi} Let $\alpha$ be an automorphism of $\La_\cO(\Gamma')\simeq\cO[[T_1,...,T_n]]$. Suppose $M$ is a finitely generated torsion $\La_{\cO}(\Gamma')$-module with
$$[M]=\bigoplus_{i=1}^m \La_{\cO}(\Gamma')/\xi_i^{r_i}\La_{\cO}(\Gamma')\,.$$
Then
$$[\La_{\cO}(\Gamma')\,{}_\alpha\!\otimes_{\La_{\cO}(\Gamma')}M]=\La_{\cO}(\Gamma')\,{}_\alpha\!\otimes_{\La_{\cO}(\Gamma')}[M]=\bigoplus_{i=1}^m \La_{\cO}(\Gamma')/\alpha(\xi_i)^{r_i}\La_{\cO}(\Gamma'),$$
and hence
$$\chi_{\cO,\Gamma'}\big(\La_{\cO}(\Gamma')\,{}_\alpha\!\otimes_{\La_{\cO}(\Gamma')}M\big)=\alpha(\chi_{\cO,\Gamma'}(M)).$$
\end{lemma}

\begin{proof}
It is immediate from \eqref{e:twistaction2} that if a pseudo-null $N$ is annihilated by coprime $f_1$, $f_2$, then $\La_{\cO}(\Gamma')\,{}_\alpha\!\otimes N$ is annihilated by coprime $\alpha(f_1)$, $\alpha(f_2)$, whence pseudo-null. Thus the functor $\La_{\cO}(\Gamma')\,{}_\alpha\!\otimes-$ preserves pseudo-isomorphisms. Since it also commutes with direct sums, we are reduced to check the equality
$$\La_{\cO}(\Gamma')\,{}_\alpha\!\otimes\big(\La_{\cO}(\Gamma')/\xi_i^{r_i}\La_{\cO}(\Gamma')\big)=\La_{\cO}(\Gamma')/\alpha(\xi_i)^{r_i}\La_{\cO}(\Gamma')\,,$$
which is obvious by exactness of $\La_{\cO}(\Gamma')\,{}_\alpha\!\otimes-$.
\end{proof}

We are going to apply the above with the isomorphisms considered in (H1), (H2) of \S\ref{ss:iwH1}. In particular, we shall write
$$M^\sharp:=\La_{\cO}(\Gamma')\,{}_\sharp\!\otimes_{\La_{\cO}(\Gamma')}M$$
and, for $\phi$ as in (H2),
\begin{equation} \label{e:twistbyphi} M(\phi):=\La_{\cO}(\Gamma')\,{}_{\phi^*}\!\otimes_{\La_{\cO}(\Gamma')}M. \end{equation}
Since $\cdot^\sharp$ is an involution, \eqref{e:twistaction2} shows that the action of $\La_{\cO}(\Gamma')$ becomes $\lambda\cdot m=\lambda^\sharp m$.
As for $\phi^*$, note that, if we endow $\cO$ with the trivial action of $\Gamma'$, then the $\La_{\cO}$-module
$\cO(\phi)$ can be viewed as the free rank one $\cO$-module with the action of $\Gamma'$ through multiplication by $\phi$, in the sense that
$$\gamma \cdot a=\phi(\gamma) a \text{ for all } \gamma\in\Gamma', a\in \cO(\phi)\,.$$
Then for a $\La_{\cO}$-module $M$ we have
$$M(\phi)=\cO(\phi)\otimes_{\cO}M,$$
where $\Gamma'$ acts by
$$\gamma\cdot(a\otimes x):=(\gamma\cdot a)\otimes (\gamma\cdot x)=\phi(\gamma)\cdot (a\otimes \gamma x)\,.$$

\subsection{Some more notations} In order to lighten the notation, for an $\cO$-module $\fM$ and an $\cO$-algebra $R$, we will often use the shortening
$$R\fM:=R\otimes_{\cO}\fM\,.$$

The Pontryagin dual of an abelian group $B$ will be denoted $B^\vee$. Since we are going to deal mostly with finite $p$-groups and their inductive and projective limits, we generally won't distinguish between the Pontryagin dual and the set of continuous homomorphisms into the group of roots of unity $\boldsymbol\mu_{p^{\infty}}:=\cup_m\boldsymbol\mu_{p^{m}}$. Note that we shall usually think of $\boldsymbol\mu_{p^{\infty}}$ as a subset of $\bar\QQ_p$ (hence with the discrete topology), so that for a $\La$-module $M$ homomorphisms in $M^\vee$ will often take value in $\bar\QQ_p$.

We shall denote the $\psi$-part of a $G$-module $M$ (for $G$ a group and $\psi\in G^\vee$) by
\begin{equation} \label{e:psipart} M^{(\psi)}:=\{x\in M\;\mid\; g\cdot x=\psi(g)x\;\text{for all}\;g\in G\}. \end{equation}

\end{subsection}
\end{section}


\begin{part}{The algebraic functional equation} \label{part:algebraic}

\begin{section}{Controlled $\Gamma$-systems and the algebraic functional equation} \label{s:con}

\begin{subsection}{$\Gamma$-systems}\label{su:ga}
Consider a collection
$$\fA=\{\fa_n,\fb_n,\langle\;,\;\rangle_n,\fr_m^n,\fk_m^n\;\mid\; n,m\in\NN\cup\{0\},\; n\geq m\}$$
such that
\begin{enumerate}
\item[($\Gamma$-1)]  $\fa_n, \fb_n$ are finite abelian groups, with an action of $\La$ factoring through $\ZZ_p[\Gamma_n]$.
\item[($\Gamma$-2)] For $n\geq m$,
$$\fr_m^n\colon\fa_m\times \fb_m\longrightarrow \fa_n\times \fb_n\,,$$
$$\fk_m^n\colon\fa_n\times \fb_n\longrightarrow \fa_m\times \fb_m$$
    are $\Gamma$-morphisms such that $\fr_m^n(\fa_m)\subset \fa_n$, $\fr_m^n(\fb_m)\subset \fb_n$, $\fk_m^n(\fa_n)\subset \fa_m$, $\fk_m^n(\fb_n)\subset \fb_m$ and $\fr^n_n=\fk^n_n={\rm id}$. Also, $\{\fa_n\times \fb_n, \fr_m^n\}_n$ form an inductive system and $\{\fa_n\times \fb_n, \fk_m^n\}_n$ form a projective system.
\item[($\Gamma$-3)] We have
$$\fr_m^n\circ\fk_m^n=\Nm_{\Gamma_n/\Gamma_m}\colon\fa_n\times \fb_n\longrightarrow \fa_n\times \fb_n$$
(where $\Nm_{\Gamma_n/\Gamma_m}:=\sum_{\sigma\in\Ker(\Gamma_n\rightarrow\Gamma_m)}\sigma$ is the norm associated with $\Gamma_n\twoheadrightarrow\Gamma_m$) and
$$\fk_m^n \circ \fr_m^n=p^{d(n-m)}\cdot \text{id}\colon\fa_m\times \fb_m\longrightarrow \fa_m\times \fb_m.$$
\item[($\Gamma$-4)] For each $n$, $\langle\;,\;\rangle_n\colon\fa_n\times \fb_n\longrightarrow \QQ_p/\ZZ_p$ is a perfect pairing (and hence $\fa_n$ and $\fb_n$ are dual $p$-groups) respecting $\Gamma$-action as well as the morphisms $\fr_m^n$ and $\fk_m^n$ in the sense that
$$\langle \gamma\cdot a, \gamma\cdot b\rangle_n=\langle a,b\rangle_n\;\forall\,\gamma\in\Gamma,$$
\begin{equation} \label{e:rkpairing} \langle a,\fr_m^n(b)\rangle_n=\langle\fk_m^n(a),b\rangle_m \end{equation}
    and
$$\langle\fr_m^n(a),b\rangle_n=\langle a,\fk_m^n(b)\rangle_m.$$
\end{enumerate}
Write
$$\fa:=\lim_{\stackrel{\leftarrow}{n}}\fa_n\;\;\text{and}\;\;\fb:=\lim_{\stackrel{\leftarrow}{n}}\fb_n\,.$$
In the following, we let $\fk_n$ denote the natural map
$$\fa\times \fb\longrightarrow \fa_n\times \fb_n\,.$$

\begin{definition} \label{d:gamsys} We say $\fA$ as above is a $\Gamma$-system if both $\fa$ and $\fb$ are finitely generated torsion $\La$-modules. \end{definition}

Definition \ref{d:gamsys} can be extended to the notion of a {\bf complete} $\Gamma$-system, for which we stipulate that for each finite intermediate extension $F$ of $L/K$ there are $\Gal(F/K)$-modules $\fa_F$ and $\fb_F$ with a pairing $\langle\;,\;\rangle_F$, and for any pair $F$, $F'$ of finite intermediate extensions with $F\subset F'$, there are $\Gamma$-morphisms $\fr_F^{F'}$ and  $\fk_F^{F'}$ satisfying the obvious analogues of ($\Gamma$-1)-($\Gamma$-4).

We say that $\fA$ is part of a complete $\Gamma$-system $\{\fa_F,\fb_F,\langle\;,\;\rangle_F,\fr_F^{F'},\fk_F^{F'}\}$ if $\fa_n=\fa_{K_n}$, $\fb_n=\fb_{K_n}$, $\fr_n^m=\fr_{K_n}^{K_m}$ and $\fk_n^m=\fk_{K_n}^{K_m}$. Obviously this implies $\fa=\varprojlim_{F}\fa_F$ and $\fb=\varprojlim_{F}\fb_F$.

\subsubsection{Some more definitions} We say that the $\Gamma$-system $\fA$ is {\bf twistable} of order $k$ if there exists an integer $k$ such that $p^{n+k}\fa_n=0$ for every $n$ (this definition shall be justified in \S\ref{ss:twistgam} below).

\begin{definition}\label{d:simple}
An element $f\in\Lambda(\Gamma)$ is simple if there exist $\gamma\in\Gamma- \Gamma^p$ and $\zeta\in\boldsymbol\mu_{p^\infty}$ so that
$$f=f_{\gamma,\zeta}:=\prod_{\sigma\in\Gal(\QQ_p(\zeta)/\QQ_p)} (\gamma-\sigma(\zeta)).$$
\end{definition}

It is easy to check that simple elements are irreducible in $\La$ and that
\begin{equation} \label{e:coprimecriter}(f_{\gamma',\zeta'})=(f_{\gamma,\zeta}) \Longleftrightarrow (\gamma')^{\ZZ_p}=\gamma^{\ZZ_p}\text{ and } \zeta'\in\Gal(\QQ_p(\zeta)/\QQ_p)\cdot\zeta\,. \end{equation}
In particular, we have
\begin{equation}\label{e:fs}
(f_{\gamma,\zeta})^{\sharp}=(f_{\gamma^{-1},\zeta})=(f_{\gamma,\zeta}).
\end{equation}

Assume that $\fA$ is a complete $\Gamma$-system. Let $F$ be a finite intermediate extension and let $L'/F$ be an intermediate $\ZZ_p^{e}$-extension of $L/F$. Write
$$\fa_{L'/F}=\varprojlim_{F\subset F'\subset L'} \fa_{F'},\;\text{and}\;  \fb_{L'/F}=\varprojlim_{F\subset F'\subset L'} \fb_{F'}.$$
They are modules over $\Lambda_{L'/F}:=\ZZ_p[[\Gal(L'/F)]]$. Set the condition
\vskip5pt
\noindent ({\bf T}): {\em{For every finite intermediate extension $F$ and every intermediate $\ZZ_p^{d-1}$-extension $L'/F$ of $L/F$, $\fa_{L'/F}$ and $\fb_{L'/F}$ are finitely generated and torsion over $\La_{L'/F}$}}.
\vskip5pt

\subsubsection{Morphisms} \label{ss:mor}
A $\Gamma$-system $\fA=\{\fa_n,\fb_n,\langle\;,\;\rangle_n,\fr_m^n,\fk_m^n\}$ is oriented if we have fixed an order of the pairs $(\fa_n,\fb_n)$. We define a morphism of oriented $\Gamma$-systems
$$\fA=\{\fa_n,\fb_n,\langle\;,\;\rangle_n^{\fA},\fr(\fA)_m^n,\fk(\fA)_m^n\} \longrightarrow\fC=\{\fc_n,\fd_n,\langle\;,\;\rangle_n^{\fC},\fr(\fC)_m^n,\fk(\fC)_m^n\}$$
to be a collection of morphisms of $\Gamma$-modules $f_n\colon\fa_n\rightarrow\fc_n$, $g_n\colon\fd_n\rightarrow\fb_n$ commuting with the structure maps and such that $\langle f_n(a),d\rangle_n^{\fC}=\langle a,g_n(d)\rangle_n^{\fA}$ for all $n$.

A pseudo-isomorphism of (oriented) $\Gamma$-systems is a morphism $\fA\rightarrow\fC$ such that the induced maps $\fa\rightarrow\fc$, $\fd\rightarrow\fb$ are pseudo-isomorphisms of $\Gamma$-modules.

\begin{example} \label{eg:morf} {\em Given an oriented $\Gamma$-system $\fA=\{\fa_n,\fb_n\}$ and $\lambda\in\La$, we can define $\lambda\cdot\fA:=\{\lambda\fa_n,\lambda^\sharp\fb_n\}$  and $\fA[\lambda]:=\{\fa_n[\lambda],\fb_n/\lambda^\sharp\fb_n\}$, with the pairing and the transition maps induced by those of $\fA$. It is easy to check that $\lambda\cdot\fA$ and $\fA[\lambda]$ are $\Gamma$-systems and that the exact sequences $\fa_n[\lambda]\hookrightarrow\fa_n\twoheadrightarrow\lambda\fa_n$  and $\lambda^\sharp\fb_n \hookrightarrow\fb_n\twoheadrightarrow\fb_n/\lambda^\sharp\fb_n$ provide morphisms of oriented $\Gamma$-systems $\fA[\lambda]\rightarrow\fA$ and $\fA\rightarrow\lambda\cdot\fA$.} \end{example}

\subsubsection{The algebraic functional equation}\label{ss:ge}
The pairing in ($\Gamma$-4) induces for any $n$ an isomorphism of $\La$-modules between $\fa_n^\sharp$ and $\fb_n^\vee$ (the twist by the involution $^\sharp$ is due to \eqref{e:twistpairing}). The main question we are going to ask about $\Gamma$-systems is if the following holds:
\begin{equation}\label{e:basharp}
\fb\sim \fa^{\sharp}.
\end{equation}
A weaker question is if we have the algebraic functional equation:
\begin{equation}\label{e:chisharp}
\chi(\fb)=\chi(\fa)^{\sharp}.
\end{equation}

Next we give our answers to these questions.

\begin{definition} \label{d:pscontr}
Given a $\Gamma$-system $\fA$ we put
$$\fa^0\times \fb^0:=\lim_{\stackrel{\leftarrow}{n}}  \bigcup_{n'\geq n}\Ker(\fr_n^{n'}).$$
A $\Gamma$-system $\fA$ is {\em pseudo-controlled}  if $\fa^0\times \fb^0$ is pseudo-null.
\end{definition}

The following, proved in in $\S$\ref{se:al}, is our theorem on the {\em algebraic functional equation}.

\begin{mytheorem}\label{t:al}
Let $$\fA=\{\fa_n,\fb_n,\langle\;,\;\rangle_n,\fr_m^n,\fk_m^n\;\mid\; n,m\in\NN,\; n\geq m\}$$ be a pseudo-controlled $\Gamma$-system.
Then there is a pseudo-isomorphism
$$\fa^\sharp\sim \fb$$
in the following three cases:\begin{enumerate}
\item there exists $\xi\in\La$ not divisible by any simple element and such that $\xi\fb$ is pseudo-null;
\item $\fA$ is pseudo-isomorphic to a twistable pseudo-controlled $\Gamma$-system;
\item $\fA$ is part of a complete $\Gamma$-system enjoying property {\em({\bf T})}.
\end{enumerate}
\end{mytheorem}

If $\fA$ is not pseudo-controlled, it seems quite unlikely that the functional equation \eqref{e:chisharp} holds at all.
Indeed, as the modules $\fa/\fa^0$, $\fb/\fb^0$ come from a $\Gamma$-system (see \S\ref{ss:a'} below), the theorem yields $\chi(\fa/\fa^0)^\sharp=\chi(\fb/\fb^0)$, so for \eqref{e:chisharp} to be true we need equality between $\chi(\fa^0)^\sharp$ and $\chi(\fb^0)$. But we see no reason to expect such an equality.

\begin{remark}
\label{r:aefail}
{\em In the classical case of $d=1$ (i.e., $\Gamma\simeq\ZZ_p$) the functional equation \eqref{e:chisharp} can be obtained as a consequence of results by Mazur and Wiles.  More precisely, by \cite[Appendix]{MW84} one obtains an equality of Fitting ideals
$$Fitt_\Lambda(\fb)=\cap Fitt_\Lambda(\fb_n^\vee)=\cap Fitt_\Lambda(\fa_n^\sharp)=Fitt_\Lambda(\fa^\sharp)$$
(apply {\em loc.\,cit.},  Proposition 3 together with Property 10 on page 325). The Fitting ideal is always contained in the characteristic ideal and they are equal if the former is principal (see e.g.\,\cite[Lemma 5.10]{bl09a}), so \eqref{e:chisharp} follows in many cases.

For $d\geq 2$, Fitting ideals do not seem to yield a promising approach for a proof of \eqref{e:chisharp}. Some ideas for the case $d=2$ can be found in \cite[Appendix]{gk08} (but note that the dual they study is not the Pontryagin dual and the need for a reduced $R$ in their Theorem A.3 prevents us from mimicking the proof of their Theorem A.8 in the case of our finite $\fa_n$, $\fb_n$); however, \cite[Remarks A.9 and A.10]{gk08} suggest that there is no hope for $d\geq3$.}
\end{remark}

\subsubsection{Derived systems}\label{su:de}
Suppose for each $n$ we are given a $\Gamma$-submodule $\fc_n\subset \fa_n$ such that $\fr_m^n(\fc_m)\subset \fc_n$ and $\fk_m^n(\fc_n)\subset \fc_m$. Using these, we can obtain two derived $\Gamma$-systems from $\fA$. Let $\ff_n\subset \fb_n$ be the annihilator of $\fc_n$, via the duality induced from $\langle\;,\;\rangle_n$, and let $\fd_n:=\fb_n/\ff_n$.
Then we also have $\fr_m^n(\ff_m)\subset \ff_n$ and $\fk_m^n(\ff_n)\subset \ff_m$.
Hence $\fr_m^n$ induces a morphism $\fc_m\times \fd_m\rightarrow \fc_n\times \fd_n$, which, by abuse of notation, we also denote as $\fr_m^n$. Similarly, we have the morphism $\fk_m^n\colon\fc_n\times \fd_n\rightarrow \fc_m\times \fd_m$ and the pairing $\langle\;,\;\rangle_n$ on $\fc_n\times \fd_n$. Let $\fC$ denote the $\Gamma$-system
$$\{\fc_n,\fd_n,\langle\;,\;\rangle_n,\fr_m^n,\fk_m^n\;\mid\; m,n\in \NN, n\geq m\}.$$
We also write $\fe_n:=\fa_n/\fc_n$ and let $\fE$ denote the $\Gamma$-system
$$\{\fe_n,\ff_n,\langle\;,\;\rangle_n,\fr_m^n,\fk_m^n\;\mid\; m,n\in \NN, n\geq m\}.$$
Then we have the sequences
\begin{equation}\label{e:c}
0\longrightarrow \fc \longrightarrow \fa\longrightarrow \fe\longrightarrow 0
\end{equation}
and
\begin{equation}\label{e:e}
0\longrightarrow \ff \longrightarrow \fb\longrightarrow \fd\longrightarrow 0.
\end{equation}
Here $\fc$, $\fd$, $\fe$ and $\ff$ are the obvious projective limits; the systems $\{\fc_n\}$ and $\{\ff_n\}$ satisfy the Mittag-Leffler condition (because all groups are finite), so \eqref{e:c} and \eqref{e:e} are exact.

\begin{lemma} \label{l:e0f0} Assume $\fa_n=\fk_n(\fa)$ for all $n$. Then $\fe\sim 0$ implies $\ff\sim 0$.
\end{lemma}

\begin{proof} The assumption implies $\fk_n(\fe)=\fe_n$. Thus $f\cdot \fe=0$ implies $f\cdot \fe_n=0$, and consequently, by the duality, $f^{\sharp}\cdot \ff_n=0$ for all $n$, yielding $f^{\sharp}\cdot \ff=0$. Now apply Lemma \ref{l:psn}.
\end{proof}

\subsubsection{The system $\fA'$} \label{ss:a'}
In the following case, we apply the above two methods together. We first get a system $\{\fa_n/\fa^0_n,\fb^1_n\}$ by putting
\begin{equation} \label{e:dirlim0} \fa_n^0\times \fb_n^0:=\bigcup_{n'\geq n}\Ker(\fr_n^{n'})= \Ker\!\big(\fa_n\times \fb_n \longrightarrow \lim_{\stackrel{\longrightarrow}{m}}\fa_m\times \fb_m\big) \end{equation}
and letting $\fa_n^1$, $\fb_n^1$ be respectively the annihilators of $\fb_n^0$, $\fa_n^0$, via $\langle\;,\;\rangle_n$.
Then we apply the $\fC$-construction to $\{\fa_n/\fa^0_n,\fb^1_n\}$ defining $\fa_n'\subset\fa_n/\fa_n^0$ via
$$\fa_n'\times \fb_n':=\image\!\big(\fa_n^1\times \fb_n^1\longrightarrow (\fa_n/\fa_n^0)\times (\fb_n/\fb_n^0)\big)\,.$$
Notice that $\fb_n'$ is dual to $\fa_n'$, as can be seen by dualizing the diagram
\[\begin{CD}
0 @>>> \fa^1_n @>>> \fa_n \\
&& @VVV @VVV \\
0 @>>> \fa_n' @>>> \fa_n/\fa^0_n  \end{CD}\]
(recall that the duals of $\fa^1_n$ and $\fa_n/\fa^0_n$ are respectively $\fb_n/\fb_n^0$ and $\fb^1_n$).\\
Thus we get a $\Gamma$-system
$$\fA':=\{\fa_n',\fb_n',\langle\;,\;\rangle_n,\fr_m^n,\fk_m^n\;\mid\; m,n\in \NN, n\geq m\}.$$

Denote, for $i=0,1$,
$$\fa^i\times \fb^i:=\lim_{\stackrel{\leftarrow}{n}} \fa_n^i\times \fb_n^i,$$
and
$$\fa'\times \fb':=\lim_{\stackrel{\leftarrow}{n}}\fa_n'\times \fb_n'=\image\!\big(\fa^1\times\fb^1\longrightarrow (\fa/\fa^0)\times (\fb/\fb^0)\big).$$
The pairings $\langle\;,\;\rangle_n$ allow identifying each $\fa_n\times\fb_n$ with its own Pontryagin dual and this identification is compatible with the maps $\fr_m^n$, $\fk_m^n$. Then $\fa\times\fb$ is the dual of $\displaystyle \lim_{\rightarrow} \fa_n\times\fb_n$. Consider the exact sequence
\begin{equation}   \begin{CD} \label{e:exseqX}
0 @>>> \fa_n^0\times\fb^0_n @>>> \fa_n\times\fb_n @>>> (\fa_n\times\fb_n)/(\fa_n^0\times\fb^0_n) @>>> 0. \end{CD}\end{equation}
By construction, $\fa_n^1\times\fb^1_n$ is the dual of $(\fa_n\times\fb_n)/(\fa_n^0\times\fb^0_n)$. The inductive limit of \eqref{e:exseqX} gets the identity
$$\lim_{\rightarrow} \fa_n\times\fb_n=\lim_{\rightarrow}\, (\fa_n\times\fb_n)/(\fa_n^0\times\fb^0_n)$$
($\varinjlim\fa_n^0\times\fb_n^0=0$ is immediate from \eqref{e:dirlim0}) and hence, taking duals,
$$\fa^1\times \fb^1=\fa\times \fb.$$
Thus we have an exact sequence
\begin{equation}\label{e:0'}
0\longrightarrow \fa^0\times \fb^0\longrightarrow \fa\times \fb\longrightarrow \fa'\times \fb'\longrightarrow 0.
\end{equation}

\subsubsection{Strongly-controlled $\Gamma$-systems}\label{su:st}
In the previous section we saw that, as $\fb=\fb^1$ and $\fa=\fa^1$ , the information carried by $\fa^0$ and $\fb^0$
does not pass to $\fb$ and $\fa$: this explains Definition \ref{d:scontr}. Here we consider a condition stronger than being pseudo-controlled.

\begin{definition}\label{d:scontr}
A $\Gamma$-system $\fA$ is strongly controlled if  $\fa_n^0\times \fb_n^0=0$ for every $n$.
\end{definition}

\begin{lemma}\label{l:eq}
A $\Gamma$-system $\fA$ is strongly controlled if and only if $\fr_m^n$ is injective  {\em{(}}resp. $\fk_m^n$ is surjective{\em{)}} for $n\geq m$.
\end{lemma}

\begin{proof} The definition and the duality. \end{proof}

\begin{lemma}\label{l:st}
Suppose $\fA$ is a $\Gamma$-system. Then the following holds:
\begin{enumerate}
\item the system $\fA'$ is strongly controlled;
\item if $\fA$ is pseudo-controlled, then $\fa\sim \fa'$ and $\fb\sim \fb'$.
\end{enumerate}
\end{lemma}

\begin{proof} Statement (1) follows from the definition of $\fA'$ and (2) is immediate from the exact sequence \eqref{e:0'}. \end{proof}

\begin{lemma} \label{l:pcsuffic} Let $\fA$ be a pseudo-controlled $\Gamma$-system. Then the functional equations \eqref{e:basharp} and \eqref{e:chisharp} hold for $\fA$ if and only if they hold for $\fA'$. \end{lemma}

\begin{proof} Obvious by Lemma \ref{l:st}(2). \end{proof}

\begin{lemma}\label{l:an} Suppose $\fA$ is strongly controlled. Then $\xi\cdot \fb=0$, for some $\xi\in\La$, if and only if $\xi^{\sharp}\cdot \fa=0$. \end{lemma}

\begin{proof} By Lemma \ref{l:eq}, we have $\fb_n=\fk_n(\fb)$.  Thus $\xi\cdot \fb=0$ implies $\xi\cdot \fb_n=0$, and consequently, by the duality, $\xi^\sharp\cdot \fa_n=0$ for all $n$, yielding $\xi^\sharp\cdot \fa=0$. \end{proof}

\end{subsection}


\subsection{Two maps}
For simplicity, in the following we shall use the notations $Q_n:=\QQ_p[\Gamma_n]$ and $\Lambda_n:=\ZZ_p[\Gamma_n]$. The projections $\pi_m^n\colon\Gamma_n\to\Gamma_m$ are canonically extended to ring morphisms $\colon\Lambda_n\to\Lambda_m$. Let
$$Q_\infty:=\varprojlim Q_n=\QQ_p[[\Gamma]]\,.$$
Thanks to the inclusions $\La_n\hookrightarrow Q_n$ we can see $\La$ as a subring of $Q_\infty$.

\subsubsection{The Fourier map} \label{ss:four} Let $\fA$ be a $\Gamma$-system as above. In this section, we construct a $\Lambda$-linear map
$$\Phi\colon\fa^\sharp\lr \Hom_{\Lambda}(\fb,Q_\infty/\La)\,.$$
First recall that the pairing in ($\Gamma$-4) induces for any $n$ an isomorphism of $\La$-modules \footnote{Here $\Hom_{\ZZ_p}(\fb_n,\QQ_p/\ZZ_p)$ is a $\La$-module via $(\gamma\cdot f)\colon x\mapsto f(\gamma x)$ for $\gamma\in\Gamma$.}
$$\fa_n^\sharp\simeq \Hom_{\ZZ_p}(\fb_n,\QQ_p/\ZZ_p),$$
the twist by the involution $\cdot^\sharp$ being due to \eqref{e:twistpairing}. Equality \eqref{e:rkpairing} shows that these isomorphisms form an isomorphism of projective systems, where the right hand side is endowed  with the transition maps induced by the direct system $(\fb_n,\fr_m^n)$.
Passing to the projective limit, we deduce a $\Lambda$-isomorphism
$$\fa^\sharp\simeq \lim_{\stackrel{\leftarrow}{n}}\Hom_{\ZZ_p}(\fb_n,\QQ_p/\ZZ_p).$$
Now the map $\Phi$ is obtained as the composed of this isomorphism and the following  $\Lambda$-linear maps:
\begin{enumerate}
\item[($\Phi$-1)] the homomorphism
$$\lim_{\stackrel{\leftarrow}{n}}\Hom_{\ZZ_p}(\fb_n,\QQ_p/\ZZ_p) \lr \lim_{\stackrel{\leftarrow}{n}} \Hom_\Lambda(\fb_n,Q_n/\La_n)$$
obtained by sending $(f_n)_n$ to $\big(\hat{f}_n\colon x\mapsto\sum_{\gamma\in\Gamma_n}f_n(\gamma^{-1}x)\gamma\big)_n\,$;
\item[($\Phi$-2)] the homomorphism
$$\varprojlim \Hom_\Lambda(\fb_n,Q_n/\La_n)\lr \varprojlim \Hom_{\Lambda}(\fb,Q_n/\La_n)$$
induced by $\fk_n\colon\fb\to \fb_n\,$;
\item[($\Phi$-3)] the canonical isomorphism
$$\lim_{\stackrel{\leftarrow}{n}} \Hom_{\Lambda}(\fb,Q_n/\La_n)\simeq  \Hom_{\Lambda}(\fb,\lim_{\stackrel{\leftarrow}{n}}Q_n/\La_n)$$
and the identification $\varprojlim Q_n/\La_n=Q_\infty/\La$ (since the maps $\La_n\rightarrow\La_m$ are surjective).
\end{enumerate}
Here as transition maps in $\varprojlim\Hom_\Lambda(\fb_n,Q_n/\La_n)$ we take (for $n\geq m$)
$$\Hom_\Lambda(\fb_n,Q_n/\La_n)\lr\Hom_\Lambda(\fb_m,Q_m/\La_m)$$
\begin{equation}\label{e:twistmap}\varphi \mapsto p^{-d(n-m)}(\pi_m^n\circ\varphi\circ\fr_m^n)\,.\end{equation}
We have to check that ($\Phi$-1) and ($\Phi$-2) define maps of projective systems. For ($\Phi$-1),
this means to verify that for any $n\geq m$ we have \begin{equation}\label{e:check1}\hat{f_m}=p^{-d(n-m)}(\pi_m^n\circ\hat{f_n}\circ\fr_m^n)\,, \end{equation}
where, by definition, $f_m=f_n\circ\fr_m^n$. For $x\in\fb_m$,
$$\pi_m^n(\hat{f_n}(\fr_m^nx))=\pi_m^n\big(\sum_{\gamma\in\Gamma_n}f_n(\gamma^{-1}(\fr_m^n x))\gamma\big)=\sum_{\gamma\in\Gamma_n}f_n(\gamma^{-1}\fr_m^n x)\pi_m^n(\gamma)$$
(using the fact that, by ($\Gamma$-2), $\fr^n_m$ is a $\Gamma$-morphism)
$$=\sum_{\gamma\in\Gamma_n}f_n(\fr_m^n(\gamma^{-1}x))\pi_m^n(\gamma)=\frac{|\Gamma_n|}{|\Gamma_m|}\sum_{\gamma\in\Gamma_m}f_n(\fr_m^n(\gamma^{-1}x))\gamma=p^{d(n-m)}\hat{f}_m(x)\,,$$
so \eqref{e:check1} holds. As for ($\Phi$-2), the transition map
$$\Hom_{\Lambda}(\fb,Q_n/\La_n)\lr \Hom_{\Lambda}(\fb,Q_m/\La_m)$$
is $\psi\mapsto\pi^n_m\circ\psi$ and the map defined in ($\Phi$-2) is $(\varphi_n)_n\mapsto (\varphi_n\circ\fk_n)_n\,$. By \eqref{e:twistmap},
$$\varphi_m\circ\fk_m=p^{-d(n-m)}(\pi_m^n\circ\varphi_n\circ\fr_m^n)\circ\fk_m=p^{-d(n-m)}(\pi_m^n\circ\varphi_n\circ\fr_m^n\circ\fk^n_m\circ\fk_n)=$$
(by property ($\Gamma$-3) of $\Gamma$-systems)
$$=p^{-d(n-m)}(\pi_m^n\circ\varphi_n\circ\Nm_{\Gamma_n/\Gamma_m}\circ \fk_n)=\pi_m^n\circ\varphi_n\circ\fk_n$$
(since $\varphi_n$, being a $\La$-morphism, commutes with $\Nm_{\Gamma_n/\Gamma_m}$ and $\pi_m^n\circ\Nm_{\Gamma_n/\Gamma_m}=p^{d(n-m)}\pi_m^n$).
So also ($\Phi$-2) is a map of projective systems.

\begin{remark}{\em Actually, one can also check that the maps $f_n\mapsto\hat{f_n}$ used in ($\Phi$-1) are isomorphisms. The inverse is $f\mapsto\delta_e\circ f$, where $\delta_e\colon Q_n/\La_n\to\QQ_p/\ZZ_p$ is the function sending $\sum_{\Gamma_n}a_\gamma\gamma$ to $a_e$ ($e$ being the neutral element in $\Gamma_n$).}
\end{remark}

If the $\Gamma$-system $\fA$ is strongly controlled then the map $\Phi$ is clearly injective (since $\fb$ surjects onto $\fb_n$ for all $n$). In general, we have the following.

\begin{lemma} \label{l:kernPhi}
The kernel of $\Phi$ equals $(\fa^0)^\sharp$.
\end{lemma}

\begin{proof} The image of $a=(a_n)_n\in\fa$ in $\varprojlim\Hom_\Lambda(\fb_n,Q_n/\La_n)$ is the map
$$b=(b_n)_n\mapsto\big(\sum_{\gamma\in\Gamma_n}\langle a_n,\gamma^{-1}b_n\rangle_n\gamma\big)_n\,.$$
To conclude, observe that $\fa_n\to\varinjlim\fa_m$ is dual to $\fb\to\fb_n$. Hence $\langle a_n,b_n\rangle_n=0$, for every $b_n$ contained in the image of $\fb\rightarrow \fb_n$, if and only if $a_n\in\fa_n^0$.
\end{proof}

\subsubsection{} Let $\fb$ be a finitely generated torsion $\La$-module. In \S\ref{ss:Psi} below we shall construct a map
$$\Psi\colon\Hom_{\Lambda}(\fb, Q_\infty/\La)\to \Hom_{\Lambda}(\fb,Q(\Lambda)/\Lambda),$$
where $Q(\La)$ is the field of fractions of $\La$. The interest of having such a $\Psi$ comes from the following lemma.

\begin{lemma} \label{l:b&Iwadj}
For $\fb$ a finitely generated torsion $\La$-module, we have a pseudo-isomorphism
$$\fb\sim \Hom_{\Lambda}(\fb,Q(\Lambda)/\Lambda).$$
\end{lemma}

\begin{proof}
From the exact sequence
$$0\longrightarrow [\fb]\longrightarrow \fb\longrightarrow \fn\lr 0,$$
where $\fn$ is pseudo-null, we deduce the exact sequence
$$\Hom_{\Lambda}(\fn,Q(\La)/\Lambda)\hookrightarrow \Hom_{\Lambda}(\fb,Q(\La)/\La) \to \Hom_{\Lambda}([\fb],Q(\La)/\La)\to\text{Ext}_{\Lambda}^1(\fn,Q(\La)/\La)\,.$$
The annihilator of $\fn$ also kills $\Hom_{\Lambda}(\fn,\;)$ and its derived functors, so by Lemma \ref{l:psn}, it follows $\Hom_{\La}(\fb,Q(\Lambda)/\Lambda)\sim \Hom_{\La}([\fb],Q(\Lambda)/\Lambda)$, and we can assume that $\fb=[\fb]$. Write
$$\fb=\Lambda/(\xi_1)\oplus\dots\oplus\Lambda/(\xi_n)\,.$$
Then
$$ \Hom_{\Lambda}(\fb,Q(\Lambda)/\Lambda)=\bigoplus\Hom_{\Lambda}(\La/(\xi_i),Q(\Lambda)/\Lambda)=\bigoplus\Hom_{\Lambda}(\Lambda/(\xi_i),\xi_i^{-1}\Lambda/\Lambda)$$
because $(Q(\La)/\La)[\xi_i]=\xi_i^{-1}\La/\La.$
Since
$$\Hom_{\Lambda}(\Lambda/(\xi_i),\xi_i^{-1}\Lambda/\Lambda)=\Lambda/(\xi_i)\,,$$
we conclude that in this situation, $\Hom_{\Lambda}(\fb,Q(\Lambda)/\Lambda)=\fb\,.$
\end{proof}

\subsubsection{A theorem of Monsky} \label{su:monsky}
Let $\Gamma^\vee$ (resp. $\Gamma^\vee_n$ ) denote the group of continuous characters $\Gamma\rightarrow \boldsymbol\mu_{p^\infty}$ (resp. $\Gamma_n\rightarrow \boldsymbol\mu_{p^{\infty}}$); we view $\Gamma^\vee_n$ as a subgroup of $\Gamma^\vee$. For each $\omega\in \Gamma^\vee$, let $E_\omega:=\QQ_p(\boldsymbol\mu_{p^m})\subset \bar\QQ_p$ be the subfield generated by the image $\omega(\Gamma) =\boldsymbol\mu_{p^m}$, and write $\cO_\omega=\ZZ_p[\boldsymbol\mu_{p^m}]$. Then $ \omega$ induces a continuous ring homomorphism $\omega\colon\Lambda\rightarrow \cO_\omega\subset E_\omega$.\\

Let $\xi\in\La$: we say that $\omega$ is a zero of $\xi$, if and only if $\omega(\xi)=0$, and denote the zero set
$$\triangle_{\xi}:=\{ \omega\in {\Gamma^\vee}\;\mid\; \omega(\xi)=0\}.$$
Then we recall a theorem of Monsky (\cite[Lemma 1.5 and Theorem 2.6]{monsky}).

\begin{definition} \label{d:monsky}
A subset $\Xi\subset \Gamma^\vee$ is called a $\ZZ_p$-flat of codimension $k$, if there exists $\{\gamma_1,...,\gamma_k\}\subset \Gamma$ expandable to a $\ZZ_p$-basis of $\Gamma$ and $\zeta_1,...,\zeta_k\in\boldsymbol\mu_{p^{\infty}}$ so that
$$\Xi=\{ \omega \in \Gamma^\vee\;\mid\; { \omega}(\gamma_i)=\zeta_i,i=1,...,k\}.$$
\end{definition}

\noindent This definition is due to Monsky: in \cite[\S1]{monsky}, he proves that $\ZZ_p$-flats generate the closed sets of a certain (Noetherian) topology on $\Gamma^\vee$. It turns out that in this topology the sets $\triangle_\xi$ are closed, and they are proper subsets (possibly empty) if $\xi\neq0$ (\cite[Theorem 2.6]{monsky}). Hence

\begin{mytheorem}[Monsky] \label{t:monsky}
Suppose $\cO$ is a discrete valuation ring finite over $\ZZ_p$ and $\xi\in\cO[[\Gamma]]$ is non-zero. Then the zero set $\triangle_{\xi}$ is a proper subset of $\Gamma^\vee$ and is a finite union of $\ZZ_p$-flats.
\end{mytheorem}

\subsubsection{Structure of $Q_\infty$} \label{ss:qorbit}
The group $\Gal(\bar\QQ_p/\QQ_p)$ acts on $\Gamma^\vee$ by $(\sigma\cdot\omega)(\gamma):=\sigma(\omega(\gamma))$. Let $[\omega]$ denote the $\Gal(\bar\QQ_p/\QQ_p)$-orbit of $ \omega$. Attached to any character $ \omega\in\Gamma^\vee_n$ there is an idempotent
\begin{equation} \label{e:idempotent} e_\omega:=\frac{1}{|\Gamma_n|}\sum_{\gamma\in\Gamma_n} \omega(\gamma^{-1})\gamma\in\bar\QQ_p[\Gamma_n]\,. \end{equation}
Accordingly, we get the decomposition
$$\QQ_p[\Gamma_n]=\bar\QQ_p[\Gamma_n]^{\Gal(\bar\QQ_p/\QQ_p)}=\big(\prod_{\omega\in\Gamma^\vee_n}e_\omega\bar\QQ_p[\Gamma_n]\big)^{\Gal(\bar\QQ_p/\QQ_p)}=\prod_{[ \omega]\subset \Gamma^\vee_n}  E_{[\omega]},$$
where $[\omega]$ runs through all the ${\Gal(\bar\QQ_p/\QQ_p)}$-orbits of $\Gamma^\vee_n$ and
$$E_{[\omega]}:=(\prod_{\chi\in [\omega]}e_\chi\bar\QQ_p[\Gamma_n])^{\Gal(\bar\QQ_p/\QQ_p)}.$$
Observe that the homomorphism $\omega\colon\QQ_p[\Gamma_n]\rightarrow\bar\QQ_p$ induces an isomorphism $E_{[\omega]}\simeq E_\omega$ (the inverse being given by $1\mapsto\sum_{\sigma\in Gal(E_\omega/\QQ_p)}\sigma(e_\omega)=(e_\chi)_{\chi\in [\omega]}$).

Since $\pi^n_m(e_\omega)$ equals $e_{\omega'}$ if $\omega=\omega'\circ\pi^n_m$ and is $0$ otherwise, we have the commutative diagram
\[\begin{CD}  \QQ_p[\Gamma_n] @>>> \prod_{[ \omega]\subset \Gamma^\vee_n}  E_{[\omega]}\\
   @VV{\pi_m^n}V @VVV \\
\QQ_p[\Gamma_m] @>>> \prod_{[ \omega]\subset \Gamma^\vee_m}  E_{[\omega]}  \end{CD}\]
where the right vertical arrow is the natural projection by the inclusion $\Gamma^\vee_n\hookrightarrow\Gamma^\vee_m$.
It follows that we have identities
\begin{equation} \label{e:decQinfty} Q_\infty=\lim_{\stackrel{\leftarrow}{n}}Q_n\simeq\prod_{[\omega]\subset\Gamma^\vee}E_{[\omega]} \end{equation}
so that
\begin{equation} \label{e:torsQinfty} Q_\infty[\lambda]=\prod_{[\omega]\subset \triangle_\lambda} E_{[\omega]}
\end{equation}
for all $\lambda\in\La$ (here $Q_\infty[\lambda]$ denotes the $\lambda$-torsion subgroup).

\subsubsection{The map $\Psi$} \label{ss:Psi}
Let $\fb$ be a finitely generated torsion $\La$-module: so is $ \Hom_{\Lambda}(\fb,Q_\infty/\La)$.
We assume that $\xi\cdot \fb=0$, for some non-zero $\xi\in\Lambda$. Let $\triangle_\xi^c:=\Gamma^\vee-\triangle_\xi$ denote the complement of $\triangle_\xi$. From \eqref{e:decQinfty} and \eqref{e:torsQinfty} one deduces the direct sum decomposition
$$Q_\infty=Q_\infty[\xi]\oplus Q_\infty^c\,,$$
where $Q_\infty^c=\prod_{[\omega]\subset \triangle_\xi^c} E_{[\omega]}$.
Let $\varpi\colon Q_\infty\rightarrow Q_\infty^c$ be the natural projection and put $\La^c:=\varpi(\La)$ (here $\La$ is thought of as a subset of $Q_\infty$ via the maps $\ZZ_p[\Gamma_n]\hookrightarrow\QQ_p[\Gamma_n]$).

\begin{lemma} \label{l:xitorsionHom} We have a $\Lambda$-isomorphism
$$\Hom_{\La}(\fb,Q_\infty^c/\La^c) \simeq \Hom_{\Lambda}(\fb,Q(\Lambda)/\Lambda).$$
\end{lemma}

\begin{proof}
Since $\fb$ is annihilated by $\xi$, the image of each $\eta\in\Hom_{\La}(\fb,Q_\infty^c/\La^c)$ is contained in $(Q_\infty^c/\La^c)[\xi]\,.$
Note that, since $\omega(\xi)\not=0$ for every $\omega\in\triangle_\xi^c$, the element $\varpi(\xi)$ is a unit in $Q_\infty^c$. Denote
$$\xi^{-1}\La^c:=\{x\in Q_\infty^c\;\mid\; \xi\cdot x\in\La^c\}.$$
Then
$$(Q_\infty^c/\La^c)[\xi]=\xi^{-1}\La^c/\La^c$$
and hence
$$\Hom_{\La}(\fb,Q_\infty^c/\La^c)=\Hom_{\La}(\fb,\xi^{-1}\La^c/\La^c).$$
Similarly,
$$\Hom_{\La}(\fb,Q(\La)/\La)=\Hom_{\La}(\fb,\xi^{-1}\La/\La).$$
To conclude the proof, it suffices to show that $\varpi\colon\La\rightarrow \La^c$ is an isomorphism, because then so is the induced map
$$\xi^{-1}\Lambda/\Lambda\,\lr \xi^{-1}\La^c/\La^c.$$
Since $\La^c=\varpi(\La)$ by definition, we just need to check injectivity. Suppose $\varpi(\epsilon)=0$ for some $\epsilon\in\Lambda$. Then $\omega(\epsilon)=0$ for every $\omega\not\in\triangle_{\xi}$, and hence $\omega(\xi\epsilon)=0$ for every $ \omega\in\Gamma^\vee$. Monsky's theorem (or, alternatively, the isomorphism \eqref{e:decQinfty}) implies that $\xi\epsilon=0$ and hence $\epsilon=0$.
\end{proof}

Let
\begin{equation} \label{e:upsilondef} \Upsilon\colon\Hom_{\La}(\fb,Q_\infty/\La)\,\lr\Hom_{\La}(\fb,Q_\infty^c/\La^c)\end{equation}
be the morphism induced from $\varpi$. By composition of the isomorphism of Lemma \ref{l:xitorsionHom} with $\Upsilon$, we deduce the $\Lambda$-morphism
$$\Psi\colon\Hom_{\La}(\fb,Q_\infty/\La)\,\lr\Hom_{\Lambda}(\fb,Q(\Lambda)/\Lambda).$$


\begin{subsection}{Proof of the algebraic functional equation} \label{se:al}

\subsubsection{Non-simple annihilator} \label{ss:simple}
We start the proof of case (1) of Theorem \ref{t:al} by reformulating our hypothesis
\vskip5pt
\noindent
({\bf NS}): {\em{$\xi$ is not divisible by any simple element}}.

\begin{lemma} \label{l:hypH}
Hypothesis {\em (}{\bf NS}{\em)} holds if and only if $\triangle_{\xi}$ contains no codimension one $\ZZ_p$-flat.
\end{lemma}

\begin{proof}
If $\xi$ is divisible by a simple element $f=f_{\gamma,\zeta}$, then
$\triangle_{\xi}$ contains $\triangle_{f}$ which is a union of the codimension one $\ZZ_p$-flats
$$\{ \omega \in \Gamma^\vee\;\mid\; \omega(\gamma)=\sigma(\zeta)\},\;\; \sigma\in\Gal(\QQ_p(\zeta)/\QQ_p).$$
Conversely, assume that $\triangle_{\xi}$ contains the codimension one $\ZZ_p$-flat
$$\Xi=\{ \omega \in \Gamma^\vee\;\mid\; { \omega}(\gamma)=\zeta\}.$$
Each $\omega\in\Xi$ factors through
$$\pi\colon\Lambda \longrightarrow \ZZ_p[\zeta][[\Gamma]]/(\gamma-\zeta)=\ZZ_p[\zeta][[\Gamma']]\,,$$
where $\Gamma'$ is the quotient $\Gamma/\gamma^{\ZZ_p}$, and vice versa every continuous character of $\Gamma'$ can be uniquely lifted to a character in $\Xi$. Thus the zero set of $\pi(\xi)\in\ZZ_p[\zeta][[\Gamma']]$ equals $(\Gamma')^\vee$. Then Monsky's theorem implies that $\pi(\xi)=0$ and hence is divisible by $\gamma-\zeta$ in $\ZZ_p[\zeta][[\Gamma]]$.
This implies that $\xi$ is divisible by $f_{\gamma,\zeta}$ in $\Lambda$.
\end{proof}

By Monsky's theorem, we have either $\triangle_{\xi}=\emptyset$ or $\triangle_{\xi}=\cup_j\,\Xi_j\,$, with
$$\Xi_j=\{ \omega \in \Gamma^\vee\;\mid\; { \omega}(\gamma_i^{(j)})=\zeta_i^{(j)},i=1,...,k^{(j)}\}. $$
In the second case, for all $j$ let $G_j$ be the $\ZZ_p$-submodule of $\Gamma$ generated by the $\gamma_i^{(j)}$'s, $i=1,...,k^{(j)}$: if ({\bf NS}) holds, each $G_j$ has rank at least 2. Hence, since there is just a finite number of $j$, it is possible to choose $\{\sigma_1^{(j)},\sigma_2^{(j)}\}_j$ so that $\sigma_i^{(j)}\in G_j-\Gamma^p$ and each pair $(\sigma_i^{(j)},\sigma_{i'}^{(j')})$ consists of $\ZZ_p$-independent elements unless $(i,j)=(i',j')$. Let $\varepsilon_i^{(j)}$ denote the common value that all characters in $\Xi_j$ take on $\sigma_i^{(j)}$ and write $\varphi_i:=\prod_j f_{\sigma_i^{(j)},\varepsilon_i^{(j)}}$. Then the coprimality criterion \eqref{e:coprimecriter} ensures that $\varphi_1$ and $\varphi_2$ are relatively prime. Moreover $\omega(\varphi_i)=0$ for all $\omega\in\triangle_{\xi}$, that is $\triangle_{\xi}\subseteq\triangle_{\varphi_i}$. By \eqref{e:torsQinfty} it follows
\begin{equation} \label{e:triangles} \varphi_i\cdot Q_\infty[\xi] =0  \text{ for both } i. \end{equation}

\begin{remark}{\em The case $\Delta_\xi\neq\emptyset$ can actually occur. For example, let $\gamma_1,\gamma_2$ be two distinct elements of a $\ZZ_p$-basis of $\Gamma$ and consider
$$\xi=\gamma_1-1+p(\gamma_2-1)+p^2(\gamma_1-1)(\gamma_2-1).$$
Then $\Delta_\xi=\{ \omega \mid \omega(\gamma_1)=\omega(\gamma_2)=1\}$ as one easily sees comparing $p$-adic valuations of the three summands $\omega(\gamma_1-1)$, $\omega(p(\gamma_2-1))$ and $\omega(p^2(\gamma_1-1)(\gamma_2-1))$. }
\end{remark}

\begin{lemma} \label{l:Upsilonpsn}
Assume $\xi\fb=0$ for some $\xi\in\La$ satisfying hypothesis {\em (}{\bf NS}{\em)}. Then the restriction of $\Psi$ to $\Phi(\fa^\sharp)$ is pseudo-injective.
\end{lemma}

\begin{proof} We just need to control the kernel of the map $\Upsilon$ of \eqref{e:upsilondef}. If $\Delta_\xi$ is empty then $\Upsilon$ is the identity and we are done. If not, we show that kernel and cokernel of $\Upsilon$ are annihilated by both $\varphi_1$ and $\varphi_2$. Consider the exact sequence
$$0\lr \Ker(\varpi) \lr Q_\infty/\La \lr Q_\infty^c/\La^c \lr 0$$
(where by abuse of notation we denote the map induced by $\varpi$ with the same symbol). This induces the exact sequence
$$\Hom_{\Lambda}(\fb,\Ker(\varpi)) \hookrightarrow \Hom_{\Lambda}(\fb,Q_\infty/\La) \to \Hom_{\Lambda}(\fb,Q_\infty^c/\La^c) \to\text{Ext}^1_{\Lambda}(\fb,\Ker(\varpi)).$$
Since $\Ker(\varpi)$ is a quotient of $Q_\infty[\xi]$, \eqref{e:triangles} yields $\varphi_i\cdot\Ker(\varpi)=0$. Therefore both $\Ker(\Upsilon)$ and $\Coker(\Upsilon)$ are annihilated by $\varphi_1$ and $\varphi_2$. (Note that we cannot say that $\Upsilon$ is a pseudo-isomorphism, because $\Hom_{\La}(\fb,Q_\infty/\La)$ is not a finitely generated $\La$-module: e.g., any group homomorphism $\fb\mapsto E_\omega$ for $\omega\in\Delta_\xi$ is also a $\La$-homomorphism.)
\end{proof}

Now we can complete the proof of Theorem \ref{t:al}.(1).

\begin{proof}[{\bf Proof of Theorem \ref{t:al}(1)}] To start with, assume $\xi\fb=0$. Then, by Lemmata \ref{l:kernPhi}, \ref{l:Upsilonpsn} and \ref{l:b&Iwadj}, we get a pseudo-injection $\fa^\sharp\rightarrow \fb$. Moreover, thanks to Lemma \ref{l:pcsuffic}, we may assume that $\fA$ is strongly controlled. By Lemma \ref{l:an} this implies that $\fa$ is killed by $\xi^\sharp$, which is also not divisible by simple elements. Exchanging the role of $\fa$ and $\fb$, we deduce a pseudo-injection  $\fb^\sharp\rightarrow \fa$ and therefore a pseudo-injection $\fb\rightarrow \fa^\sharp$. The Theorem now follows from Lemma \ref{l:pseudoinj}.

In the general case when $\xi\fb$ is pseudo-null but not 0, we can still assume that $\fA$ is strongly controlled. Let $\ff_n$ be the kernel of the morphism $\fb_n\rightarrow\fb_n$, $b\mapsto\xi b$, and construct two derived systems as in \S\ref{su:de} (but with $\ff_n$ playing the role of $\fc_n$). We get again the two exact sequences \eqref{e:c} and \eqref{e:e}. By hypothesis $\fd=\xi\fb\sim0$ and then Lemma \ref{l:e0f0} implies $\fc\sim0$. Hence
$$\fb\sim\ff\sim\fe^\sharp\sim\fa^\sharp$$
(where the central pseudo-isomorphism holds because $\xi\ff=0$).
\end{proof}

\subsubsection{The non-simple part}\label{ss:nons} For any finitely generated torsion $\La$-module $M$, we get a decomposition in simple and non-simple part
$$[M]=[M]_{si}\oplus [M]_{ns}.$$
in the following way: recalling that $[M]$ is a direct sum of components $\La/\xi_i^{r_i}\La$, we define $[M]_{si}$ as the sum over those $\xi_i$ which are simple and $[M]_{ns}$ as its complement.

\begin{corollary}\label{c:nons}
For any $\Gamma$-system $\fA$, we have
$$[\fa']_{ns}^\sharp= [\fb']_{ns}.$$
\end{corollary}

\begin{proof}
By Lemma \ref{l:st}(1) we can lighten notation and assume that $\fA$ is strongly controlled (replacing $\fA$ by $\fA'$ if necessary).
Write $\chi(\fb)=(\lambda\mu)$, with $\chi([\fb]_{ns})=(\lambda)$ and $\chi([\fb]_{si})=(\mu)$. Since $\fb/[\fb]$ is pseudo-null, there are
$\eta_1,\eta_2\in \La$, coprime to each other and both coprime to $\chi(\fb)$, so that $\eta_1\cdot (\fb/ [\fb])=\eta_2\cdot (\fb/ [\fb])=0$ .
Then $\lambda\mu\eta_1\cdot \fb=\lambda\mu\eta_2\cdot \fb=0$. By Lemma \ref{l:an}
\begin{equation}\label{e:918}
(\lambda\mu\eta_1)^\sharp\cdot \fa=(\lambda\mu\eta_2)^\sharp\cdot \fa=0.
\end{equation}
This shows that $\chi(\fa)$ divides sufficiently high powers of both $(\lambda\mu\eta_1)^\sharp$ and $(\lambda\mu\eta_2)^\sharp$.
But since $\eta_1^{\sharp}$ and $\eta_2^\sharp$ are coprime, they must be both coprime to $\chi(\fa)$.

Set $\fc=(\mu\eta_1)^\sharp\cdot \fa$, $\fc_n=\fk_n(\fc)$ for each $n$, and form the $\Gamma$-systems $\fC$, $\fE$ by the construction in \S\ref{su:de}.
Let $\fd$, $\fe$, and $\ff$ be as in \eqref{e:c} and \eqref{e:e}.
Since $\fk_n(\fb)=\fb_n$, we have $\fk_n(\fd)=\fd_n$, and hence, by Lemma \ref{l:eq}, $\fC$ is also strongly controlled.
By \eqref{e:918}, $\lambda^\sharp\cdot \fc=0$, whence $\lambda\cdot\fd=0$ thanks to Lemma \ref{l:an}. Then case (1) of Theorem \ref{t:al} says $\fc^\sharp\sim \fd$.
To complete the proof it is sufficient to show that
$$\xymatrix{[\fb]_{ns}\times [\fa]_{ns}\ar[r]^-{\varphi\times\psi} & \fd\times \fc}$$
(where $\varphi$ and $\psi$ are respectively the restrictions to $[\fb]_{ns}$ and $[\fa]_{ns}$ of the projection $\fb\rightarrow \fd=\fb/\ff$ and of the multiplication by $(\mu\eta_1)^\sharp$ on $\fa$) is a pseudo-isomorphism.

The inclusion $\mu\eta_1\cdot \fb\subset \mu\cdot [\fb]\subset[\fb]_{ns}$ implies $\mu\eta_1\cdot \Coker ( \varphi)=0$. Furthermore, as $\lambda\cdot \Coker (\varphi)$ is a quotient of $\lambda\cdot\fd=0$, it must be trivial. Thus, by Lemma \ref{l:psn}, $\Coker (\varphi)$, being annihilated by coprime $\lambda$ and $\mu\eta_1$, is pseudo-null.
Next, we observe that $(\mu\eta_1)^\sharp\cdot \fe=(\mu\eta_1)^\sharp\cdot \fa/\fc=0$ yields $(\mu\eta_1)^\sharp\cdot \fe_n=0$. The duality implies that each $\ff_n$ is annihilated by $\mu\eta_1$, and by taking the projective limit we see that $\ff$ is also annihilated by $\mu\eta_1$.
It follows that $\Ker (\varphi)=[\fb]_{ns}\cap \ff=0$ as no nontrivial element of $[\fb]_{ns}$ is annihilated by $\mu\eta_1$ (because $\eta_1$ is coprime to $\chi(\fb)$ while $\mu$ is a product of simple elements). Similarly, $\Ker(\psi)=0$ as no nontrivial element of $[\fa]_{ns}$ is annihilated by $(\mu\eta_1)^\sharp$.
To show that $\Coker (\psi)$ is pseudo-null, we choose an $\eta_3\in\La$, coprime to $\lambda\eta_1$, so that $\eta_3^\sharp\cdot \fa\subset [\fa]$. Then \eqref{e:918} together with the fact that $\lambda$ is non-simple imply that $(\mu\eta_1\eta_3)^\sharp \cdot \fa\subset (\mu\eta_1)^\sharp \cdot [\fa]\subset[\fa]_{ns}$. This implies $(\mu\eta_1\eta_3)^\sharp \cdot \Coker (\psi)=0$.
Since $\lambda^\sharp\cdot\Coker (\psi)$, being a quotient of $\lambda^\sharp\cdot \fc=0$, is trivial and $\lambda^\sharp$, $(\mu\eta_1\eta_3)^\sharp$ are coprime, the proof is completed.
\end{proof}

\subsubsection{Twists of $\Gamma$-systems} \label{ss:twistgam}
Recall that associated to a continuous group homomorphism $\phi\colon\Gamma\rightarrow\ZZ_p^\times$, there is the ring isomorphism $\phi^*\colon\La\rightarrow\La$ defined in \S\ref{ss:iwH1}. Given such a $\phi$ and a $\Gamma$-system $\fA$, we can form
$$\fA(\phi):=\{\fa_n(\phi^{-1}),\fb_n(\phi),\langle\;,\;\rangle^\phi_n,\fr(\phi)_m^n,\fk(\phi)_m^n\;\mid\; n,m\in\NN\cup\{0\},\; n\geq m\}\,,$$
where $\fa_n(\phi^{-1})$ and $\fb_n(\phi)$ are twists as defined in \eqref{e:twistbyphi},
$$\langle x\otimes a_n, y\otimes b_n\rangle^\phi_n:=\langle \phi^*(x) a_n, ({\phi^{-1}})^*(y) b_n\rangle_n$$
and $\fr(\phi)_m^n$, $\fk(\phi)_m^n$ are respectively the maps induced by $1\otimes\fr_m^n$ and, $1\otimes\fk_m^n$. In general $\fA(\phi)$ won't be a $\Gamma$-system, because the action of $\Gamma$ on $\fa_n(\phi^{-1})$, $\fb_n(\phi)$ does not factor through $\Gamma_n$. However if we take $\fA$ twistable of order $k$ and $\phi$ so that
\begin{equation}\label{e:phicondition}
\phi(\Gamma)\subseteq 1+p^k\ZZ_p,
\end{equation}
then both $\fa_n(\phi^{-1})$ and $\fb_n(\phi)$ are still $\Gamma_n$-modules, as $\phi(\Gamma^{(n)})\subset 1+p^{n+k}\ZZ_p$ by \eqref{e:phicondition} and $p^{n+k}\fa_n=0$.

\begin{lemma}\label{l:nonsimpletwist}
For any $k\in\NN$ and $\xi\in\La-\{0\}$, there exists a continuous group homomorphism $\phi\colon\Gamma\rightarrow\ZZ_p^\times$ so that \eqref{e:phicondition} holds and both $\phi^*(\xi)$ and $({\phi^{-1}})^*(\xi)$ are not divisible by simple elements.
\end{lemma}

\begin{proof} First of all, note that $({\phi^{-1}})^*(\xi)$ is not divisible by any simple element if and only if the same holds for $({\phi^{-1}})^*(\xi)^\sharp=\phi^*(\xi^\sharp)$. So we just need to find $\phi$ such that $\phi^*(\xi\xi^\sharp)$ has no simple factor.
An abstract proof of the existence of such $\phi$ can be obtained by the Baire category theorem, observing that if $\lambda\in\La-\{0\}$ then $\Hom(\Gamma,\ZZ_p^\times)$ cannot be contained in $\cup_\omega\Ker\big(\phi\mapsto\omega(\phi^*(\lambda))\big)$, since all these kernels have empty interior. A more concrete approach is the following.

Call an element $\lambda\in\La$ a simploid if it has the form $\lambda=u\cdot f_{\gamma,\beta}$ where $u\in\Lambda^\times$ and
$$f_{\gamma,\beta}:=\prod_{\sigma\in \Gal(\QQ_p(\beta)/\QQ_p)} (\gamma-\sigma(\beta))$$
with $\gamma\in\Gamma-\Gamma^p$ and $\beta$ a unit in some finite Galois extension of $\QQ_p$. Simploids are easily seen to be irreducible, so by unique factorization any principal ideal $(\lambda)\subset\Lambda$ can be written as $(\lambda)=(\lambda)_s(\lambda)_n$ with no simploid dividing $(\lambda)_n$. Moreover, given any $\phi\colon\Gamma\rightarrow\ZZ_p^\times$, the equality
$$\phi^*(f_{\gamma,\beta})=\phi(\gamma)^{-[\QQ_p(\beta):\QQ_p]}\cdot f_{\gamma,\phi(\gamma)\beta}$$
shows that the set of simploids is stable under the action of $\phi$, so that $(\phi^*(\lambda))_s=\phi^*((\lambda)_s)$.
Thus, if $f_{\gamma_1,\beta_1},...,f_{\gamma_l,\beta_l}$ is a maximal set of coprime simploid factors of $\xi\xi^\sharp$ and if $\phi$ is chosen so that no $\phi(\gamma_i)\beta_i$, $i=1,...,l$, is a root of unit, then $\phi^*(\xi\xi^\sharp)$ is not divisible by any simple element.
\end{proof}

\begin{proof}[{\bf Proof of Theorem \ref{t:al}(2)}]
Let $\xi$ be a generator of $\chi(\fa)\chi(\fb)$ and let $\phi$ be as in Lemma \ref{l:nonsimpletwist}.
Then $\fA(\phi)$ also form a pseudo-controlled $\Gamma$-system with $\fa(\phi^{-1})=\varprojlim_n\fa_n(\phi^{-1})$ and $\fb(\phi)=\varprojlim_n\fb_n(\phi)$. By Lemma \ref{l:phi[]chi}, both $\chi(\fa(\phi^{-1}))$ and $\chi(\fb(\phi))$ are not divisible by simple elements, and hence $[\fa(\phi^{-1})]^\sharp=[\fa(\phi^{-1})]_{ns}^\sharp$ and $[\fb(\phi)]=[\fb(\phi)]_{ns}$. Therefore,
$$[\fa]^{\sharp}=[\fa(\phi^{-1})](\phi)^{\sharp}=[\fa(\phi^{-1})]^\sharp(\phi^{-1})=[\fb(\phi)](\phi^{-1})=[\fb],$$
where the first and the last equality are consequence of Lemma \ref{l:phi[]chi} and the third follows from Theorem \ref{t:al}(1) applied to $\fA(\phi)$.
\end{proof}

\end{subsection}

\subsubsection{Complete $\Gamma$-systems}
Now we assume that our original $\fA$ is just a part of a complete $\Gamma$-system which we still denote by $\fA$. The original $\fA$ is pseudo-controlled if and only if so is its complete system. Also, if the original $\fA$ is strongly controlled, then by replacing $\fa_F\times \fb_F$ by $\mathfrak{k}_F(\fa\times \fb)$ we can make the complete system strongly controlled without altering $\fa$ and $\fb$. So we shall assume that $\fA$ is strongly controlled.\\

First we assume that $\fa$ is annihilated by a simple element $\xi=f_{\gamma_1,\zeta}$ and extend $\gamma_1$ to a basis $\gamma_1,...,\gamma_d$ of $\Gamma$ over $\ZZ_p$.
Let $\Psi$ and $\Gamma'$ be the subgroups of $\Gamma$ with topological generators respectively $\gamma_1$ and $\{\gamma_2,...,\gamma_d\}$. Note that for $H\subset\Gamma$ a closed subgroup we shall write $H^{(n)}$ for $H^{p^n}$. Let $K_{n',n}$ denote the fixed field of the subgroup $\Psi^{(n)}\oplus(\Gamma')^{(n')}$ and write $\fa_{\infty,n}:=\varprojlim_{n'} \fa_{n',n}$, $\fb_{\infty,n}:=\varprojlim_{n'} \fb_{n',n}$ with the obvious meaning of indexes. They are $\Lambda$-modules. Let $K_{\infty,n}$ denote the subfield of $L$ fixed by  $\Psi^{(n)}$. Then the restriction of Galois action gives rise to a natural isomorphism $\Gamma'\simeq\Gal(K_{\infty,n}/K_{0,n})$. Write $\La':=\La(\Gamma')$. We shall view $\La'$ as a subring of $\La$.

Since $\fA$ is strongly controlled, $\fa_{\infty,n}=\fk_{\infty,n}(\fa)$ and $\fb_{\infty,n}=\fk_{\infty,n}(\fb)$ are finitely generated over $\La$, and hence finitely generated over $\La'$, as they are fixed by $\Psi^{(n)}$.

\begin{proposition}\label{p:twistable} Suppose $\fA$ is a strongly controlled complete $\Gamma$-system such that \begin{enumerate}
\item $\fa$ and $\fb$ are annihilated by the simple element $\xi=f_{\gamma_1,\zeta}$ defined above, with $\zeta$ of order $p^l$;
\item $\fa_{\infty,m}$ and  $\fb_{\infty,m}$ are torsion over $\La'$ for some $m\geq l$. \end{enumerate}
Then there exists some non-trivial $\eta\in\Lambda'$ so that $\eta\cdot \fA$ is twistable.
\end{proposition}

Here $\eta\cdot\fA$ is the complete $\Gamma$-system as defined in Example \ref{eg:morf}. It is also strongly controlled if so is $\fA$.

\begin{proof} Since $\zeta$ is of order $p^l$, the action of $\gamma_1^{p^l}$ is trivial on both $\fa_{\infty,n}$ and $\fb_{\infty,n}$  for all $n$. Assume that $m\geq l$ and suppose both $\fa_{\infty,m}$ and $\fb_{\infty,m}$ are annihilated by some non-zero $\eta\in\La'$. Then $\eta\cdot \fa_{n',m}=0$ and $\eta\cdot \fb_{n',m}=0$ for all $n'$. Hence for $n\geq m$,
$$p^{n-m}\eta\fa_{n',n}=\fr^{n',n}_{n',m}(\fk^{n',n}_{n',m}(\eta\fa_{n',n}))=0$$
as $\gamma_1^{p^n}$ acts trivially on $\fa_{n',n}$. In particular, $p^{n-m}\eta\cdot\fa_n=0$ and by similar argument
$p^{n-m}\eta\cdot\fb_n=0$. Then choose $k$ so that $p^k\fa_i=p^k\fb_i=0$ for each $1\leq i< m$.
\end{proof}

\begin{corollary}\label{c:twistable}
Suppose $\fA$ satisfies the condition of {\em Proposition \ref{p:twistable}}. Then
$$\fa^\sharp\sim \fb.$$
\end{corollary}

\begin{proof} The morphism $\fA\rightarrow\eta\cdot\fA$ of Example \ref{eg:morf} in this case is a pseudo-isomorphism, because $\fa[\eta]$ and $\fb/\eta^\sharp\fb$ are both killed by $f_{\gamma_1,\zeta}$ and either $\eta$ or $\eta^\sharp$. Now apply Theorem \ref{t:al}(2). \end{proof}

\begin{proof}[{\bf Proof of Theorem \ref{t:al}(3)}] We may assume that $\fA$ is strongly controlled. Suppose $\fa$ is annihilated by $\xi\in\Lambda$, and hence $\fb$ is annihilated by $\xi^\sharp$. Write $\xi=\xi_1^{s_1}\cdot \cdots \cdot \xi_k^{s_k}$, where each $\xi_i$ is irreducible and $s_i$ is a positive integer. The proof is by induction on $k$.

First assume $k=1$. If $\xi$ is non-simple, then the theorem has been proved. Thus, we may assume that $\xi_1$ is simple and we proceed by induction on $s_1$. The case $s_1=1$ is Corollary \ref{c:twistable}.
If $s_1>1$ let $\fc_F:=\xi_1\cdot \fa_F$ and form the derived systems $\fC$ and $\fE$ as in \S\ref{su:de}. Note that both enjoy property ({\bf T}), as immediate from the sequences  \eqref{e:c} and \eqref{e:e}. Besides $\fC$ is strongly controlled and $\fc$ is annihilated by $\xi_1^{s_1-1}$, whence (as $\xi_1$ is simple) $[\fc]=[\fc]^\sharp=[\fd]$
by the induction hypothesis. We still have $\ff^0=0$, but we don't know if $\fe^0=0$. However, induction tells us that $[\fe/\fe^0]=[\ff]$, or equivalently, there is an injection $[\ff]\hookrightarrow [\fe]$. This actually implies an inclusion $[\fb]\hookrightarrow[\fa]$: to see it, write
$$[\fa]=(\La/\xi_1\La)^{a_1}\oplus (\La/\xi_1^2\La)^{a_2}\oplus \cdots \oplus (\La/\xi_1^{s_1}\La)^{a_{s_1}},$$
and
$$[\fb]=(\La/\xi_1\La)^{b_1}\oplus (\La/\xi_1^2\La)^{b_2}\oplus \cdots \oplus (\Lambda/\xi_1^{s_1}\La)^{b_{s_1}}.$$
Then
$$[\fc]= (\Lambda/\xi_1\La)^{a_2}\oplus \cdots \oplus (\Lambda/\xi_1^{s_1-1}\La)^{a_{s_1}},$$
and
$$[\fd]=(\Lambda/\xi_1\La)^{b_2}\oplus \cdots \oplus (\Lambda/\xi_1^{s_1-1}\La)^{b_{s_1}},$$
while
$$[\fe]= (\Lambda/\xi_1\La)^{a_1+a_2+\cdots +a_{s_1}},\; \;[\ff]= (\Lambda/\xi_1\La)^{b_1+b_2+\cdots +b_{s_1}}. $$
Thus, we have $a_1\geq b_1$ and $a_i=b_i$ for $1<i\leq s_1$. Then by symmetry, we also have $[\fa]\hookrightarrow[\fb]$, whence $[\fa]=[\fb]$ as desired. This proves the $k=1$ case.

For $k>1$, form again $\fC$ and $\fE$, this time setting $\fc_F:=\xi_1^{s_1}\fa_F$. Then induction yields $[\fc]^\sharp=[\fd]$ and $[\fe]^\sharp=[\ff]$. To conclude, use the decompositions $[\fa]=[\fc]\oplus[\fe]$, $[\fb]=[\fd]\oplus[\ff]$ which hold because in the sequences \eqref{e:c}, \eqref{e:e} the extremes have coprime annihilators.
\end{proof}

\end{section}


\begin{section}{The pairings}

In this section, $A$ and $B$ denote abelian varieties over the global field $K$.

\subsubsection{The Selmer groups}\label{sss:sel}
Let $i\colon A_{p^n}\hookrightarrow A$ be the group scheme of $p^n$-torsion of $A$.
The $p^n$-Selmer group $\Sel_{p^n}(A/K)$ is defined to be the kernel of the composition
\begin{equation}\label{e:nselmer} \begin{CD} \coh^1_{\mathrm{fl}}(K,A_{p^n}) @>{i^*}>> \coh^1_{\mathrm{fl}}(K,A)@>{loc_K}>> \bigoplus_v \coh^1_{\mathrm{fl}}(K_v,A)\,, \end{CD}\end{equation}
where $\coh_{\mathrm{fl}}^\bullet$ denotes the flat cohomology and $loc_K$ is the localization map to the direct sum of local cohomology groups over all places of $K$.
The same definition works over any finite extension $F/K$. Taking the direct limit as $n\rightarrow\infty$, we get
\begin{equation}\label{e:selmer} \Sel_{p^\infty}(A/F):=\Ker\big(\coh^1_{\mathrm{fl}}(F,A_{p^\infty})\lr \bigoplus_{\text{all}\ v}\coh^1_{\mathrm{fl}}(F_v,A)\big) \end{equation}
where $A_{p^{\infty}}$ is the $p$-divisible group associated with $A$. The Selmer group $\Sel_{p^\infty}(A/L)$ is then defined by taking the inductive limit over all finite subextensions. The Galois group $\Gamma$ acts on $ \Sel_{p^\infty}(A/L)$ turning it into a $\La$-module.

The Selmer group sits in an exact sequence
\begin{equation} \label{e:selmersha} 0 \lra \QQ_p /\ZZ_p\otimes_{\ZZ}A(F) \lra \Sel_{p^\infty}(A/F) \lra \Sha(A/F)[p^\infty] \lra 0 , \end{equation}
where
$$\Sha(A/F) := \Ker\big( \coh^1_{\mathrm{fl}}(F,A)\lr\bigoplus_v \coh^1_{\mathrm{fl}} (F_v,A)\big) $$
is the Tate-Shafarevich group for $A/F$.

\begin{subsection}{The height pairing} \label{su:d}
Let $A^t$ denote the dual abelian variety of $A$. Then we have the N\'eron-Tate height pairing
\begin{equation} \label{e:nthp} \tilde h_{A/K}\colon A(K)\times A^t(K)\lr \RR\,. \end{equation}
We briefly recall the definition of $\tilde h_{A/K}$: for details, see \cite[V, \S4]{lan83}. Let
$$P_A\longrightarrow A\times A^t$$
denote the Poincar\'e line bundle: then $\tilde h_{A/K}$ is the canonical height on $A\times A^t$ associated with the divisor class corresponding to $P_A$.

\begin{proposition}\label{p:2} Let $\phi\colon A\rightarrow B$ and $\phi^t\colon B^t\rightarrow A^t$ be an isogeny and its dual.
Then the following diagram is commutative:
\[ \begin{CD} \tilde h_{A/K}\colon & A(K) &\,\times & \,A^t(K)  @>>>  \RR\\
 &  @VV{\phi}V @AA{\phi^t}A @| \\
\tilde h_{B/K}\colon & B(K)&\,\times  & \,B^t(K)  @>>> \RR.
\end{CD} \]
\end{proposition}

\begin{proof} By definition of the N\'eron-Tate pairing and functorial properties of the height (\cite[Proposition V.3.3]{lan83}), $\tilde h_{A/K}(\cdot,\phi^t(\cdot))$ and $\tilde h_{B/K}(\phi(\cdot),\cdot)$ are the canonical heights on $A\times B^t$ associated with the divisor classes corresponding respectively to ${(1\times\phi^t)^*(P_A)}$ and $(\phi\times 1)^*(P_B)$. But the theorem in \cite[\S13]{mum74} implies
\begin{equation}\label{e:i}
(1\times \phi^t)^*(P_A)\simeq (\phi\times 1)^*(P_B)
\end{equation}
(see \cite[p. 130]{mum74}).
\end{proof}

\subsubsection{The $p$-adic height pairing} We extend \eqref{e:nthp} to a pairing of $\ZZ_p$-modules.

\begin{lemma}\label{l:pmw} For every finite extension $F/K$ there exists a $p$-adic height pairing
\begin{equation} \label{e:ntpairing} h_{A/F}\colon \big( \ZZ_p\otimes A(F)\big) \times \big(\ZZ_p \otimes A^t(F)\big) \lr E_F,\end{equation}
where $E_F$ is a finite extension of $\QQ_p$, with the left and right kernels equal to the torsion parts of $\ZZ_p\otimes A(F)$ and $\ZZ_p \otimes A^t(F)$. If $char(K)=p$ one can choose $E_F=\QQ_p$.
\end{lemma}

\begin{proof}
If $char(K)=p$, then after scaling by a factor $\log(p)$, the pairing $\tilde h_{A/F}$ takes values in $\QQ$ (see for example \cite[\S3]{sch82}): in this case we define $h_{A/F}$ by $\tilde h_{A/F}=-\log(p)h_{A/F}$ and extend it to get \eqref{e:ntpairing}. In general, the image of the N\'eron-Tate height $\tilde{h}_{A/F}$ generates a subfield $E_F'\subset \RR$. By the Mordell-Weil Theorem, $E_F'$ is finitely generated over $\QQ$, and hence can be embedded into a finite extension $E_F$ of $\QQ_p$. Then we have the pairing
$$\tilde{h}_{A/F}\colon A(F)\times A^t(F) \lr E_F'\subset E_F\,,$$
which is obviously continuous on the $p$-adic topology, and thus can be extended to a pairing $h_{A/F}$ as required.
Since the left and right kernels of $\tilde{h}_{A/F}$ are the torsion parts of $A(F)$ and $A^t(F)$, if $x_1,...,x_r$ and $y_1,...,y_r$ are respectively $\ZZ$-basis of the free parts of $A(F)$ and $A^t(F)$, then
$${\det}_{i,j}(h_{A/F}(x_i,y_j))={\det}_{i,j}(\tilde{h}_{A/F}(x_i,y_j))\not=0,$$
which actually means that $h_{A/F}$ is non-degenerate on the free part of its domain.
\end{proof}

\end{subsection}


\begin{subsection}{The Weil-Barsotti formula} \label{su:wbformula}
Let $W^A$ denote the canonical bi-extension of $A\times A^t$ obtained by removing the zero section of $P_A$ (see \cite[IX, \S1]{sga7i}, or \cite[p.\,311]{mum68}).
For each $a\in A$ (resp. $a'\in A^t$), let ${}_aW^A$ (resp. $W^A_{a'}$) denote the part of $W^A$ sitting over $\{a\}\times A^t$ (resp. $A\times \{a'\}$).
Then we have exact sequences
$$\xymatrix{0\ar[r] &  \G_m\ar[r] &  {}_aW^A\ar@{->>}[r] & \{a\}\times A^t\simeq A^t ,}$$
and
$$\xymatrix{0\ar[r] &  \G_m\ar[r] &  W^A_{a'}\ar@{->>}[r] & A\times \{a'\}\simeq A.}$$
Thus, the assignment $a'\mapsto `` 0\longrightarrow \G_m\longrightarrow W^A_{a'}\longrightarrow  A"$ gives rise to the isomorphism (the Weil-Barsotti formula)
\begin{equation}\label{e:wb}
A^t \simeq \cExt_K^1(A,\G_m)
\end{equation}
of sheaves over the flat topology of $K$ (it might be useful to recall that the functors $\cExt^i$ can be interpreted as Yoneda extensions: \cite[III, 1.6(b)]{mil80}). Associated with an isogeny $\phi\in\Hom(A,B)$ there is the ``pull-back'' $\phi^*\colon \cExt_K^1(B,\G_m)\rightarrow \cExt_K^1(A,\G_m)$ induced from the Yoneda pairing:
$$\cExt_K^1(B,\G_m)\times \Hom(A,B)\longrightarrow \cExt_K^1(A,\G_m).$$
The isomorphism (\ref{e:i}) implies the commutative diagram
\begin{equation}\label{e:ext} \begin{CD}
B^t @>{\sim}>> \cExt_K^1(B,\G_m)\\
@VV{\phi^t}V @VV{\phi^*}V\\
A^t @>{\sim}>>  \cExt_K^1(A,\G_m)\,.
\end{CD}\end{equation}
When $K$ is a number field, let $C$ denote the spectrum of the ring of integers of $K$, and when $K$ is a function field with constant field $\FF$, let $C$ denote the unique connected smooth complete curve over $\FF$ having $K$ as its function field.

Let $U=\Spec R$, where $R$ is a Dedekind domain in $K$, be an open set of $C$ so that $A$ has good reduction at every place of $U$, and let $\A$ and $\A^t$ denote the N\'eron model of $A$ and $A^t$ over $U$. Then $W$ extends uniquely to a bi-extension $\mathcal{W}$ of $\A\times\A^t$ by $\G_m$ and we have the generalized Weil-Barsotti formula for sheaves on the smooth site of $U$ (see \cite[VIII, 7.1b]{sga7i}, \cite[III.18]{ort66} or \cite[III.C.12]{mil86a}):
\begin{equation}
\A^t \stackrel{\sim}{\longrightarrow}  \cExt_U^1(\A,\G_m).
\end{equation}

The isogenies $\phi$ and $\phi^t$ extend uniquely to $\phi\colon\A\rightarrow\B$ and $\phi^t\colon\B^t\rightarrow\A^t$ (\cite[\S1.2]{blr90}) and the Yoneda pairing induces the pull-back
$$\phi^*\colon \cExt_U^1(\B,\G_m)\longrightarrow  \cExt_U^1(\A,\G_m).$$

\begin{lemma}\label{l:pullback}
We have the commutative diagram:
\begin{equation}\label{e:extu}\begin{CD}
\mathcal B^t @>{\sim}>> \cExt_U^1(\mathcal B,\G_m)\\
@VV{\phi^t}V @VV{\phi^*}V\\
\mathcal A^t @>{\sim}>>  \cExt_U^1(\mathcal A,\G_m)\,.
\end{CD}\end{equation}
\end{lemma}

\begin{proof}
The lemma is local: by \cite[Proposition I.3.24(b)]{mil80} it is enough to show the commutative diagram
\[\begin{CD}
\mathcal B^t(V) @>{\sim}>> \Ext_V^1(\mathcal B,\G_m)\\
@VV{\phi^t}V @VV{\phi^*}V\\
\mathcal A^t(V) @>{\sim}>> \Ext_V^1(\mathcal A,\G_m)
\end{CD}\]
for the case where $V=\Spec S$ and $S$ is \'etale over the polynomial ring $R_0[X_1,...,X_n]$ for some $n$ and some localization $R_0$ of $R$. Moreover, by \cite[Proposition I.3.19]{mil80} we can assume that $S$ is a localization of $R_0[X_1,...,X_n,T]/(P)$, $P$ an irreducible polynomial. Without loss of generality $R_0$ can be taken a unique factorization domain: then so is $R_0[X_1,...,X_n]$ and it follows that $S$ is an integral domain. Let $M$ be the field of fractions of $S$: then $\A^t(V)=A^t(M)$, $\B^t(V)=B^t(M)$ (\cite[\S1.2]{blr90}) and the diagrams
\begin{equation}\begin{CD}
\A^t(V) @>{\sim}>> \Ext_V^1(\A,\G_m)\\
@| @VVV\\
A^t(M) @>{\sim}>>  \Ext_M^1(A,\G_m)
\end{CD}\end{equation}
and
\begin{equation}\begin{CD}
\B^t(V) @>{\sim}>> \Ext_V^1(\B,\G_m)\\
@| @VVV\\
B^t(M) @>{\sim}>>  \Ext_M^1(B,\G_m)
\end{CD}\end{equation}
are commutative. The lemma follows by applying \eqref{e:ext}.
\end{proof}

\end{subsection}


\begin{subsection}{The Cassels-Tate pairing} \label{CTpairing}
The main reference of this section is \cite[II.2 and II.5]{mil86a}.

Let
$$\langle\;,\;\rangle_{A/K}\colon \Sha(A/K)\times \Sha(A^t/K)\longrightarrow \QQ/\ZZ$$
denote the Cassels-Tate pairing. (We shall recall its construction below.)

Let $A$, $B$ be abelian varieties defined over the global field $K$ and let $A^t$, $B^t$ be the dual abelian varieties.
Suppose $\phi\colon A\rightarrow B$ is an isogeny and $\phi^t\colon B^t\rightarrow A^t$ is its dual.

\begin{proposition}\label{p:1}
We have the commutative diagram:
\[ \begin{CD} \langle\;,\;\rangle_{A/K}\colon & \Sha(A/K) &\,\times & \,\Sha(A^t/K)  @>>> \QQ/\ZZ \\
 &  @VV{\phi_*}V @AA{\phi^t_*}A @| \\
\langle\;,\;\rangle_{B/K}\colon & \Sha(B/K)&\,\times  & \,\Sha(B^t/K)  @>>> \QQ/\ZZ.
\end{CD} \]
\end{proposition}

This result must be well-known. It has been used before (see the proof of \cite[I, Theorem 7.3]{mil86a}); in the case of elliptic curves over number fields, it is \cite[Theorem 1.2]{cas65}. We will provide a proof below.\\

Let $U$ be as in Subsection \ref{su:wbformula} and let us be in the \'etale site of $U$.
For a sheaf $\F$, let $\coh^r_\mathrm{c}(U,\F)$ denote the $r$th cohomology with compact support.

Since $\mathcal{H}om_U(\A^t,\G_m)=0$, the local-global spectral sequence for Exts (see \cite[Theorem III.1.22]{mil80}) gives rise to a map $\coh_{\mathrm{\acute et}}^1(U,\A^t)\rightarrow \Ext_U^2(\A,\G_m)$. Thus, associated with each element $\xi\in\coh_{\mathrm{\acute et}}^1(U,\A^t)$ there is an exact sequence
$$0\longrightarrow \G_m\longrightarrow D_{\xi}\stackrel{\alpha_{\xi}}{\longrightarrow} E_{\xi}\longrightarrow \A\longrightarrow 0$$
which divides into two exact sequences:
$$0\longrightarrow \image(\alpha_{\xi})\longrightarrow E_{\xi}\longrightarrow \A\longrightarrow 0$$
and
$$0\longrightarrow \G_m\longrightarrow D_{\xi}\longrightarrow \image(\alpha_{\xi})\longrightarrow 0\,.$$
These induce two boundary maps:
$$\partial_{\xi,1}\colon\coh^1_\mathrm{c}(U,\A)\longrightarrow \coh^2_\mathrm{c}(U,\image(\alpha_{\xi}))$$
and
$$\partial_{\xi,2}\colon\coh^2_\mathrm{c}(U,\image(\alpha_{\xi}))\longrightarrow \coh^3_c(U,\G_m).$$
Composing these with the isomorphism $\coh^3_c(U,\G_m)\simeq\QQ/\ZZ$ (see \cite[II.2.6]{mil86a}), we get the map
$$\digamma_{\xi}:\coh^1_\mathrm{c}(U,\A)\longrightarrow\QQ/\ZZ.$$
Then we define the pairing
\begin{eqnarray*}
\langle\;,\;\rangle_{\A/U}: \coh^1_\mathrm{c}(U,\A)\times \coh_{\mathrm{\mathrm{\acute et}}}^1(U,\A^t) & \longrightarrow & \QQ/\ZZ\\
(\eta,\xi) & \mapsto & \digamma_{\xi}(\eta).
\end{eqnarray*}
Define
$$D^1(U,\A):=\image\big(\coh^1_\mathrm{c}(U,\A)\,\lr \coh^1_\mathrm{et}(U,\A)\big)=\Ker\big(\coh^1_\mathrm{et}(U,\A)\lr \prod_{v\not\in U}\coh^1(K_v,A)\big).$$
Then we have (see \cite[II.5.5]{mil86a})
$$D^1(U,\A)=\Sha(A/K),\;\;\;\;D^1(U,\A^t)=\Sha(A^t/K)$$
and the pairing $\langle\;,\;\rangle_{\A/U}$ induces (see \cite[II.5.3, II.5.6]{mil86a}) the Cassels-Tate pairing
$$\Sha(A/K)\times \Sha(A^t/K)\longrightarrow\QQ/\ZZ.$$

\begin{proof}[\bf{Proof of Proposition \ref{p:1}}]  We show that the diagram
\[ \begin{CD} \langle\;,\;\rangle_{\mathcal A/U}\colon & \coh^1_\mathrm{c}(U,\A) &\,\times &\,\coh_{\mathrm{\acute et}}^1(U,\A^t)  @>>> \QQ/\ZZ\\
 &  @VV{\phi_*}V @AA{\phi^t_*}A @| \\
\langle\;,\;\rangle_{\mathcal B/U}\colon & \coh^1_\mathrm{c}(U,\B) &\,\times &\,\coh_{\mathrm{\acute et}}^1(U,\B^t)  @>>> \QQ/\ZZ.
\end{CD} \]
commutes. Suppose $\xi\in \coh_{\mathrm{\acute et}}^1(U,\B^t)$ and it is associated with the exact sequence
\begin{equation} \label{ext2xi}
0\longrightarrow \G_m\longrightarrow D_{\xi}\stackrel{\beta_{\xi}}{\longrightarrow} E_{\xi}\longrightarrow \B\longrightarrow 0. \end{equation}
From Lemma \ref{l:pullback} we get the diagram
\[ \begin{CD}
\coh_{\mathrm{\acute et}}^1(U,\B^t)  @>>> \Ext_U^2(\B,\G_m)\\
@VV{\phi_*^t}V @VV{\phi^*}V\\
\coh_{\mathrm{\acute et}}^1(U,\A^t)  @>>> \Ext_U^2(\A,\G_m).
\end{CD} \]
Hence $\phi^t_*(\xi)$ is associated with the pull back of \eqref{ext2xi}, i.e. the exact sequence which is the composition of the top row of the commutative diagram
\begin{equation} \label{pullbackxi} \begin{CD}
0 @>>> \image(\beta_{\xi}) @>>>  E_{\xi}\times_{\B}\A @>>> \A @>>> 0\\
&  & @| @VVV @VVV \\
0\ @>>> \image(\beta_{\xi}) @>>>  E_{\xi} @>>> \B @>>> 0
\end{CD} \end{equation}
with
\[\begin{CD} 0 @>>> \G_m @>>> D_\xi @>>> \image(\beta_{\xi}) @>>> 0.\end{CD}\]
Let $\phi_*\colon\coh_c^1(U,\A)\rightarrow \coh_c^1(U,\B)$ be the map induced from $\phi$. Taking the cohomology of \eqref{pullbackxi} we have
\[ \begin{CD}
\coh_c^1(U,\A)  @>\partial_{\phi^t_*(\xi),1}>> \coh_c^2(U,\image(\phi^*\beta_{\xi}))\\
@VV{\phi_*}V @| \\
\coh_c^1(U,\B)  @>\partial{\xi,1}>> \coh_c^2(U,\image(\beta_{\xi}))
\end{CD} \]
and thus
$$\langle \phi_*(\eta),\xi \rangle_{\mathcal B,U}=\partial_{\xi,2}\circ\partial_{\xi,1}(\phi_*(\eta))=\partial_{\xi,2}\circ\partial_{\phi^t_*(\xi),1}(\eta)=\langle \eta,\phi_*^t(\xi) \rangle_{\mathcal A,U}\,.$$
\end{proof}

The kernels of the Cassels-Tate pairing are the divisible subgroups of $\Sha(A/K)$ and $\Sha(A^t/K)$: see \cite[II, Theorem 5.6(a) and III, Corollary 9.5]{mil86a}.
\end{subsection}
\end{section}


\begin{section}{The algebraic functional equation for abelian varieties} \label{s:al}

In this section, we apply Theorem \ref{t:al} to Selmer groups.

\begin{subsection}{Selmer groups over $\ZZ_p^d$-extensions} \label{se:se}
Let $A$ be an abelian variety defined over our global field $K$. We take the extension $L/K$ as established in Section \ref{s:setting}, with layers $K_n$ and $\Gamma=\Gal(L/K)$. Recall our hypothesis that $L/K$ is unramified outside a finite set of places $S$. We also assume that $A$ has {\em potentially ordinary} reduction at every place in $S$: then $X_p(A/L)$, the Pontryagin dual of $\Sel_{p^{\infty}}(A/L)$, is finitely generated over $\La$ (if the reduction is ordinary at all places in $S$ this is \cite[Theorem 5]{tan10a}; for passing from potentially ordinary to ordinary, one can reason similarly to \cite[Lemma 2.1]{ot09}).

\subsubsection{The divisible limit} For each $n$, let $\Sel_{p^{\infty}}(A/K_n)_{div}$ denote the $p$-divisible part of $\Sel_{p^{\infty}}(A/K_n)$ and let $Y_p(A/K_n)$ be its Pontryagin dual. Also, let $Y_p(A/L)$ denote the projective limit of $\{Y_p(A/K_n)\}_n$, so that
$$Y_p(A/L) = \Sel_{div}(A/L)^\vee,$$
where
$$\Sel_{div}(A/L):=\varinjlim_{n}\Sel_{p^{\infty}}(A/K_n)_{div}\,.$$

\begin{mytheorem}[Tan] \label{t:flat}
Suppose $X_p(A/L)$ is a torsion $\La$-module. Then there exist relatively prime simple elements $f_1,...,f_m$  {\em (}$m\geq 1${\em )} so that
$$f_1\cdots f_m\cdot \Sel_{div}(A/L)=0.$$
\end{mytheorem}

This theorem was originally proved in a first draft of \cite{tan10b} as a consequence of a careful study of the growth of Selmer groups. But it really is just a special case of the following. (For a more detailed statement in the Selmer group case, see \cite[Theorem 5]{tan10b}.)

\begin{mytheorem}
Let $M$ be a cofinitely generated torsion $\La$-module. Then there exist relatively prime simple elements $f_1,...,f_m$  {\em (}$m\geq 1${\em )} so that
$$f_1\cdots f_m\cdot (M^{\Gal(L/F)})_{div}=0$$
for any finite intermediate extension $K\subset F\subset L$.
\end{mytheorem}

\begin{proof}
By hypothesis, there is some $\xi\in\La-\{0\}$ annihilating $M$. Let $f\in\La$ be such that $\omega(f)=0$ for all $\omega\in\Delta_\xi\,$. Fix a finite intermediate extensions $F/K$ and, to lighten notation, put $G:=\Gal(F/K)$ and $\cO$ the ring of integers of $\QQ(\boldsymbol\mu_{p^d})$, where $p^d$ is the exponent of $G$. Also, let $N:=\cO\otimes_{\ZZ_p}M^{\Gal(L/F)}$. Then defining, as in \eqref{e:idempotent},  $e_\omega':=\sum_{g\in G}\omega(g^{-1})g$ for each $\omega\in G^\vee$, one finds $|G|\cdot N=\sum e_\omega'N$ and $\La$ acts on $e_\omega'N$ by $\lambda\cdot n=\omega(\lambda)n$. In particular, one has $f\cdot e_\omega'N=0$ for all $\omega\in\Delta_\xi$. On the other hand, if $\omega\notin\Delta_\xi$ then $e_\omega'N$ is finite because it is a cofinitely generated module over the finite ring $\omega(\La)/(\omega(\xi))$. It follows that $f\cdot N$ is finite. Since a finite divisible group must be trivial, this proves that $f\cdot N_{div}=0$, and hence $f\cdot (M^{\Gal(L/F)})_{div}=0$.

It remains to prove that one can find a product of distinct simple elements which is killed by $\Delta_\xi$. This is a consequence of Monsky's theorem: one has $\Delta_\xi=\cup T_j$ where the $T_j$'s are $\ZZ_p$-flats, and by definition for each $T_j$ there is at least one simple element vanishing on it.
\end{proof}

\end{subsection}


\begin{subsection}{The Cassels-Tate $\Gamma$-system} Define
\begin{equation} \label{e:bnabvar} \fb_n:=\Sel_{p^{\infty}}(A/K_n)/\Sel_{p^{\infty}}(A/K_n)_{div} \end{equation}
and let $\fa_n$ denote its Pontryagin dual. In fact, we can make these more explicit.
First, if $\Sha_{p^\infty} (A/K_n)$ denotes the $p$-primary part of the Tate-Shafarevich group and $\Sha_{p^\infty} (A/K_n)_{div}$ is its $p$-divisible part, then
\begin{equation} \label{e:bnsha} \fb_n=\Sha_{p^\infty}(A/K_n)/\Sha_{p^\infty}(A/K_n)_{div}\,. \end{equation}
Also, if $A^t/K$ denotes the dual abelian variety of $A$, then we can identify
\begin{equation} \label{e:ansha} \fa_n=\Sha_{p^\infty}(A^t/K_n)/\Sha_{p^\infty}(A^t/K_n)_{div}\,, \end{equation}
thanks to the duality via the perfect pairing
\begin{equation} \label{e:perfectCT}
\langle\;,\;\rangle_n\colon \fa_n\times\fb_n \lr \QQ_p/\ZZ_p \end{equation}
induced from the Cassels-Tate pairing on $\Sha(A^t/K_n)\times \Sha(A/K_n)$.

Let $\fr_m^n$ and $\fk_m^n$ be the morphisms induced respectively from the restriction
$$\coh^1(K_m,A^t\times A)\longrightarrow \coh^1(K_n,A^t\times A)$$
and the co-restriction
$$\coh^1(K_n,A^t\times A)\longrightarrow \coh^1(K_m,A^t\times A).$$
Then the collection $\fA:=\{\fa_n,\fb_n,\langle\,,\rangle_n,\fr_m^n,\fk_m^n\}$ satisfies axioms ($\Gamma$-1)-($\Gamma$-4).
As in Section \ref{s:con}, write $\fa$ and $\fb$ for the projective limits of $\{\fa_n\}_n$ and $\{\fb_n\}_n$. We have the exact sequence
\begin{equation}\label{e:axy}  0\lr \fa\lr X_p(A/L)\lr Y_p(A/L)\lr 0. \end{equation}
In particular, if $X_p(A/L)$ is a finitely generated torsion $\La$-module, then so are $\fa$ and, by the following lemma, $\fb$.

\begin{lemma} \label{l:isogAtA}
The $\La$-module $X_p(A^t/L)$ is torsion if and only if so is $X_p(A/L)$.
\end{lemma}

\begin{proof}
Any isogeny $\varphi\colon A\rightarrow A^t$ defined over $K$ gives rise to a homomorphism of $\La$-modules $X_p(A^t/L)\rightarrow X_p(A/L)$ with kernel and co-kernel annihilated by $\deg(\varphi)$.
\end{proof}

\subsubsection{The module $\fa^{00}$}\label{ss:ddiv} To study $\fa_n^0$, we first consider the small piece $\fa^{00}_n$ which is the image of
$$\fs_n^{00}:=\Ker\!\big(\Sel_{p^{\infty}}(A^t/K_n)\,\lr \Sel_{p^{\infty}}(A^t/L)\big)$$
under the projection $\Sel_{p^{\infty}}(A^t/K_n)\twoheadrightarrow \fa_n$. Obviously, $\fs_n^{00}$ is a $\Gamma_n$-submodule of
$$\fs_n^0:=\Ker\big(\coh_{fl}^1(K_n,A^t_{p^{\infty}})\,\lr \coh_{fl}^1(L,A^t_{p^{\infty}})\big)=\coh^1(\Gamma^{(n)}, A^t_{p^{\infty}}(L))\,.$$

\begin{lemma} \label{l:h1d}
All the groups $\coh^1\!\big(\Gamma^{(n)},A^t_{p^{\infty}}(L)\big)$ are finite.
\end{lemma}

\begin{proof} We just prove the $d=1$ case, since it is the only one we are going to actually use. For the general case, see \cite[Proposition 3.3]{gr03}.

Write
\begin{equation}\label{e:ALtors} D:=A^t_{p^{\infty}}(L)=A^t(L)[p^{\infty}] \end{equation}
and let $D_{div}$ be its $p$-divisible part. We have an exact sequence
\begin{equation}\label{e:hddd}
(D/D_{div})^{\Gamma^{(n)}} \lr \coh^1(\Gamma^{(n)},D_{div})\,\lr \coh^1(\Gamma^{(n)},D) \,\lr \coh^1(\Gamma^{(n)}, D/D_{div})\,.
\end{equation}
If $d=1$ and $\Gamma=\gamma^{\ZZ_p}$, then we observe that
$$(\gamma^{p^n}-1)D_{div}=D_{div},$$
since $\Ker(D\stackrel{\gamma^{p^n}-1}{\lr} D)=D^{\Gamma^{(n)}}$ is finite. Therefore, $\coh^1(\Gamma^{(n)},D_{div})=0$
and hence, as $D/D_{div}$ is finite, from the exact sequence \eqref{e:hddd} we deduce (see e.g.\! \cite[Lemma 4.1]{bl09a})
\begin{equation}\label{e:ddd}
|\coh^1(\Gamma^{(n)},D)|\leq |D/D_{div}|.
\end{equation}
\end{proof}

\begin{lemma}\label{l:a00}
The projective limit
$$\fa^{00}:=\varprojlim\fa_n^{00}$$
is pseudo-null.
\end{lemma}

\begin{proof} If $d=1$, $\fa^{00}$ is pseudo-null as it is a finite set: inequality \eqref{e:ddd} together with the surjection $\fs^{00}_n\twoheadrightarrow\fa^{00}_n$ gives a bound on its cardinality.

Let $D$ be as in \eqref{e:ALtors} and let $r$ the $\ZZ_p$-rank of its Pontryagin dual $D^\vee$.
Then the action of $\Gamma$ gives rise to a representation
\begin{equation} \label{e:ro} \rho\colon\Gamma\lr \Aut(D_{div})\simeq \GL(r,\ZZ_p)\,. \end{equation}
For each $\gamma$, let $f_{\gamma}(x)$ be the characteristic polynomial of $\rho(\gamma)$. Then $f_{\gamma}(\gamma)$ is an element in $\La$ which annihilates $D_{div}$. Let $m_0$ be large enough so that $A^t(K_{m_0})[p^\infty]$ generates $D/D_{div}$. Then for every $\gamma$, $(\gamma^{p^{m_0}}-1)f_{\gamma}(\gamma)$ annihilates both $\fs_n^{00}$ and $\fa_n^{00}$ for $n\geq m_0$, as we have
$$(\gamma^{p^{m_0}}-1)f_{\gamma}(\gamma)\cdot D=0.$$
If $d\geq 2$ and $\Gamma=\bigoplus_{i=1}^d\gamma_i^{\ZZ_p}$, then each $\delta_i:=(\gamma_i^{p^{m_0}}-1)f_{\gamma_i}(\gamma_i)$ lives in a different $\ZZ_p[T_i]$ under the identification $\Lambda=\ZZ_p[[T_1, ...,T_d]]$,  $\gamma_i-1=T_i$. Therefore $\delta_1,...,\delta_d$ are relatively prime and $\fa^{00}$ is pseudo-null by Lemma \ref{l:psn}.
\end{proof}

\begin{remark} \label{r:fintors} {\em Lemmas \ref{l:h1d} and \ref{l:a00} become trivial if $A_{p^\infty}(L)$ is finite. Actually, it happens pretty often that $A_{p^\infty}$ has only finitely many points in abelian or $p$-adic extensions, and even in the separable closure of $K$ if the latter is a characteristic $p$ function field: for example, this is always the case with non-isotrivial elliptic curves (see \cite[Theorem 4.2]{BLV09} for a proof). In the number field case, the paper \cite{zar87} provides a thorough discussion of the torsion of $A(F)$ for $F/K$ abelian, while results for $\Gal(F/K)$ a $p$-adic Lie group can be found in \cite[Proposition 3.2]{gr03}; in the function field setting, a sufficient condition for general abelian varieties can be found in \cite[Section 4]{vol95}. On the other hand, there are obvious counterexamples to the finiteness of $A_{p^\infty}(L)$ when $A/K$ is CM or isotrivial. See also \cite[Proposition 2.3.8]{tan10b} for necessary conditions to have an infinite $A_{p^\infty}(L)$.} \end{remark}

Put ${\bar \fa}_n^0:=\fa_n^0/\fa_n^{00}$. Applying the snake lemma to the diagram $$\begin{CD}
0 @>>> \Sel_{p^{\infty}}(A^t/K_n)_{div} @>>> \Sel_{p^{\infty}}(A^t/K_n) @>>> \fa_n @>>> 0\\
&& @VVV @VVV @VVV\\
0 @>>> \Sel_{div}(A^t/L) @>>> \Sel_{p^{\infty}}(A^t/L) @>>> \varinjlim\fa_n @>>> 0\end{CD}$$
we find an injection
\begin{equation}\label{e:subset}
{\bar \fa}_n^0\hookrightarrow \Sel_{div}(A^t/L)/\Sel_{p^{\infty}}(A^t/K_n)_{div}\,.
\end{equation}
By construction we have an exact sequence
\begin{equation}\label{e:aabar}
0\longrightarrow \fa^{00}\longrightarrow \fa^0\longrightarrow {\bar \fa}^0:=\varprojlim_n{\bar \fa}_n^0\longrightarrow 0.
\end{equation}

\subsubsection{The functional equation}
Suppose $M$ is a finitely generated torsion $\La$-module and $[M]=\bigoplus_{i=1}^m \La/\xi_i^{r_i}\La$. Define
$$[M]_p:=\bigoplus_{(\xi_i)=(p)}\La/\xi_i^{r_i}\La\subset [M],$$
$$[M]_{np}:=\bigoplus_{(\xi_i)\not=(p)}\La/\xi_i^{r_i}\La\subset [M],$$
so that $[M]=[M]_p\oplus [M]_{np}\,.$ Then
\begin{equation}\label{e:pd}
[M]_p^{\sharp}=\bigoplus_{(\xi_i)=(p)}\La/(\xi_i^{\sharp})^{r_i}\La
=\bigoplus_{(\xi_i)=(p)}\La/\xi_i^{r_i}\La=[M]_p\,.
\end{equation}
Recall the submodules $[M]_{si}$ and $[M]_{ns}$ defined in \S\ref{ss:nons}. By \eqref{e:fs} we have
\begin{equation}\label{e:fld}
[M]_{si}^{\sharp}=[M]_{si}\,.
\end{equation}
With this notation, Theorem \ref{t:flat} implies the following.

\begin{corollary}\label{c:flat} Suppose $X_p(A/L)$ is torsion over $\La$. Then the following holds:
\begin{enumerate}
\item[(a)] $[Y_p(A/L)]= [Y_p(A/L)]_{si}$.
\item[(b)] $[X_p(A/L)]_{ns}=[\fa]_{ns}$.
\item[(c)] $[\fa']_{ns}=[\fa]_{ns}$.
\end{enumerate}
\end{corollary}

\begin{proof}
The equality at point (a) follows directly from Theorem \ref{t:flat} and, together with the exact sequence \eqref{e:axy}, it implies (b). Lemma \ref{l:a00} and the sequence \eqref{e:aabar} yield $\fa^0\sim\bar\fa^0$; by \eqref{e:subset} and Theorem \ref{t:flat} (applied to $A^t$) we have $[\bar\fa^0]_{ns}=0$, thus $[\fa^0]_{ns}=0$ and (c) follows by exact sequence \eqref{e:0'}.
\end{proof}

Now we can prove Theorem \ref{t:xaat}.

\begin{proof}[\bf{Proof of Theorem \ref{t:xaat}}]
We just need to check \eqref{e:xaat} when $X_p(A/L)$ is $\La$-torsion. Corollaries \ref{c:flat} and \ref{c:nons} imply
\begin{equation}\label{e:nf}
[X_p(A/L)]_{ns}^{\sharp}=[\fa]_{ns}^\sharp=[\fa']_{ns}^\sharp=[\fb']_{ns}=[\fb]_{ns}=[X_p(A^t/L)]_{ns}
\end{equation}
(where the last passage comes from the fact that Corollary \ref{c:flat} holds as well replacing $A$, $\fa$ with $A^t$, $\fb$).
By \eqref{e:fld} we see that
$$[X_p(A/L)]_{si}=[X_p(A/L)]_{si}^{\sharp}.$$
The proof of Lemma \ref{l:isogAtA} shows that any isogeny $A\rightarrow A^t$ gives rise to an isomorphism
\begin{equation}\label{e:qp}
\QQ_p X_p(A/L)\stackrel{\sim}{\longrightarrow} \QQ_p X_p(A^t/L),
\end{equation}
and hence $[X_p(A^t/L)]_{si}=[X_p(A/L)]_{si}=[X_p(A/L)]_{si}^{\sharp}$. Therefore,
$$X_p(A^t/L)\sim X_p(A/L)^{\sharp}.$$
This together with \eqref{e:qp} implies
$$[X_p(A/L)]_{np}^{\sharp}=[X_p(A^t/L)]_{np}=[X_p(A/L)]_{np}.$$
By \eqref{e:pd}, $[X_p(A/L)]_p^{\sharp}=[X_p(A/L)]_{p}$. Therefore,
$$X_p(A/L)^{\sharp}\sim X_p(A/L).$$
\end{proof}

\subsubsection{} The following proposition shows that the Cassels-Tate system $\{\fa_n,\fb_n\}$ can be a $\Gamma$-system also when $X_p(A/L)$ is not a torsion $\La$-module. Note that it also implies that $\fA$ enjoys property ({\bf T}) (just replace $L/K$ with an arbitrary $\ZZ_p^{d-1}$-subextension $L'/F$).

\begin{proposition}\label{p:shatate}
If $L/K$ only ramifies at good ordinary places, then $\fa$ and $\fb$ are torsion over $\Lambda$.
\end{proposition}

\begin{proof}
Recall that $Q(\Lambda)$ denotes the fraction field of $\La$. Suppose $\fa$ were non-torsion. Let $r$ and $s$ denote respectively the dimensions over $Q(\La)$ of the vector spaces $Q(\La)\fa$ and $Q(\La)X_p(A/L)$: by \eqref{e:axy}, the former is contained in the latter. Let $e_1,...,e_r,...,e_s\in X_p(A/L)$ form a basis of $Q(\La)X_p(A/L)$ such that $e_1,...,e_r$ are in $\fa$.

Write $\fa'=\Lambda\cdot e_1+\cdots+\Lambda\cdot e_r\subset \fa$ and  $X'=\Lambda\cdot e_1+\cdots+\Lambda\cdot e_s\subset X_p(A/L)$. Then $X_p(A/L)/X'$ is torsion over $\Lambda$, and hence is annihilated by some nonzero $\eta\in \Lambda$. Let $\omega\in \Gamma^\vee$ be a character not contained in $\Delta_\eta$ and, as usual, extend it to a ring homomorphism $\omega\colon\La \twoheadrightarrow \cO_\omega$ whose kernel we denote by $\K$. Then we have the exact sequences
\begin{equation}\label{e:tor1}
\begin{CD} \Tor_{\La} (\La/\K, X_p(A/L)/\fa' ) @>>> (\Lambda/\K)\otimes_{\La} \fa'@>{i_*}>> (\La/\K)\otimes_{\La} X_p(A/L) \end{CD}
\end{equation}
and
$$\begin{CD} 0 @>>>  X'/\fa'  @>>>  X_p(A/L)/\fa'  @>>> X_p(A/L)/X' @>>> 0 \end{CD}\,.$$
The fact that $X'/\fa'$ is free over $\La$ implies that the natural map
\begin{equation*}\label{e:tor2} \Tor_{\La} (\La/\K, X_p(A/L)/\fa' )\, \lr  \Tor_{\La} (\La/\K, X_p(A/L)/X' ) \end{equation*}
is an injection. Since the residue class of $\eta$ in $\La/\K\simeq \cO_\omega$, a discrete valuation ring, is nonzero, the $\cO_\omega$-module $\Tor_{\Lambda} (\La/\K, X_p(A/L)/X' )$ must be finite, and the same holds for $\Tor_{\La}(\La/\K, X_p(A/L)/\fa' )$. Thus the homomorphism $i_*$ in \eqref{e:tor1} must be injective, because $(\Lambda/\K)\otimes_{\La} \fa'$ is a free $\cO_\omega$-module,
and hence its image is free of positive rank over $\cO_\omega$. Let $\Gamma^\omega\subset \Gamma$ denote the kernel of the character $\omega$ and write $\Gamma_\omega:=\Gamma/\Gamma^\omega$, $\La_\omega:=\ZZ_p[\Gamma_\omega]$, $K_\omega:=L^{\Gamma^\omega}$.  Then we have the commutative diagram:
$$\begin{CD}\fa' @>>> \La_\omega\otimes_{\La} \fa' @>>> (\La/\K)\otimes_{\La} \fa'\\
@VVV @VVV @VVV \\
\fa @>>> \La_\omega\otimes_{\La} \fa @>>> (\Lambda/\K)\otimes_{\La} \fa \\
@VVV @VVV @VVV \\
X_p(A/L) @>>> \La_\omega\otimes_{\La} X_p(A/L) @>>> (\La/\K)\otimes_{\La} X_p(A/L).   \end{CD}$$
The diagram together with the duality implies that the image of the composition
$$(\cO_\omega\Sel_{p^\infty}(A/L))^{(\omega^{-1})}\hookrightarrow (\cO_\omega\Sel_{p^\infty}(A/L))^{\Gamma^\omega}\lr \cO_\omega(\varinjlim \Sha_{p^\infty}(A/F)/\Sha_{p^\infty}(A/F)_{div})$$
has positive $\ZZ_p$-corank. By the assumption of the proposition, we are allowed to apply the control theorem \cite[Theorem 4]{tan10a}, which states that the restriction map
$$res_\omega\colon\Sel_{p^\infty}(A/K_\omega)\longrightarrow \Sel_{p^\infty}(A/L)^{\Gamma^\omega}$$
has finite kernel and cokernel. Consequently, the image of
$$ \Sel_{p^\infty}(A/K_\omega)\longrightarrow \Sha_{p^\infty}(A/K_\omega)/\Sha_{p^\infty}(A/K_\omega)_{div}$$
must have positive $\ZZ_p$-corank. But this is absurd, as the target of the map is finite.

The proof for $\fb$ just replaces $A$ with $A^t$.
\end{proof}

\begin{remark} {\em Proposition \ref{p:shatate} actually holds under the weaker hypothesis that $L/K$ is ramified only at ordinary places. The proof first follows the same argument. Then, instead of using the control theorem, it applies Lemma 5.3.2 and Lemma 5.3.3 of \cite{tan10b} as well as the argument in their proofs to show that as long as $\omega$ is chosen to have nontrivial restriction on the decomposition subgroups at the ramified multiplicative places then the map $res_\omega\colon(\cO_\omega\Sel_{p^\infty}(A/K_\omega))^{(\omega^{-1})}\rightarrow (\cO_\omega\Sel_{p^\infty}(A/L))^{(\omega^{-1})}$ will have finite kernel and cokernel.}
\end{remark}

\subsubsection{} It is natural to ask if $\fA$ is pseudo-controlled: we tend to believe that this actually holds as long as $X_p(A/L)$ is torsion.
Propositions \ref{p:pscontrl} and \ref{p:gdordpscontrl} below give evidence to support our belief.

\begin{lemma} \label{l:zerocoker}
For $n\geq m$ the restriction map
$$\Sel_{p^\infty}(A/K_m)_{div} \longrightarrow (\Sel_{p^\infty}(A/K_n)^{\Gamma^{(m)}})_{div}$$
is surjective.
\end{lemma}

\begin{proof}
The commutative diagram of exact sequences
$$\xymatrix{  \Sel_{p^\infty}(A/K_m)_{div}=\Sel_{p^\infty}(A/K_m)_{div} \ar@{^{(}->}[r]\ar@<-9ex>[d]^{j'} \ar@<9ex>[d]^j & \Sel_{p^\infty}(A/K_m) \ar@{->>}[r]\ar[d]^i & \fb_m \ar[d]^{\fr_m^n}\\
(\Sel_{p^\infty}(A/K_n)^{\Gamma^{(m)}})_{div} \subset (\Sel_{p^\infty}(A/K_n)_{div})^{\Gamma^{(m)}} \ar@{^{(}->}[r] & \Sel_{p^\infty}(A/K_n)^{\Gamma^{(m)}} \ar[r] & \fb_n^{\Gamma^{(m)}} } $$
induces the exact sequence
$$\Ker (\fr_m^n) \longrightarrow \Coker (j)\longrightarrow \Coker (i).$$
Since $\Ker (\fr_m^n)$ is finite while $\Coker (j')$ is $p$-divisible, it is sufficient to show that $\Coker (i)$ is annihilated by some positive integer. Consider the commutative diagram of exact sequences
$$\xymatrix{ \Sel_{p^\infty}(A/K_m) \ar@{^{(}->}[r]\ar[d]^i &  \coh^1_{\mathrm{fl}}(K_m,A_{p^\infty})\ar[r]^-{loc_m} \ar[d]^{res_m^n} & \prod_{\text{all}\;v} \coh^1({K_m}_v, A) \ar[d]^{r_m^n}\\
\Sel_{p^\infty}(A/K_n)^{\Gamma^{(m)}} \ar@{^{(}->}[r] & \coh^1_{\mathrm{fl}}(K_n,A_{p^\infty})^{\Gamma^{(m)}} \ar[r]^-{loc_n} &  \prod_{\text{all}\;w} \coh^1({K_n}_w, A)^{\Gamma_w^{(m)}}}  $$
that induces the exact sequence
$$\Ker \big(\xymatrix{\image (loc_m)\ar[r]^{r_m^n} & \image (loc_n)}\big)\lr \Coker (i)\lr \Coker (res_m^n) \,.$$
By the Hochschild-Serre spectral sequence, the right-hand term $\Coker (res_m^n)$ is a subgroup of $\coh^2(K_n/K_m, A_{p^\infty}(K_n))$, and hence is annihilated by $p^{d(n-m)}=[K_n:K_m]$. Similarly, the left-hand term, being a subgroup of $\prod_v \coh^1({K_n}_v/{K_m}_v, A({K_n}_v))$, is also annihilated by $[K_n:K_m]$.
\end{proof}

\begin{lemma}\label{l:fingen}
If $L/K$ is a $\ZZ_p$-extension and $X_p(A/L)$ is a torsion $\La$-module, then there exists some $N$ so that
$$\Sel_{p^\infty}(A/K_n)_{div}=(\Sel_{p^\infty}(A/K_n)_{div})^{\Gamma^{(m)}}=(\Sel_{p^\infty}(A/K_n)^{\Gamma^{(m)}})_{div}$$
holds for all $n\geq m\geq N$.
\end{lemma}

\begin{proof}
The second equality is an easy consequence of the first one.
The assumption $\Gamma=\gamma^{\ZZ_p}$ implies that if $f\in\La$ is simple then $f$ divides $\gamma^{p^m}-1$ for some $m$. Therefore, by Theorem \ref{t:flat}, there exists an integer $N$ such that $(\gamma^{p^N}-1)\Sel_{div}(A/L)^{\Gamma^{(n)}}=0$ for every $n$. By Lemma \ref{l:h1d} the kernel of the map $\Sel_{p^\infty}(A/K_n)_{div} \rightarrow \Sel_{div}(A/L)^{\Gamma^{(n)}}$ is finite.
This implies that $(\gamma^{p^N}-1)\Sel_{p^\infty}(A/K_n)_{div}$ must be trivial, as it is both finite and $p$-divisible.
\end{proof}

\begin{proposition}\label{p:pscontrl}
If $L/K$ is a $\ZZ_p$-extension and $X_p(A/L)$ is a torsion $\La$-module, then $\fA$ is pseudo-controlled.
\end{proposition}

\begin{proof} We apply Lemmata \ref{l:zerocoker} and \ref{l:fingen}.
If we are given an element $x\in\Sel_{p^\infty}(A^t/K_n)$ with $res_n^l(x)\in\Sel_{p^\infty}(A^t/K_l)_{div}$ for some $l\geq n\geq N$, then we can find $y\in \Sel_{p^\infty}(A^t/K_N)_{div}$ such that $res_N^l(y)=res_n^l(x)$. Then $x-res_N^n(y)\in\Ker(res_n^l)\subset \fs_n^{00}$. This actually shows that $\fa_n^0=\fa_n^{00}$. Then apply Lemma \ref{l:a00}.
\end{proof}

\begin{proposition}\label{p:gdordpscontrl}
If $L/K$ is a $\ZZ_p^d$-extension ramified only at good ordinary places and $X_p(A/L)$ is a torsion $\La$-module, then $\fA$ is pseudo-controlled.
\end{proposition}

\begin{proof} Let $f_1,...,f_m$ be those simple element described in Theorem \ref{t:flat}. and write $g_i:=f_i^{-1} f_1\dots f_m$.
Since $f_i$ divides $\gamma_i^{p^{r_i}}-1$, for some $\gamma_i$ and $r_i$, by Theorem \ref{t:flat} we get
\begin{equation}\label{e:tem01}
g_i\cdot \Sel_{div}(A^t/L)\subset \Sel_{p^\infty}(A^t/L)^{\Psi_i^{(r_i)}},
\end{equation}
where $\Psi_i\subset \Gamma$ is the closed subgroup topologically generated by $\gamma_i$ and we use the notation $H^{(i)}:=H^{p^{r_i}}$ for any subgroup $H<\Gamma$. In view of Proposition \ref{p:pscontrl}, we may assume that $d\geq 2$. Then we can find $\delta_1,...,\delta_d$ as in the proof of Lemma \ref{l:a00} and such that the elements $\delta_jg_i$, $i=1,...,m$, $j=1,...,d$, are coprime. By construction each $\delta_i$ annihilates $\fs_n^{0}$ for all $n$. We are going to show that if $a=(a_n)_n$, $a_n\in \fa_n^0$, is an element in $\fa^0$, then for $n\geq r_i$,
\begin{equation}\label{e:tem02} \delta_jg_i\cdot a_n=0,\;  j=1,...,d. \end{equation}
Then it follows that $\delta_jg_i\cdot \fa^0=0$ for every $i$ and $j$, and hence $\fa^0$ is pseudo-null.

Fix $n\geq r_i$. Choose closed subgroups $\Phi_i\subset\Gamma$, $i=1,...,m$, isomorphic to $\ZZ_p^{d-1}$ so that $\Gamma=\Psi_i\oplus\Phi_i$, and then set $L^{(i)}=L^{\Phi_i^{(n)}}$ and let $L_l^{(i)}$ denote the $l$th layer of the $\ZZ_p$-extension $L^{(i)}/K_n$, so that $\Gal(L^{(i)}/L_l^{(i)})$ is canonically isomorphic to $\Psi_i^{(n)}$.

For each $l\geq n$, let $\xi_l$ be a pre-image of $a_l$ under
$$\xymatrix{\Sel_{p^\infty}(A^t/K_l)\ar[r]^-{\pi_l} & \fa_l}$$
and denote by $\xi'_l$ the image of $\xi_l$ under the corestriction map $\Sel_{p^\infty}(A^t/K_l)\rightarrow \Sel_{p^\infty}(A^t/L^{(i)}_{l-n})$.
Note that $K_l$ is an extension of $L^{(i)}_{l-n}$ as $\Gamma_l\subset \Psi_i^{(l)}\Phi_i^{(n)}=\Gal(L/L^{(i)}_{l-n})$.
Thus, the restriction map sends $\xi'_l$ to an element $\theta_l\in \Sel_{div}(A^t/L)^{\Psi_i^{(l)}\Phi_i^{(n)}}$. Then by \eqref{e:tem01}
\begin{equation}\label{e:tem03} g_i\cdot\theta_l\in \Sel_{div}(A^t/L)^{\Gamma^{(n)}}. \end{equation}
By the control theorem \cite[Theorem 4]{tan10a} the cokernel of the restriction map
$$\xymatrix{\Sel_{p^\infty}(A^t/K_n)_{div}\ar[r]^{res_n} & \Sel_{div}(A^t/L)^{\Gamma^{(n)}}}$$
is finite: we denote by $p^e$ its order. Choose $l\geq n+e$ and choose $\xi_n$ to be the image of $\xi_l$ under the corestriction map $\Sel_{p^\infty}(A^t/K_l)\rightarrow \Sel_{p^\infty}(A^t/K_n)$. Then \eqref{e:tem03} implies that $g_i\cdot \theta_n=p^eg_i\cdot\theta_l$, which shows that $g_i\cdot \theta_n=res_n(\theta_n')$, for some $\theta_n'\in \Sel_{div}(A^t/K_n)$. Then $\pi_n(g_i\cdot \xi_n-\theta'_n)=g_i\cdot a_n$. Since $g_i\cdot \xi_n-\theta'_n\in \fs_n^0$ which is annihilated by $\delta_j$, we have $\delta_j g_i\cdot a_n=0$ as desired.
\end{proof}
\end{subsection}

\begin{subsection}{Idempotents of the endomorphism rings}\label{ss:idempotents}
Let $\E$ denote the ring of endomorphisms of $A/K$ and write $\ZZ_p\,\E:=\ZZ_p\otimes_{\ZZ}\E$. For the rest of Section \ref{s:al}, we assume that there exists a non-trivial idempotent $e_1$ contained in the center of $\ZZ_p\,\E$. Set $e_2:=1-e_1$. Then we have the decomposition:
\begin{equation} \label{e:abstractdec} \ZZ_p\,\E=e_1\ZZ_p\,\E\times e_2\ZZ_p\,\E. \end{equation}
Let $\E^t$ denote the endomorphism ring of $A^t/K$. Since the assignment $\psi\mapsto \psi^t$ sending an endomorphism $\psi\in\E$ to its dual endomorphism can be uniquely extended to a $\ZZ_p$-algebra anti-isomorphism $\cdot^t\colon\ZZ_p\,\E\rightarrow \ZZ_p\,\E^t$, we find idempotents $e_1^t,e_2^t$ and the analogue of \eqref{e:abstractdec}.
If $\E$ and $\E^t$ act respectively on $p$-primary abelian groups $M$ and $N$, then these actions can be extended to those of $\ZZ_p\,\E$ and $\ZZ_p\,\E^t$. We have the following $\ZZ_p$-version of Proposition \ref{p:1}.

\begin{lemma}\label{l:psha} For every $a\in\fa_n$, $b\in\fb_n$ and $\psi\in\ZZ_p\,\E$ we have
\begin{equation} \label{e:ctgeneralized} \langle a,\psi_*(b)\rangle_n=\langle \psi_*^t(a), b\rangle_n\,. \end{equation}
\end{lemma}

\begin{proof} First note that any $\psi\in\E$ can be obtained as a sum of two isogenies (e.g., because $k+\psi$ is an isogeny for some $k\in\ZZ$). Thus Proposition \ref{p:1} and linearity of the Cassels-Tate pairing imply that \eqref{e:ctgeneralized} holds for such $\psi$.

In the general case, since $\ZZ_p\,\E$ is the $p$-completion of $\E$, for each positive integer $m$ there exists $\varphi_m\in\E$
such that $\psi-\varphi_m\in p^m\ZZ_p\,\E$. Choose $m$ so that $p^ma=p^mb=0$. Then
$$\langle a,\psi_*(b)\rangle_n=\langle a,{\varphi_m}_*(b)\rangle_n=\langle {\varphi_m^t}_*(a), b\rangle_n=\langle \psi_*^t(a), b\rangle_n.$$
\end{proof}

If $M$ and $N$ are respectively $\ZZ_p\,\E$ and $\ZZ_p\,\E^t$ modules, write $M^{(i)}$ for $e_i\cdot M $ and $N^{(i)}$ for $e_i^t\cdot N$. Then \eqref{e:abstractdec} implies $M=M^{(1)}\oplus M^{(2)}$ and $N=N^{(1)}\oplus N^{(2)}$. In particular,
$$\fa=\fa^{(1)}\oplus \fa^{(2)} \text{ and }\fb=\fb^{(1)}\oplus \fb^{(2)}\,,$$
with $\fa^{(i)}=\varprojlim_{n}\fa^{(i)}_n$ and $\fb^{(i)}=\varprojlim_{n}\fb^{(i)}_n$.

\begin{corollary}\label{c:psha} For every $n$, we have a perfect duality between $\fa_n^{(1)}$, $\fb_n^{(1)}$ and one between $\fa_n^{(2)}$, $\fb_n^{(2)}$.\end{corollary}

\begin{proof}
We just need to check  $\langle \fa_n^{(1)},  \fb_n^{(2)}\rangle_n=\langle \fa_n^{(2)},  \fb_n^{(1)}\rangle_n=0$.
If $a\in \fa_n^{(1)}$ and $b\in \fb_n^{(2)}$, then
$\langle a,  b\rangle_n=\langle e_1^t\cdot a,  e_2\cdot b \rangle_n=\langle  a,  e_1e_2\cdot b\rangle_n=\langle  a, 0\rangle_n=0$.
\end{proof}

Then $\fA^{(i)}:=\{\fa_n^{(i)},\fb_n^{(i)}, \langle\,,\,\rangle,\fr_m^n,\fk_m^n\}$ satisfies conditions ($\Gamma$-1) to ($\Gamma$-4) and Theorem \ref{t:al} implies the following.

\begin{proposition}\label{p:idemsha} If $X_p(A/L)$ is a torsion $\La$-module and $\fA$ satisfies the hypotheses of either {\em Theorem \ref{t:al}}{\em (1)} or {\em \ref{t:al}(3)}, then
$$[\fa^{(1)}]=[\fb^{(1)}]^{\sharp}\,\;\;\text{and}\;\; \,[\fa^{(2)}]=[\fb^{(2)}]^{\sharp}.$$
\end{proposition}

\noindent This holds because the conditions of being pseudo-controlled, killed by a ``non-simple'' element or enjoying property ({\bf T}) are all preserved by the operators $e_i$. However, a pseudo-isomorphism between $\fA$ and another $\Gamma$-system need not be compatible with the $\ZZ_p\,\E$-action and so Theorem \ref{t:al}(2) cannot be used here.\\

Applying $e_i$ to the exact sequence \eqref{e:axy} we get
\begin{equation}\label{e:axy{(1)}} 0\lr \fa^{(i)}\lr X_p(A/L)^{(i)}\lr Y_p(A/L)^{(i)}\lr 0 \,. \end{equation}
Unfortunately in general we are unable to compare either $X_p(A/L)^{(i)}$ with $X_p(A^t/L)^{(i)}$ or $Y_p(A/L)^{(i)}$ with $Y_p(A^t/L)^{(i)}$, so we fail short of an $e_i$-version of Theorem \ref{t:xaat}. However, we can get some partial result, which will be explained in \S\ref{sss:idemZ}.

\subsubsection{}\label{sss:padhpairidem}
For each finite extension $F/K$ write $M(A/F):=\ZZ_p\otimes A(F)$ and recall the $p$-adic height pairing $h_{A/F}$ established in Lemma \ref{l:pmw}. The action of $\E$ on $A(F)$ extends to that $\ZZ_p\,\E$ on $M(A/F)$ and the following results are proven by the same reasoning as in the proofs of Lemma \ref{l:psha} and Corollary \ref{c:psha}.

\begin{lemma}\label{l:pmw2}  For every $x\in M(A/F)$, $y\in M(A^t/F)$ and $\psi\in\ZZ_p\,\E$ we have
\begin{equation}\label{e:p2mw}
h_{A/F}(\psi(x),y)=h_{A/F}(x,\psi^t(y))\, .
\end{equation}
\end{lemma}

\begin{corollary} \label{c:hdualidemp}
The height pairing gives rise to perfect dualities between $\QQ_p\otimes_{\ZZ_p}M(A/F)^{(i)}$ and $\QQ_p\otimes_{\ZZ_p}M(A^t/F)^{(i)}$, for $i=1,2$.
\end{corollary}

\subsubsection{} \label{sss:idemZ}
Via \eqref{e:selmersha} we view $\M(A/F):=\QQ_p/\ZZ_p\otimes_{\ZZ}\,A(F)$ as a subgroup of $\Sel_{p^\infty}(A/F)_{div}$ and also
$$\M_\infty(A):=\varinjlim_n \QQ_p/\ZZ_p\otimes_{\ZZ}A(K_n)$$
as a subgroup of $\Sel_{div}(A/L)$. Denote
$$Z_p(A/L):=\M_\infty(A)^\vee.$$
If all Tate-Shafarevich groups are finite, then $Z_p(A/L)=Y_p(A/L)$. In any case, there is a surjection $Y_p(A/L)\twoheadrightarrow Z_p(A/L)$ and it is fair to say that $Z_p(A/L)$ carries a big chunk of the information on $Y_p(A/L)$.

The action of $\ZZ_p\,\E$ on $\M_\infty(A)$ extends to its dual as $(e\cdot\varphi)(x):=\varphi(ex)$. Thus we can write
\begin{equation}\label{e:z12} Z_p(A/L)=Z_p(A/L)^{(1)}\oplus Z_p(A/L)^{(2)}. \end{equation}
where $Z_p(A/L)^{(i)}=(\M_\infty(A)^{(i)})^\vee$.

\begin{proposition}\label{p:idemmw} If $X_p(A/L)$ is a torsion $\La$-module, $\Sha_{p^\infty}(A/K_n)$ is finite for every $n$ and $L/K$ is ramified only at good ordinary places, then
$$[Z_p(A/L)^{(1)}]=[Z_p(A^t/L)^{(1)}]^\sharp\;\; \text{and}\;\;  [Z_p(A/L)^{(2)}]=[Z_p(A^t/L)^{(2)}]^\sharp.$$
\end{proposition}

\begin{proof} Fix $i\in\{1,2\}$. By Theorem \ref{t:flat}, we write
$$[Z_p(A/L)^{(i)}]=\bigoplus_{\nu=1}^m (\La/(f_{\nu}))^{r_{\nu}},\;\;\;[Z_p(A^t/L)^{(i)}]=\bigoplus_{\nu=1}^m (\La/(f_{\nu}))^{s_{\nu}},$$
where $f_1,...,f_m$ are coprime simple elements and $r_{\nu}$, $s_{\nu}$ are nonnegative integers.
We need to show that $r_{\nu}=s_\nu$ for every $\nu$, since $\La/(f_\nu)=(\La/(f_\nu))^\sharp$.

Let $P_1$ denote the quotient $Z_p(A/L)^{(i)}/[Z_p(A/L)^{(i)}]$ and $P_2$ the analogue for $A^t$. Since $P_1$, $P_2$ are pseudo-null $\La$-modules, there are $\eta_1,\eta_2\in\La$ coprime to $f:=f_1 \cdots f_m$ such that $\eta_jP_j=0$. Then $f_\nu$ is coprime to $\eta_1\eta_2ff_\nu^{-1}$. We choose $\omega\in \Gamma^\vee$ so that $\omega(f_\nu)=0$ and $\omega(\eta_1\eta_2ff_\nu^{-1}) \not=0$. Let $E$ be a finite extension of $\QQ_p$ containing the values of $\omega$. We see $E$ as an $E\La$-module via the ring epimorphism $E\La\rightarrow E$ induced by $\omega$. The exact sequence
$$0=\Tor^1_{E\La}(E,EP_1)\lr E\otimes_{E\La} E[Z_p(A/L)^{(i)}] \lr  E\otimes_{E\La} EZ_p(A/L)^{(i)} \lr E\otimes_{E\La} EP_1= 0$$
yields
$$r_{\nu}=\dim_{E} (E\otimes_{E\La} EZ_p(A/L)^{(i)})\,.$$
Let $\Gamma^\omega\subset\Gamma$ denote the kernel of $\omega$ and write $W_\omega(A)$ for the coinvariants $EZ_p(A/L)_{\Gamma^\omega}^{(i)}$. The isomorphisms
$$E\otimes_{E\La} EZ_p(A/L)^{(i)}\simeq E\otimes_{E\La} EZ_p(A/L)_{\Gamma^\omega}^{(i)}\simeq (EZ_p(A/L)_{\Gamma^\omega}^{(i)})^{(\omega)}$$
show that $r_{\nu}=\dim_E W_\omega(A)^{(\omega)}$. A similar argument proves $s_{\nu}=\dim_E W_\omega(A^t)^{(\omega)}$.

Write $K_\omega:=L^{\Gamma^\omega}$ and $\Gamma_{\omega}:=\Gal(K_\omega/K)$. By our assumption on Tate-Shafarevich groups, also $\Sha_{p^\infty}(A/K_\omega)$ is finite: then the control theorem \cite[Theorem 4]{tan10a} implies that the restriction map $\M(A/K_\omega)\rightarrow \M_{\infty}(A)^{\Gamma^\omega}$ has finite kernel and cokernel. Thus we find
$$\rank_{\ZZ_p}\big(\M(A/K_\omega)^{(i)}\big)^\vee=\rank_{\ZZ_p}\big((\M_{\infty}(A)^{\Gamma^\omega})^{(i)}\big)^\vee=\rank_{\ZZ_p} Z_p(A/L)_{\Gamma^\omega}^{(i)}.$$
Together with the obvious equality $\rank_{\ZZ_p} M(A/K_\omega)=\rank_{\ZZ_p}\M(A/K_\omega)^\vee$ and Corollary \ref{c:hdualidemp}, this yields
$$\rank_{\ZZ_p} Z_p(A/L)_{\Gamma^\omega}^{(i)}=\rank_{\ZZ_p} Z_p(A^t/L)_{\Gamma^\omega}^{(i)}.$$
Similarly, for the character $\varpi=\omega^p$, we have
$$\rank_{\ZZ_p} Z_p(A/L)_{\Gamma^{\varpi}}^{(i)}=\rank_{\ZZ_p} Z_p(A^t/L)_{\Gamma^{\varpi}}^{(i)}.$$
As in \S\ref{ss:qorbit}, let $[\omega]$ denote the $\Gal({\bar{\QQ}}_p/\QQ_p)$-orbit of $\omega$. By $\Gamma_\omega^\vee=[\omega]\sqcup\Gamma_{\omega^p}^\vee$ we get the exact sequence of $E[\Gamma_\omega]$-modules:
$$\xymatrix{0 \ar[r] & \prod_{\chi\in [\omega]} (W_\omega(A))^{(\chi)} \ar[r] & W_\omega(A) \ar[r] & W_{\omega^p}(A) \ar[r] & 0}.$$
Since for all $\chi\in [\omega]$ the eigenspaces $(W_{\omega}(A))^{(\chi)}$ have the same dimension over $E$, the two equalities and the exact sequence above imply
$$\dim_E (W_{\omega}(A))^{(\omega)}=\dim_E (W_{\omega}(A^t))^{(\omega)}\,,$$
which completes the proof.\end{proof}

\begin{mytheorem}\label{t:idemx} If $X_p(A/L)$ is a torsion $\La$-module, $\Sha_{p^\infty}(A/K_n)$ is finite for every $n$ and $L/K$ is ramified only at good ordinary places, then
$$\chi(X_p(A/L))=\chi(X_p(A^t/L)^{(1)})^\sharp\cdot \chi(X_p(A/L)^{(2)})=\chi(X_p(A/L)^{(1)})\cdot \chi(X_p(A^t/L)^{(2)})^\sharp\,.$$
\end{mytheorem}

\begin{proof} Just use the exact sequence \eqref{e:axy{(1)}} (together with its $A^t$-analogue with $\fb^{(i)}$) and Propositions \ref{p:shatate}, \ref{p:gdordpscontrl}, \ref{p:idemsha} and \ref{p:idemmw}, remembering that $Z_p(A/L)=Y_p(A/L)$ by the hypothesis on Tate-Shafarevich groups, to get
\begin{equation} \label{e:idemx} \chi(X_p(A^t/L)^{(i)})^\sharp = \chi(X_p(A/L)^{(i)})\,.\end{equation}
\end{proof}

\end{subsection}
\end{section}
\end{part}


\begin{part}{The case of a constant ordinary abelian variety} \label{part:constantA}

From now on $K$ will be a function field with constant field $\FF$ of cardinality $q$ a power of $p$. In this part we assume that $A/K$ is an ordinary abelian variety of dimension $g$ defined over the constant field $\FF$ of $K$.

\begin{section}{The Frobenius Action} \label{s:frobact}
Let $K^a$ be a fixed algebraic closure of $K$. For each positive integer $m$, let $K^{(1/p^m)}\subset K^a$ denote the unique purely inseparable intermediate extension of degree $p^m$ over $K$, and let $K^{(1/p^{\infty})}:=\bigcup_{m=1}^{\infty} K^{(1/p^m)}$. Let ${\bar K}^{(1/p^m)}\subset K^a$ denote the separable closure of $K^{(1/p^m)}$.

Also, let $\Frob_{p^m}\colon x\mapsto x^{p^m}$ denote the Frobenius substitution: it induces a morphism $Frob_{p^m}\colon K^{(1/p^m)}\longrightarrow K$. On Galois groups, $Frob_{p^m}$ gives rise to an identification
$$\Gal({\bar K}/K)=\Gal({\bar K}^{(1/p^m)}/K^{(1/p^m)}).$$

\begin{subsection}{Absolute and relative Frobenius}\label{su:fm}
Let $\cG$ be a commutative finite group scheme over $K$ and denote $\cG_0$ its connected component and $\cG^{\mathrm{\acute et}}$ its maximal \'etale factor. Then we have the connected-\'etale sequence
\begin{equation} \label{e:connetal} 0 \lr \cG_0\, \lr \cG\, \lr \cG^{\mathrm{\acute et}}\, \lr 0. \end{equation}
We shall view $\cG$, $\cG_0$ and $\cG^{\mathrm{\acute et}}$ as sheaves on the flat topology of $K$, so that we can consider their cohomology groups. The base change $K\lr K^{(1/p^m)}$ induces the restriction map:
$$\res_m\colon\coh^1_{\mathrm{fl}}(K,\cG)\,\lr \coh^1_{\mathrm{fl}}(K^{(1/p^m)},\cG).$$

\begin{lemma}\label{l:image}
The image $\res_m(\coh^1_{\mathrm{fl}}(K,\cG))$ injects into $\coh^1_{\mathrm{fl}}(K^{(1/p^m)},\cG^{\mathrm{\acute et}})$ for any $m$ such that $p^m\geq\dim_K K[\cG_0]$.
\end{lemma}

\begin{proof}
Suppose $\xi\in \coh^1_{\mathrm{fl}}(K,\cG_0)$. Since $\coh^1_{\mathrm{fl}}(K,\cG_0)$ can be computed as the \v{C}ech cohomology (see \cite[Proposition III.6.1]{mil86a})
$$ \breve{\coh}^1(K,\cG_0):=\breve{\coh}^1(K^a/K,\cG_0),$$
by \cite[Corollary III.4.7]{mil80} the class $\xi$ corresponds to a principal homogeneous space $\cP$ for $\cG_0$ such that $\cP(L)$, the set of $L$-points of $\cP$, is non-empty for some finite extension $L/K$. The field $L$ can be taken as the residue field of the ring $K[\cP]$ modulo a maximal ideal.
As $\cP(K^a)=\cG_0(K^a)=\{id\}$, the ring $K[\cP]$ must contain exactly one maximal ideal and the residue field must be a purely inseparable extension over $K$. This shows that $L$ can be taken as $K^{(1/p^n)}$ for some $n$, and hence $\cG_0\simeq\cP$ over $K^{(1/p^n)}$. Also, since $K[\cP]$ is a finite algebra over $K$ with dimension the same as $\dim_K K[\cG_0]:=d$, the degree of $L/K$ is at most $d$. Therefore, $\coh^1_{\mathrm{fl}}(K^{(1/p^m)},\cG_0)=0$ for $p^m\geq d$.

To conclude, it is enough to take the cohomology of the sequence \eqref{e:connetal}.
\end{proof}

\begin{remark} {\em If $\cG$ is defined over a perfect field then \eqref{e:connetal} splits and the decomposition
$$\cG=\cG_0\times \cG^{\mathrm{\acute et}}$$
holds (see \cite[\S6.8]{wat79}). In particular we have
\begin{equation}\label{e:dec}
\coh^1_{\mathrm{fl}}(K,\cG)=\coh^1_{\mathrm{fl}}(K,\cG_0)\times \coh^1_{\mathrm{fl}}(K,\cG^{\mathrm{\acute et}}).
\end{equation} }
\end{remark}

\subsubsection{Application to abelian varieties} \label{sss:applabvar} Assume that $B/K$ is an abelian variety. Again, we view $B$ as a sheaf on the flat topology of $K$. Let $B_{p^n}^{\mathrm{\acute et}}$ denote the maximal \'etale factor of $B_{p^n}$. Lemma \ref{l:image} immediately yields the following.

\begin{corollary}\label{c:image2}
For each $n$, there exists some $m$ so that $\res_m(\coh^1_{\mathrm{fl}}(K,B_{p^n}))$ injects into $\coh^1_{\mathrm{fl}}(K^{(1/p^m)},B_{p^n}^{\mathrm{\acute et}})$.
\end{corollary}

By definition, the abelian variety $B^{(p^m)}:=B\times_K K$ is the base change over the absolute Frobenius $\Frob_{p^m}\colon K\rightarrow K$; as a sheaf on the flat topology of $K$, $B^{(p^m)}/K$ is identified with the inverse image $\Frob_{p^m}^*B$ (\cite[Remark II.3.1.(d)]{mil80}). As we noticed before, the absolute Frobenius factors through
\[ \begin{CD} K @>>> K^{(1/p^m)} @>{Frob_{p^m}}>> K. \end{CD} \]
Therefore we get a commutative diagram
\begin{equation} \begin{CD} \label{d:frob} \coh^1_{\mathrm{fl}}(K,B_{p^n})  @>\res_m>> \coh^1_{\mathrm{fl}}( K^{(1/p^m)},B_{p^n}^{\mathrm{\acute et}}) \\
@|   @VV{Frob_{p^m}^*}V  \\
\coh^1_{\mathrm{fl}}(K,B_{p^n}) @>{\Frob_{p^m}^*}>> \coh^1_{\mathrm{fl}}(K,(B^{\mathrm{\acute et}}_{p^n})^{(p^m)}).
\end{CD} \end{equation}

Let $\tF^{(m)}_B\colon B\rightarrow B^{(p^m)}$ denote the relative Frobenius homomorphism: it is the map induced by the commutative diagram
\[\begin{CD} B  @>{\Frob_{p^m}}>> B \\
@VVV    @VVV  \\
K @>{\Frob_{p^m}}>> K. \end{CD}\]

\begin{lemma} \label{l:absrelfrob} As maps of sheaves on the flat topology of $K$, $(\tF^{(m)}_B)_*$ and $\Frob_{p^m}^*$ coincide.
\end{lemma}

\begin{proof} Denote by $\pi$ the projection $B^{(p^m)}=B\times_K K\rightarrow B$. Let $s$ be a section of $B/K$: by definition $(\tF^{(m)}_B)_*(s)=\tF^{(m)}_B\circ s$, while $\Frob_{p^m}^*(s)$ is the unique section such that $\pi\circ\Frob_{p^m}^*(s)=s\circ\Frob_{p^m}$. Since $\pi\circ\tF^{(m)}_B=\Frob_{p^m}$, the claim is reduced to check that $\Frob_{p^m}\circ s=s\circ\Frob_{p^m}$ and this follows from the trivial observation that the map $x\mapsto x^{p^m}$ commutes with every homomorphism of $\FF_p$-algebras.
\end{proof}

Together with Corollary \ref{c:image2} and diagram \eqref{d:frob}, Lemma \ref{l:absrelfrob} implies the following corollary in an obvious way.

\begin{corollary}\label{c:image1}
For each $n$, there exists some $m$ so that $(\tF^{(m)}_B)_*(\coh^1_{\mathrm{fl}}(K,B_{p^n}))$, and hence $(\tF^{(m)}_B)_*(\Sel_{p^n}(B/K))$, is contained in $\coh^1_{\mathrm{fl}}(K,(B_{p^n}^{\mathrm{\acute et}})^{(p^m)})$.
\end{corollary}

\end{subsection}


\begin{subsection}{The Frobenius decomposition}\label{su:fa} In the case of our constant abelian variety $A$, the map $\tF_{A,\,q}:=\tF_A^{(a)}$ (with $a$ such that $q=|\FF|=p^a$) is an endomorphism and it satisfies $\tF_A^{(am)}=\tF_{A,\,q}^m\,$. In the following we shall generally shorten $\tF_{A,\,q}$ to $\tF_q\,$.

Let $\E=\E_A$ denote the ring of endomorphisms of $A/K$ and write $\ZZ_p\,\E:=\ZZ_p\otimes_{\ZZ}\E$. We have
$$[q^m]=\tV_q^{(m)}\circ \tF_q^m$$
for some $\tV_q^{(m)}\in \E$. The assumption that $A/\FF$ is ordinary implies that $\tV_q^{(m)}$ is separable in the sense that its kernel (considered
as a group scheme over ${\bar \FF}_q$) is the maximal \'etale subgroup of $A_{q^m}$. We have the following silly lemma.

\begin{lemma}\label{l:silly}
Suppose $B, C$ are abelian varieties defined over $\FF$ and $\Phi\in\Hom_\FF(C,B)$. Then
$$\tF_{B,\,q}^{(m)}\circ\Phi=\Phi\circ \tF_{C,\,q}^{(m)}\,,$$
$$\tV_{B,\,q}^{(m)}\circ\Phi=\Phi\circ \tV_{C,\,q}^{(m)}$$
for all $m\geq 1$.
\end{lemma}

\begin{proof}
Since $\FF$ is perfect of cardinality $q$, $B$ and $C$ are identified with $B^{(q^m)}, C^{(q^m)}$ and both $\tF_{\bullet,\,q}^{(m)}$ with the absolute Frobenius $\Frob_{q^m}$ (\cite[III \S0]{mil86a}). We already observed in the proof of Lemma \ref{l:absrelfrob} that $\Frob_{q^m}$ commutes with every algebraic map over $\FF$.
Then $\tF_{B,\,q}^{(m)}\circ (\tV_{B,\,q}^{(m)}\circ\Phi-\Phi\circ \tV_{C,\,q}^{(m)})=0$ holds, since
$$\tF_{B,\,q}^{(m)}\circ  \tV_{B,\,q}^{(m)}\circ\Phi=[q^m]\circ\Phi=\Phi\circ [q^m]=\Phi\circ\tF_{C,\,q}^{(m)}\circ\tV_{C,\,q}^{(m)}=\tF_{B,\,q}^{(m)}\circ\Phi\circ \tV_{C,\,q}^{(m)}.$$
The second statement follows from the fact that $\tF_{B,\,q}^{(m)}$ has finite kernel.
\end{proof}

The lemma implies
$$[q^m]=[q]\circ\dots\circ [q]=\tV_q^m\circ \tF_q^m,$$
where $\tV_q=\tV_q^{(1)}$ (i.e., $\tV_q^{(m)}=\tV_q^m$).

\begin{lemma}\label{l:unit}
The operator $\tV_q^m+\tF_q^m$ is invertible in $\ZZ_p\,\E$ for all $m\geq1$.
\end{lemma}

\begin{proof}
It is enough to show that $\C:=\Ker(\tV_q^m+\tF_q^m)\cap \Ker([p])$ is trivial: for then $\Ker(\tV_q^m+\tF_q^m)\subset A_n$ for some $n$ coprime with $p$, implying that $[n]$ (which is invertible in $\ZZ_p\,\E$) factors through $\tV_q^m+\tF_q^m$.

Since $\tV_q^m\circ(\tV_q^m+\tF_q^m)=\tV_q^{2m}+[q^m]$, we have $\C\subset \Ker(\tV_q^{2m})$. This implies that $\C$ is an \'etale subgroup scheme of $A_p$, and hence $\C\subset \Ker(\tV_q)\subset \Ker(\tV_q^m)$. Then we shall have $\C\subset \Ker(\tF_q^m)$, too. But since $\C$ is \'etale contained in $\Ker(\tF_q^m)$, which is the connected component of $A_{q^m}$, it must be trivial.
\end{proof}

\begin{corollary} \label{c:Esplits}
We have the decomposition of $\ZZ_p$-algebra:
$$\ZZ_p\,\E=\F\times \V\,,$$
where
$$\F:=\bigcap_{m\geq0}\tF_q^m(\ZZ_p\,\E) \hspace{17pt}{ and }\hspace{17pt} \V:=\bigcap_{m\geq0}\tV_q^m(\ZZ_p\,\E) \,. $$
\end{corollary}

\begin{proof} The ring $\ZZ_p\,\E$ has a natural topology as the $p$-adic completion of $\E$. Since $\ZZ_ p\,\E$ is compact, it is possible to find subsequences $\tF_q^{m_k}$, $\tV_q^{m_k}$ converging respectively to $\tF_\infty\in \F$, $\tV_\infty\in \V$. Lemma \ref{l:silly} shows that $\tF_\infty, \tV_\infty$ are both in the center of $\ZZ_p\, \E$ and that $\F,\V$ are two-sided ideals. By Lemma \ref{l:unit} and the fact that $(\ZZ_p\,\E)^{\times}$ is closed,
$$\tF_\infty+\tV_\infty=\lim_{m_k\rightarrow\infty}(\tF_q^{m_k}+\tV_q^{m_k})$$
is a unit. To complete the proof, we only need to note that $\tF_\infty\tV_\infty=0$ as $\tF_\infty \tV_\infty\in \tF_q^m\tV_q^m(\ZZ_p\,\E)=q^m\ZZ_p\,\E$ for all $m$.
\end{proof}

By Corollary \ref{c:Esplits}, any $\ZZ_p\,\E$-module $M$ decomposes as $\F M\oplus \V M$. In many cases, a better grasp of the two components can be obtained by the following lemma (and its obvious $\V$-analogue).

\begin{lemma}\label{l:decomposition} Let $W$ be a $\ZZ_p\,\E$-module and $W[q^n]$ its $q^n$-torsion submodule. Then
$$\F W[q^n]=\tF_q^n(W[q^n]).$$
In particular, if $W$ is a $p$-primary abelian group then $\F W=\bigcap_n \tF_q^nW$.
\end{lemma}

\begin{proof}
By Lemma \ref{l:unit}, for any $m\geq n$ and any $r\geq 1$ we have
$$\tF_q^{rm}\ZZ_p\,\E+q^n\ZZ_p\,\E=(\tF_q^m+\tV_q^m)(\tF_q^{rm}\ZZ_p\,\E+q^n\ZZ_p\,\E)=\tF_q^{(r+1)m}\ZZ_p\,\E+q^n\ZZ_p\,\E\,.$$
Hence the image of $\F$ in $\ZZ_p\,\E/q^n\ZZ_p\,\E$ is $\tF_q^n(\ZZ_p\,\E/q^n\ZZ_p\,\E)$.
\end{proof}

Before applying the theory of \S\ref{ss:idempotents}, we recall that $(\tF_{A,\,q})^t=\tV_{A^t,\,q}$ and $(\tV_{A,\,q})^t=\tF_{A^t,\,q}\,$. This should be well-known; a proof can be found in \cite[Proposition 7.34]{gm13}. As a consequence, we get
\begin{equation}\label{e:fatvt} (\F_A)^t=\V_{A^t}\, \text{ and } (\V_A)^t=\F_{A^t}\,\,. \end{equation}

\begin{proposition}\label{p:sfft}
If $X_p(A/L)$ is a torsion $\La$-module, then
$$\chi(\F_A X_p(A/L))= \chi(\V_{A^t} X_p(A^t/L))^\sharp\,\; \text{and}\,\; \chi(\V_A X_p(A/L))= \chi(\F_{A^t} X_p(A^t/L))^\sharp.$$
\end{proposition}

\begin{proof} By \eqref{e:fatvt} this is just a reformulation of \eqref{e:idemx}. Since $A$ is defined over $\FF$ and ordinary, it has good ordinary reduction at every place of $K$. By \cite[Theorem 3]{Mi68} it is known that $\Sha_{p^\infty}(A/K_n)$ is finite for every $n$. Therefore, the conditions of Theorem \ref{t:idemx} are satisfied.
\end{proof}

\begin{lemma}\label{l:sst}
If one of the following modules
$$X_p(A/L),\; \F_A X_p(A/L),\;  \V_A X_p(A/L),$$
and
$$X_p(A^t/L),\; \F_{A^t} X_p(A^t/L),\;  \V_{A^t} X_p(A^t/L),$$
is torsion over $\La$, then all of them are torsion over $\La$.
\end{lemma}

\begin{proof}
Applying any sequence of isogenies
$$A\longrightarrow A^t \longrightarrow A \longrightarrow A^t,$$
so that the composition of the last (resp. first) two arrows equals some $[n]$ (resp. $[n']$), we see that if any item of the upper (resp. lower) row is annihilated by a non-zero $\xi\in\La$, then the corresponding item in the lower (resp. upper) row is annihilated by $n'\xi$ (resp. $n\xi$). In view of the decomposition
\begin{equation} \label{e:xfv} X_p(A/L)=\F_A X_p(A/L)\oplus \V_A X_p(A/L)\,, \end{equation}
we only need to show that $\cH_AX_p(A/L)$ is torsion if and only if so is $\K_{A^t}X_p(A^t/L)$, with $\{\cH,\K\}=\{\F,\V\}$.

Suppose $\cH_AX_p(A/L)$ is annihilated by $\xi\in\La$. Then $\xi^\sharp\cdot \cH_A\Sel_{p^\infty}(A/L)=0$. Let $\gamma\in \Gamma$ be a non-trivial element: as in the proof of Lemma \ref{l:a00}, one finds a polynomial $g_\gamma$ such that $g_{\gamma}(\gamma)\cdot A(L)[p^\infty]=0$. Hence $\delta:=g_{\gamma}(\gamma)$ annihilates $\coh^1(\Gamma^{(n)},A_{p^\infty}(L))$ for every $n$. Then $\delta\xi^\sharp\cdot \cH_A\Sel_{p^\infty}(A/K_n)=0$ for every $n$ as we have
$$\xi^\sharp\cdot \cH_A\Sel_{p^\infty}(A/K_n)\subset \Ker \big(\Sel_{p^{\infty}}(A/K_n)\lr\Sel_{p^\infty}(A/L)\big)= \coh^1(\Gamma^{(n)},A_{p^\infty}(L)).$$
In particular, $\delta\xi^\sharp$ annihilates both $\cH_A\fb_n$and  $\cH_A\Sel_{p^\infty}(A/K_n)_{div}=\cH_A\cdot(\QQ_p/\ZZ_p\otimes A(K_n))$. This implies that $\delta^\sharp\xi$ annihilates both $\K_{A^t}\Sel_{p^\infty}(A^t/K_n)_{div}$ and $\K_{A^t}\fa_n$, by Corollaries \ref{c:psha} and \ref{c:hdualidemp}. Hence, $\delta^\sharp\xi\cdot \K_{A^t}\Sel_{p^\infty}(A^t/K_n)=0$ as well. Therefore, $\delta^\sharp\xi\cdot \K_{A^t}\Sel_{p^\infty}(A^t/L)=0$ and by duality $\delta\xi^\sharp\cdot \K_{A^t}X_{p^{\infty}}(A^t/L)=0$ as desired.
\end{proof}

\begin{proposition}\label{p:xffsharp}
If $A$ is an ordinary abelian variety over $\FF$, then
$$\chi(X_p(A/L))=\chi(\F X_p(A/L))\cdot \chi(\F_{A^t} X_p(A^t/L))^\sharp.$$
\end{proposition}

\begin{proof} If $\F X_p(A/L)$ is torsion, then the equality follows from Lemma \ref{l:sst}, Proposition \ref{p:sfft} and \eqref{e:xfv}; otherwise, we get $0=0$. \end{proof}

\end{subsection}


\begin{subsection}{Selmer groups and class groups} \label{su:sgcg}

In this section we reduce the computation of $\F X_p(A/L)$ to that of Galois cohomology groups.

Lemma \ref{l:decomposition} and Corollary \ref{c:image1} imply that $\F\cdot\Sel_{p^n}(A/F)$ is contained in $\coh^1_{\mathrm{fl}}(F,A_{p^n}^{\mathrm{\acute et}})$. The latter can be viewed as Galois cohomology, by \cite[III, Theorem 3.9]{mil80}: since $A_{p^n}^{\mathrm{\acute et}}$ is \'etale, in the following we will not distinguish between the sheaf $A_{p^n}^{\mathrm{\acute et}}$ and the usual $p^n$-torsion subgroup $A_{p^n}^{\mathrm{\acute et}}(\bar K)=A(\bar K)[p^n]$. Also, when considering Galois cohomology, we shall shorten $A(\bar K)[p^n]$ to $A[p^n]$ and $\Gal(\bar F/F)$ to $G_F$. In particular, we shall write
$$A[p^\infty]:=A(\bar K)[p^\infty]=\bigcup A[p^n]\,.$$

Let $\tilde K_n$ denote the compositum $\FF(A[p^n])K$: it is an abelian extension over $K$, such that, for $n\geq 1$, $\Gal(\tilde K_{n+1}/\tilde K_n)$ is an abelian $p$-group annihilated by $p$.
Suppose $\tilde K_n\subset F$: then we can identify $\coh_{\mathrm{Gal}}^1(F,A[p^n])$ with $\Hom(G_F, A[p^n])$. Recall that $\varphi\in \Hom(G_F, A[p^n])$ is said to be unramified at a place $v$ if $\Ker(\varphi)$ contains the inertia group at $v$.

\begin{lemma} \label{l:unramified}
Suppose $F/K$ is a finite extension with $\tilde K_n\subset F$. Then an element $\varphi\in \Hom(G_F, A[p^n])=\coh_{\mathrm{Gal}}^1(F,A[p^n])$ is contained in $\Sel_{p^n}(A/F)$ if and only if $\varphi$ is unramified everywhere.
\end{lemma}

\noindent Here we use \eqref{e:dec} to identify
$$\coh_{\mathrm{Gal}}^1(F,A[p^n])=\coh_{\mathrm{\acute et}}^1(F,A_{p^n}^{\mathrm{\acute et}})=\coh_{\mathrm{fl}}^1(F,A_{p^n}^{\mathrm{\acute et}})$$
with a subgroup of $\coh_{\mathrm{fl}}^1(F,A_{p^n})$.

\begin{proof} For $v$ a place of $F$, let $F_v^{unr}$ denote an unramified closure of the completion $F_v$.
Consider the diagram (with exact lines)
\begin{equation} \label{e:diagrunr} \begin{CD}
0 @>>> \coh_{\mathrm{Gal}}^1(F_v^{unr}/F_v,A[p^n]) @>>> \coh_{\mathrm{Gal}}^1(F_v,A[p^n]) @>>> \coh_{\mathrm{Gal}}^1(F_v^{unr},A[p^n]) \\
&& @VVV @VVV @VVV \\
0 @>>> \coh_{\mathrm{Gal}}^1(F_v^{unr}/F_v,A) @>>> \coh_{\mathrm{Gal}}^1(F_v,A) @>>> \coh_{\mathrm{Gal}}^1(F_v^{unr},A)\,.
\end{CD}\end{equation}
The last vertical map is injective: for, if $\varphi$ is contained in the kernel, then there exists a point $Q\in A(\bar F_v)$ such that $\varphi(\sigma)=\sigma Q-Q $ for $\sigma$ in the inertia subgroup. As the latter acts trivially on the residue field, the reduction ${\overline{ \varphi(\sigma)}}$ is $0$. But since $A={\bar A}$ (the reduction of $A$ at $v$),
the reduction map induces an isomorphism $A[p^n]\simeq {\bar A}[p^n]$.\footnote{This holds more generally for abelian varieties with good ordinary reduction: see \cite[Corollary 2.1.3]{tan10a}.}
Then we get $\varphi(\sigma)=0$, and hence $\sigma\in \Ker(\varphi)$; that is, $\varphi$ is unramified. Also, we have $\coh_{\mathrm{Gal}}^1(F_v^{unr}/F_v,A)=0$ by \cite[I, Proposition 3.8]{mil86a} (a consequence of Lang's theorem and Hensel's lemma). Therefore, the kernel of the central vertical map in \eqref{e:diagrunr} is exactly $\coh_{\mathrm{Gal}}^1(F_v^{unr}/F_v,A[p^n])$.

Now it is enough to observe that the cocycles in $\coh_{\mathrm{Gal}}^1(F,A[p^n])$ whose restriction to $\coh_{\mathrm{Gal}}^1(F_v,A[p^n])$ is contained in $\coh_{\mathrm{Gal}}^1(F_v^{unr}/F_v,A[p^n])$ are, by definition, the ones unramified at $v$.
\end{proof}

Let $\fW_F$ denote the $p$-completion of the divisor class group of $F$. Note that the class group of $F$ consists of degree zero divisor classes: let $\fC_F$ be its Sylow $p$-subgroup. Then the degree map induces the exact sequence
\begin{equation}\label{e:weil} \begin{CD} 0 @>>> \fC_F @>>> \fW_F @>\deg>> \ZZ_p @>>> 0\,. \end{CD} \end{equation}

\begin{lemma}\label{l:s}
Suppose $F/K$ is a finite extension with $\tilde K_n\subset F$. Then
$$\F\cdot\Sel_{p^n}(A/F)=\Hom(\fW_F , A[p^n]).$$
\end{lemma}

\begin{proof} Let $\fV$ denote the group of homomorphisms $\Hom(\fW_F, A[p^n])$. Class field theory identifies $\fW_F$ with $\Gal(F^{unr,p}/F)$, where $F^{unr,p}$ is the maximal everywhere unramified abelian pro-$p$-extension of $F$. Thus, since $\tF_q^m$ induces an isomorphism on $A[p^n]$, Lemma \ref{l:unramified}, Lemma \ref{l:decomposition} and Corollary \ref{c:image1} imply, for $m\gg 0$,
$$\fV={\tF_q^m}_*(\fV)\subset {\tF_q^m}_*(\Sel_{p^n}(A/F))= \F\cdot\Sel_{p^n}(A/F)\subset\fV.$$
\end{proof}

\begin{subsubsection}{The extensions $L_{ar}$ and $\tilde L$} \label{ss:LL} Let $\FF(A[p^{\infty}])\subset\bar\FF$ be the field of definition of $A[p^{\infty}]$. We have $\Gal(\FF(A[p^{\infty}])/\FF)\simeq\ZZ_p\times H$ for some finite cyclic group $H$ of order prime to $p$. Let $L_{ar}$ denote the compositum $LK_{\infty}^{(p)}$, with $\Gamma_{ar}:=\Gal(L_{ar}/K)\simeq\ZZ_p^{d'}$, $d'\geq d$, and let
$${\tilde L}:=L\FF(A[p^{\infty}]),$$
with ${\tilde \Gamma}:=\Gal({\tilde L}/K)$. By a slight abuse of notation we can write
\begin{equation}\label{e:prod2}
{\tilde \Gamma}= \Gamma_{ar} \times H \,.
\end{equation}
By the main theorem of \cite{tan10a}, $X_p(A/{\tilde L})$ is a finitely generated $\Lambda(\Gamma_{ar})$-module, whence so is $\F X_p(A/\tilde L)$. Define
$$\fW_{\tilde L}:=\varprojlim_F\fW_F.$$

\begin{proposition} \label{p:selclassg} We have a canonical isomorphism of $\Gal({\tilde L}/K)$-modules:
$$\F X_p(A/\tilde L)\simeq \fW_{\tilde L}\otimes_{\ZZ_p}A[p^{\infty}]^\vee. $$
\end{proposition}

\begin{proof} Since $\fW_{\tilde L}$ is compact and $A[p^\infty]$ discrete, we have
$$\Hom_{cont}(\fW_{\tilde L} , A[p^n])=\varinjlim_F\Hom(\fW_F , A[p^n])\,,$$
where $F$ runs through all finite degree subextensions of ${\tilde L}/K$ (note that all homomorphisms from a finitely generated $\ZZ_p$-module such as $\fW_F$ into a finite group are continuous), and of course
$$\varinjlim_n\Hom_{cont}(\fW_{\tilde L} , A[p^n])=\Hom_{cont}(\fW_{\tilde L} , A[p^\infty])\,.$$
Hence we get, by Lemma \ref{l:s},
$$\F X_p(A/\tilde L)=\Hom_{cont}(\fW_{\tilde L} , A[p^\infty])^\vee\,.$$
Finally, the map
$$\fW_{\tilde L}\otimes_{\ZZ_p}A[p^\infty]^\vee \lr \Hom_{cont}(\fW_{\tilde L} , A[p^\infty])^\vee\,,$$
which sends $w\otimes Q^*$ to $\varphi\mapsto Q^*(\varphi(w))$, is an isomorphism because $A[p^\infty]^\vee$ is a finitely generated free $\ZZ_p$-module.
\end{proof}

Note that if $F$ is a finite extension of $K$ with constant field of cardinality $q^n$, then we have the commutative diagram  \[ \begin{CD}
\fW_F  @>{\deg_F}>> \ZZ_p \\
@VVV   @VV{n}V  \\
\fW_K @>{\deg_K}>> \ZZ_p,  \end{CD} \]
where the first down-arrow is the norm map and the second is the multiplication by $n$.
Therefore, since ${\tilde L}/K$ contains the constant $\ZZ_p$-extension, the exact sequence \eqref{e:weil} implies
\begin{equation}\label{e:fwfclim} \fW_{\tilde L}=\fC_{\tilde L}:=\varprojlim_F\fC_F. \end{equation}

\begin{lemma} \label{l:wtorsion} The group $\fW_{\tilde L}$ has no $p$-torsion. \end{lemma}

\begin{proof} For $F$ a finite extension of $K$ put $F_n:=\FF(A[p^n])F$. It suffices to show that for any $F$ the projective
limit (on $n$) of $\fC_{F_n}$ has no $p$-torsion. Let $\C/\FF_F$ be the curve associated with the function field $F$ (whose field of constants is $\FF_F$), so that
$$\fC_{F_n}=Jac\,\C[p^\infty]\cap Jac\,\C(\FF_{F_n}).$$
Let $t=(t_n)$ be an element in the $p$-torsion of $\varprojlim\fC_{F_n}$: then $t_n\in Jac\,\C[p]$ for all $n$ and hence there is some $m$ such that $t_n\in Jac\,\C(\FF_{F_m})$ for every $n$. But then for $n>m$ the norm map $\fC_{F_n}\rightarrow\fC_{F_{n-1}}$ acts on $t_n$ as multiplication by a power of $p$, hence $t=0$.
\end{proof}

\end{subsubsection}
\end{subsection}

\begin{subsection}{Frobenius Twist of Class Groups}\label{su:ftcg}
Let $\cO$ denote the ring of integers of some finite extension over $\QQ_p$. Recall the algebras $\La_{\cO}(\Gamma)$ and $\La_{\cO}(\tilde \Gamma)$ defined in Subsection \ref{su:iw}.

\subsubsection{The twist matrix} \label{ss:twist}
Since $A$ is defined over $\FF$, $A[p^{\infty}]$ is actually a $\Gal(\bar\FF/\FF)$-module. The latter group is topologically generated by the Frobenius substitution $\Fr_q$.
After the choice of an isomorphism $A[p^{\infty}]\simeq(\QQ_p/\ZZ_p)^g$, the action of $\Fr_q$ becomes that of a $g\times g$ matrix $\bo u$, called the twist matrix (see page 216 of \cite{mazur72} for a more detailed discussion). By \cite[Corollary 4.37]{mazur72}, the eigenvalues of the Frobenius endomorphism $\tF_q$ of $A$ are
$$\alpha_1,...,\alpha_g,\, \beta_1:=q/\alpha_1,...,\beta_g:=q/\alpha_g,$$
where $\alpha_1$,...,$\alpha_g$ are the eigenvalues of $\bo u$, counted with multiplicities. Assume that $\cO$ contains $\alpha_i$: then
\begin{equation} \label{e:alphabeta} \alpha_i\in\cO^{\times}\text{, }\beta_i\in q\cO\,. \end{equation}

\begin{lemma}\label{l:ses}
The twist matrix $\bo u$ is semi-simple.
\end{lemma}

\begin{proof} It is sufficient to assume that $A$ is simple. Suppose $\alpha_1,...,\alpha_k$, $k\leq g$, are all the distinct eigenvalues of $\bo u$ and let
$$\Phi=(\tF_q-\alpha_1)\cdots (\tF_q-\alpha_k)(\tF_q-\beta_1)\cdots (\tF_q-\beta_k).$$
Then $\Phi\in\ZZ[\tF_q]\subset \E$ is nilpotent and hence the simplicity assumption yields $\Phi=0$.

On $A[p^{\infty}]$, $\tF_q$ acts as $\Fr_q$. Assume that $\cO$ contains $\alpha_i$. Then, by \eqref{e:alphabeta},
$$\frac{1}{1-\beta_i\Fr_q^{-1}}:=1+\beta_i\Fr_q^{-1}+\dots+ \beta_i^n\Fr_q^{-n}+\dots$$
gives rise to an endomorphism of the group $\cO A[p^{\infty}]$. Therefore,
$$\Fr_q-\beta_i=\Fr_q(1-\beta_i\Fr_q^{-1})$$
is an automorphism of $\cO A[p^{\infty}]$. This shows that $(\Fr_q-\beta_1)\cdots (\Fr_q-\beta_k)$ is an automorphism of $A[p^{\infty}]$, and thus $\Phi=0$ implies that on $A[p^{\infty}]$
$$(\Fr_q-\alpha_1)\cdots (\Fr_q-\alpha_k)=0.$$
\end{proof}

\subsubsection{The Frobenius action on $A[p^\infty]^\vee$}\label{su:lambdai}
The action of $\Fr_q$ on $A[p^{\infty}]^\vee\simeq \ZZ_p^g$ is through the inverse matrix $\bo u^{-1}$. Suppose $\cO$ contains all $\alpha_1,...,\alpha_g$.
Let $\lambda_i$ denotes the continuous homomorphism $\tilde\Gamma \rightarrow \Gal(K(A[p^{\infty}])/K)\rightarrow \cO^{\times}$ such that
\begin{equation}\label{e:lambdai} \lambda_i(\Fr_q):=\alpha_i^{-1} \end{equation}
Let $\cO(\lambda_i)$ be the twist of $\cO$ by $\lambda_i$ defined in \S\ref{su:twistmodule} (see \eqref{e:twistbyphi} and the lines just after it). Denote $\psi_i:=\lambda_{i|H}$.

Note that Proposition \ref{p:selclassg} implies that $\fW_{\tilde L}$ is finitely generated over $\La(\Gamma_{ar})$.

\begin{proposition}\label{p:cs} Assume that $\fW_{\tilde L}$ is torsion over $\La(\Gamma_{ar})$ and that $\cO$ contains $\alpha_1,...,\alpha_g$ as well as $\psi(h)$ for all $\psi\in H^\vee$, $h\in H$. Then for every $\psi\in H^\vee$ there is a pseudo-isomorphism of $\Lambda_{\cO}(\Gamma_{ar})$-modules
$$\big(\cO\F X_p(A/\tilde L)\big)^{(\psi)}\sim \bigoplus_{i=1}^g \cO\fW_{\tilde L}^{(\psi\psi_i^{-1})}\otimes_{\cO} \cO(\lambda_i)\,.$$
Moreover, we have
$$\cO\F X_p(A/L_{ar}) \sim \bigoplus_{i=1}^g \cO\fW_{\tilde L}^{(\psi_i^{-1})}\otimes_{\cO} \cO(\lambda_i)\,.$$
\end{proposition}

\begin{proof} Since $\Fr_q$ topologically generates $\Gal(K(A[p^\infty])/K)$, Lemma \ref{l:ses} yields an exact sequence of $\cO[[\Gal(K(A[p^{\infty}])/K)]]$-modules
\begin{equation}\label{e:vax} 0\lr \bigoplus_{i=1}^g \cO(\lambda_i) \lr \cO A[p^{\infty}]^\vee \lr Q \lr 0   \end{equation}
with $p^m Q=0$ for some $m$. Lemma \ref{l:wtorsion} implies that $\fW_{\tilde L}$ is a flat $\ZZ_p$-module and is annihilated by some $f\in\La(\Gamma_{ar})$ coprime with $p\,$: by Lemma \ref{l:psn}, it follows that the module $\fW_{\tilde L}\otimes_{\ZZ_p} Q$ is pseudo-null over $\La_{\cO}(\Gamma_{ar})$. Hence \eqref{e:vax} implies that $\fW_{\tilde L}\otimes_{\ZZ_p}\cO A[p^{\infty}]^\vee$ is pseudo-isomorphic to $\bigoplus_{i=1}^g \fW_{\tilde L}\otimes_{\ZZ_p} \cO(\lambda_i)$. Since the group $H$ is of order prime to $p$, we have
$$\bigoplus_{i=1}^g \fW_{\tilde L}\otimes_{\ZZ_p} \cO(\lambda_i)=\bigoplus_{i=1}^g \bigoplus_{\varphi\in H^\vee}\cO\fW_{\tilde L}^{(\varphi)}\otimes_{\cO} \cO(\lambda_i)\,.$$
The first statement then follows from Proposition \ref{p:selclassg}.

To prove the second claim, note first that the restriction map provides the identification
$$\Sel_{p^\infty}(A/L_{ar})=\Sel_{p^{\infty}}(A/{\tilde L})^{H}$$
(since composing with corestriction amounts to multiply by the order of $H$, which is prime to $p$). Therefore, if $h$ is a generator of $H$ then
$$\cO\F X_p(A/L_{ar})=\cO\F X_p(A/\tilde L)/(1-h)\cO\F X_p(A/\tilde L)$$
which, according to the above argument, is pseudo-isomorphic to $\bigoplus \fW_{\tilde L}^{(\psi_i^{-1})}\otimes_{\ZZ_p} \cO(\lambda_i).$
\end{proof}

\end{subsection}
\end{section}


\begin{section}{The Iwasawa Main Conjecture for constant abelian varieties} \label{s:imcconstant}

In this section, we keep the notation of \S\ref{su:sgcg} and \S\ref{su:ftcg}.
In the following we shall identify the Galois group $\Gal(K(A[p^\infty])/K)$ with $\Gal(\FF(A[p^\infty])/\FF)$ and consider the Frobenius substitution $\Fr_q$ as an element of it.

\begin{subsection}{The Stickelberger element and divisor class groups}\label{su:sc}
Given a topological ring $R$, the ring of formal power series $R[[u]]$ is endowed with the topology coming from identification with an infinite product of copies of $R$. Then any continuous ring homomorphism $\phi\colon R\rightarrow R'$ extends to a continuous ring homomorphism $R[[u]]\rightarrow R'[[u]]$ by applying $\phi$ to each coefficient.

\subsubsection{The Stickelberger element} \label{ss:stickelb}
Let $M/K$ be an abelian Galois extension (of finite or infinite degree) unramified outside $S$; in case $S=\emptyset$, we furthermore stipulate that $M$ is a subfield of $K\bar\FF$, so that $\Gal(M/K)$ is generated by $\Fr_q$ (note that if $S=\emptyset$ then $L=\Kpinf$, $\tilde L=K(A[p^\infty])$, because non-arithmetic abelian totally unramified extensions of $K$ have Galois group a factor of the finite group $\fC_K$).

For any place $v$ outside $S$ let the symbol $[v]_M\in\Gal(M/K)$ denote the (arithmetic) Frobenius element at $v$. Also, set $\delta_S=0$ if $S\not=\emptyset$ and $\delta_S=1$ if $S=\emptyset$. Since there are only finitely many places with degree bounded by a given positive integer, we can express the infinite product
\begin{equation}\label{e:Thetaprod}  \Theta_{M,S}(u):=(1-\Fr_q \cdot u)^{\delta_S}\prod_{v\not\in S} (1-[v]_M\cdot u^{\deg(v)})^{-1} \end{equation}
as a formal power series in $\ZZ[\Gal(M/K)][[u]]\subset \ZZ_p[[\Gal(M/K)]][[u]]$.

Clearly $\Theta_{M,S}$ is well-behaved with respect to the maps induced by $\Gal(M/K)\rightarrow\Gal(M'/K)$, for any extension $M/M'$.
Furthermore, if $\omega\colon{\Gal(M/K)}\rightarrow \CC^{\times}$ is a continuous character then it extends to a ring homomorphism $\omega\colon\ZZ[ {\tilde\Gamma}][[u]]\rightarrow \CC [[q^{-s}]]$ by $u\mapsto q^{-s}$ and we have
\begin{equation} \label{e:L*(S,omega)} \omega\big(\Theta_{M,S}(u)\big)=L_S^*(\omega,s) \,, \end{equation}
where, letting $q_v:=q^{\deg(v)}$ for $v$ a place of $K$, the right-hand side is defined by
\begin{equation} \label{e:L(S,omega)} L^*_S(\omega,s):=(1-\omega(\Fr_q)\,q^{-s})^{\delta_S}\cdot \prod_{v\not\in S} (1-\omega([v]_M)\,q_v^{-s})^{-1}. \end{equation}

Recall that for any topological ring $R$ the Tate algebra $R\langle u\rangle$ consists of those power series in $R[[u]]$ whose coefficients tend to 0.

\begin{proposition} \label{p:Thetaconv}
The power series $\Theta_{M,S}(u)$ belongs to $\ZZ_p[[\Gal(M/K)]]\langle u\rangle$.
\end{proposition}

\begin{proof} We choose an auxiliary non-empty finite set $T$ of places of $K$ such that $S\cap T=\emptyset$ and define
$$\Theta_{M,S,T}(u):=\Theta_{M,S}(u) \cdot\prod_{v\in T} (1-q_v[v]_M\cdot u^{\deg(v)}).$$
Let $\omega\colon{\Gal(M/K)}\rightarrow \CC^{\times}$ be any continuous character and denote by $S_\omega\subseteq S$ its ramification locus and $K_\omega$ the fixed field of $\Ker(\omega)$. By a result of Weil it is known that the $L$-function
\begin{equation} \label{e:L(omega)} L(\omega,s):=\prod_{v\notin S_\omega}(1-\omega([v]_{K_\omega}) q_v^{-s})^{-1} \end{equation}
is rational in $q^{-s}$: there is a polynomial $P_\omega(u)\in\CC[u]$ such that
\begin{equation} \label{e:rational} L(\omega,s)=\frac{P_\omega(q^{-s})}{(1-\omega(\Fr_q)\,q^{-s})^{\delta_\omega}(1-\omega(\Fr_q)\,q^{1-s})^{\delta_\omega}}\,, \end{equation}
where $\delta_\omega=1$ if $\omega$ factors through a quotient of $\Gal(K\bar\FF/K)$, 0 else. Equalities \eqref{e:rational} and \eqref{e:L*(S,omega)} imply that
$$\omega(\Theta_{M,S,T}(u))=(1-\omega(\Fr_q)\,q^{-s})^{\delta_S}\cdot \prod_{v\in T} (1-\omega([v]_{K_\omega})\,q_v^{1-s})\cdot \prod_{S-S_\omega} (1-\omega([v]_{K_\omega})\,q_v^{-s})\cdot L(\omega,s)$$
is an element in the polynomial ring $\mathbb{C}[q^{-s}]$; since this holds for arbitrary $\omega$, it follows that $\Theta_{F,S,T}(u)$ belongs to $\ZZ[\Gal(F/K)][u]$ for any finite degree intermediate field $F$ of $M/K$. \footnote{A more detailed discussion of $\Theta_{F,S,T}(u)$ can be found in \cite[V, Proposition 2.15]{t84}. See also \cite[\S3]{gro88}.}
Therefore the coefficients of the power series $\Theta_{M,S,T}(u)\in\ZZ[\Gal(M/K)][[u]]$ tend to zero in
$$\varprojlim_{F\subset M}\ZZ[\Gal(F/K)]=:\ZZ[[\Gal(M/K)]]\subset \ZZ_p[[\Gal(M/K)]]\,.$$
To complete the proof now it suffices to observe that all factors in the auxiliary term $\prod_{v\in T} (1-q_v[v]_M\cdot u^{\deg(v)})$ are units in $\ZZ_p[[\Gal(M/K)]]\langle u\rangle$.
\end{proof}

Therefore, if $\alpha\in{\bar{\QQ}}_p$ with $|\alpha|\leq 1$, then $\Theta_{M,S}(\alpha)$ converges $p$-adically.

\begin{lemma} \label{l:evalTheta}
Let $\alpha\in\bar\QQ$ be embedded in $\bar\QQ_p$ so that $|\alpha|\leq1$. Let $\omega\colon{\Gal(M/K)}\rightarrow \CC^{\times}$ be a continuous character and extend it to a ring homomorphism $\ZZ[\alpha][[\Gal(M/K)]]\rightarrow\CC$ by $\omega(\alpha)=q^{-s_0}$. Then
\begin{equation} \label{e:evalTheta}  \omega\big(\Theta_{M,S}(\alpha)\big)=L_S^*(\omega,s_0)\,. \end{equation}
\end{lemma}

Note that, by \eqref{e:rational}, $L_S^*(\omega,s)$ can have a pole at $s$ only when $q^{-s}=q^{-1}\omega(\Fr_q)^{-1}$: hence the right-hand side of \eqref{e:evalTheta} is well-defined.

\begin{proof} The key is to note that $\omega\big(\Theta_{M,S}(u)\big)= \omega\big(\Theta_{K_\omega,S}(u)\big).$
As observed in the proof of Proposition \ref{p:Thetaconv}, by Weil's theorems $\Theta_{K_\omega,S}(u)$ is a rational function in $\ZZ[\Gal(K_\omega/K)](u)$: as such,
$\omega\big(\Theta_{K_\omega,S}(u)\big)=\omega(\Theta_{K_\omega,S})\big(\omega(u)\big)$.
\end{proof}

Define the Stickelberger element
\begin{equation} \label{e:thetaLST} \theta_{M,S}:=\Theta_{M,S}(1)\in\ZZ_p[[\Gal(M/K)]]\,.   \end{equation}

\begin{lemma} \label{l:thetanot0} If $\Kpinf\subseteq M$, then $\theta_{M,S}\neq0\,.$
\end{lemma}

\begin{proof}
By \eqref{e:evalTheta}, it suffices to show $L_S^*(\omega,0)\neq0$ for some $\omega\in \Gal(\Kpinf/K)^\vee\subset \Gal(M/K)^\vee$. Observe that for such an $\omega$, the $L$-function $L(\omega, s)$ is just a twist of the Dedekind zeta function $\zeta_K(s)$: hence, letting $\varepsilon:=\omega(\Fr_q)$ and $\varepsilon_v=\varepsilon^{\deg(v)}$, we have
$$L_S^*(\omega,s)=(1-\varepsilon\,q^{-s})^{\delta_S} \prod_{v\in S} (1-\varepsilon_vq_v^{-s})\cdot \frac{P_K(\varepsilon q^{-s})}{(1-\varepsilon q^{-s})(1-\varepsilon q^{1-s})}$$
since $(1-q^{-s})(1-q^{1-s})\zeta_K(s)=P_K(q^{-s})$ for some polynomial $P_K$. This proves our claim, because $P_K(\varepsilon)\neq0$ by the Riemann hypothesis over function fields and if we choose $\omega$ of order sufficiently high then $\varepsilon_v\neq1$ for all $v\in S$.
\end{proof}

\begin{remark}{\em The proof was particularly simple because we assume that $M$ contains a large arithmetic extension of $K$. More generally, one can use \cite[Lemma 1.2]{tan95} to deduce that $\theta_{M,S}=0$ if and only if there exists some place in $S$ splitting completely over $M/K$.}\end{remark}

\subsubsection{Relation with class groups}
In addition to the maps defined in (H1), (H2) of \S\ref{su:iw}, we will use the following morphism:
\begin{enumerate}
\item[(H3)] Since ${\tilde \Gamma}= \Gamma_{ar}\times H$, any character $\psi\colon H\lr\cO^\times$ can be extended to a homomorphism $\psi\colon\La_\cO(\tilde \Gamma) \lr \La_\cO(\Gamma_{ar})\,$, via $(h,g)\mapsto\psi(h)g$.
\end{enumerate}

For each $\psi\in H^\vee$ consider
\begin{equation} \theta_{\psi,L_{ar},S}^\sharp:=\psi(\theta_{\tilde L,S}^\sharp)=\psi\big(\Theta_{\tilde L,S}(1)^\sharp\big)\in \La_{\cO}(\Gamma_{ar}). \end{equation}
The following theorem was proved in \cite{crw87}; a simplified proof (avoiding the use of crystalline cohomology) has recently been given in \cite{blt09}.

\begin{mytheorem}\label{t:class} Let notation be as above and assume that $\psi\in H^\vee$ with $\psi(h)\in\cO$ for every $h\in H$. Then $\cO \fW_{\tilde L}^{(\psi)}$ is a finitely generated
$\La_{\cO}(\Gamma_{ar})$-module and
$$\chi_{\cO,\Gamma_{ar}}(\cO \fW_{\tilde L}^{(\psi)})= (\theta_{\psi,L_{ar},S}^\sharp)\,.$$
\end{mytheorem}

From $\omega\in\Gamma_{ar}^\vee$ one obtains a character $\omega\times\psi$ via \eqref{e:prod2}. Then $\omega(\theta_{\psi,L_{ar},S})=(\omega\times\psi) (\theta_{\tilde L,S})= L_S^*(\omega\times\psi,0)$ and the proof of Lemma \ref{l:thetanot0} is easily adapted to show that $\theta_{\psi,L_{ar},S}\neq0$. Thus Theorem \ref{t:class} implies $\fW_{\tilde L}$ is torsion over $\Lambda(\Gamma_{ar})$. By Proposition \ref{p:cs} we get

\begin{corollary} Both $\F X_p(A/\tilde L)$ and $\F X_p(A/L_{ar})$ are torsion $\La_\cO(\Gamma_{ar})$-modules. \end{corollary}

\end{subsection}

\begin{subsection}{Frobenius Twist of Stickelberger Elements} \label{su:cis}

In this section, we compute the characteristic ideal of $X_p(A/L_{ar})$. As in \S\ref{su:lambdai}, we assume that the ring $\cO$ contains all $\alpha_1,...,\alpha_g$ as well as $\psi(h)$ for all $\psi\in H^\vee$, $h\in H$.

\begin{proposition}\label{p:cs1}
Let $\lambda_i$ and $\psi_i$ be as in {\em{\S\ref{su:lambdai}}}. Then
$$\chi_{\cO,\Gamma_{ar}}\big((\cO\F X_p(A/\tilde L))^{(\psi)}\big) = \prod_{i=1}^g \big(\lambda_i^* (\theta_{\psi\psi_i^{-1},L_{ar},S}^\sharp)\big) $$
and
$$\chi_{\cO,\Gamma_{ar}}\big(\cO\F X_p(A/L_{ar})\big) = \prod_{i=1}^g \big(\lambda_i^* (\theta_{\psi_i^{-1},L_{ar},S}^\sharp)\big)\,.$$
\end{proposition}

\begin{proof} The statement follows from Proposition \ref{p:cs}, Lemma \ref{l:phi[]chi} and Theorem \ref{t:class}. \end{proof}

\begin{subsubsection}{The Stickelberger element for $A$}  \label{ss:stickA} Let $M/K$ be an extension as in \S\ref{ss:stickelb}. For $i=1,...,g$, define
\begin{equation}\label{e:thetasi}
\theta_{A,M,S,i}^+:=\Theta_{M,S}(\alpha_i^{-1})^\sharp  \in \cO[[\Gal(M/K)]]\,.
\end{equation}
This makes sense by Proposition \ref{p:Thetaconv}, because $\alpha_i$ is a unit in $\cO$. Put
\begin{equation}\label{e:theta+} \theta_{A,M,S}^+:=\prod_{i=1}^g \theta_{A,M,S,i}^+\,. \end{equation}
Note that $\theta_{A,L_{ar},S}^+\in\La(\Gamma_{ar})$, since the set $\{\alpha_1,...,\alpha_g\}$ is stable under the action of $\Gal({\bar \QQ_p}/\QQ_p)$. Define
$$\theta_{A,M,S}:=\theta_{A,M,S}^+\cdot(\theta_{A,M,S}^+)^\sharp.$$

\begin{proposition}\label{p:cs2} We have
$$\chi(X_p(A/L_{ar}))=(\theta_{A,L_{ar},S}).$$
\end{proposition}

\begin{proof} Since $A$ and $A^t$ are isogenous, they share the same twist matrix. Thus, in view of Proposition \ref{p:xffsharp} and Proposition \ref{p:cs1}, we only need to prove the equality
$$\lambda_i^*(\theta_{\psi_i^{-1},L_{ar},S}^\sharp)=\theta_{A,L_{ar},S,i}^+\,.$$
But this is just a matter of unwinding definitions. The composition $\lambda_i^*\circ\psi_i^{-1}\circ\cdot^\sharp \colon\tilde\Gamma\lr \La_{\cO}(\Gamma_{ar})$ sends $\gamma=(\gamma_H,\gamma_{\Gamma_{ar}})$ to $\lambda_i(\gamma)\gamma_{\Gamma_{ar}}^{-1}$. In particular,
$$(\lambda_i^*\circ\psi_i^{-1}\circ\cdot^\sharp)(1-[v]_{\tilde L}\cdot u^{\deg(v)})=1-\lambda_i([v]_{\tilde L})[v]_{L_{ar}}^{-1} u^{\deg(v)}=1-[v]_{L_{ar}}^{-1}(\alpha_i^{-1} u)^{\deg(v)}\,,$$
which implies
$$\lambda_i^* \big(\psi_i^{-1}(\Theta_{\tilde L,S}(1)^\sharp)\big)= \Theta_{L_{ar},S}(\alpha_i^{-1})^\sharp =\theta_{A,L_{ar},S,i}^+\,.$$
\end{proof}

\end{subsubsection}
\end{subsection}

\begin{subsection}{$p$-adic interpolation of the $L$-function} \label{s:interpol}

Recall that the $L$-function of $A$ is
$$L(A,s):=\prod_v P_v(A,q_v^{-s})^{-1}=\prod_v\det(1- q_v^{-s}[v]|_{T_\ell A})^{-1},$$
where $[v]$ is the arithmetic Frobenius at $v$, acting on the $\ell$-adic ($\ell\neq p$) Tate module $T_\ell A$.
Let $\omega\colon\Gamma\rightarrow \boldsymbol\mu_{p^{\infty}}$ be a continuous character: the twisted $L$-function is then
$$L_S(A,\omega,s):=\prod_{v\notin S} P_v(A,\omega([v]_L)q_v^{-s})^{-1}.$$

Before stating the interpolation formula relating $\theta_{A,L,S}$ with $L_S(A,\omega,1)$, we need to introduce some notation. Recall the numbers $\alpha_i,\beta_i$ introduced in \S\ref{ss:twist}: in the following, we fix an embedding of the set $\{\alpha_i,\beta_i\}$ in $\CC$.
Let $S_\omega\subseteq S$ be the set of places where $\omega$ ramifies, $K_\omega$ the fixed field of $\Ker(\omega)$ and put
$$\Xi_{S,\omega}:=\prod_{i=1}^g\prod_{v\in S-S_\omega}\frac{1-\omega([v]_{K_\omega})^{-1}\alpha_{i,v}^{-1}}{1-\omega([v]_{K_\omega})\beta_{i,v}^{-1}}\in\CC$$
(where $\alpha_{i,v}:=\alpha_i^{\deg(v)}$ and $\beta_{i,v}:=\beta_i^{\deg(v)}=q_v\alpha_{i,v}^{-1}$). Also, set
$$\Delta_S:=\prod_{i=1}^g(1-\alpha_i^{-1}\Fr_q^{-1})^{\delta_S}(1-\alpha_i^{-1}\Fr_q)^{\delta_S}\in\CC[\Gamma]\,,$$
where $\delta_S\in\{0,1\}$ is the same as in \eqref{e:Thetaprod}.

Denote by $\kappa$ the genus of $K$ (that is, the genus of the corresponding curve $C/\FF$) and by $d_\omega$ the degree of the conductor of $\omega$. Fix an additive character $\Psi\colon{\bf A}_K/K\rightarrow\boldsymbol\mu_{p^\infty}$ on the adele classes of $K$ and let $b=(b_v)$ be a differental idele attached to $\Psi$ (\cite[p.~113]{We73}) and, for every place $v$, let $\alpha_v$ be the self-dual Haar measure on $K_v$ with reference to $\Psi_v$. Then define
$$\tau(\omega):=\begin{cases}
\displaystyle \omega(\Fr_q)^{2-2\kappa} & \text{if } \omega \text{ factors through } \Gal(K\bar\FF/K) \\ {}\\
\displaystyle \frac{1}{\omega(b)}\prod_{v\in S_\omega}\frac{1}{|b_v|_v^{1/2}}\int_{O_v^\times}\omega_v(x)\Psi_v(b_v^{-1}x)d\alpha_v(x) & \text{otherwise} \end{cases}$$
where $O_v$ is the ring of integers of $K_v$.

\begin{mytheorem} \label{t:interpo} The element $\theta_{A,L,S}$ interpolates the $L$-function of $A$: for any continuous character $\omega\colon\Gamma\rightarrow \boldsymbol\mu_{p^\infty}$ we have
\begin{equation} \label{e:interpo}
\omega(\theta_{A,L,S})=\tau(\omega)^g\big(q^{g/2}\prod_{i=1}^g\alpha_i^{-1}\big)^{2\kappa-2+d_\omega}\,\omega(\Delta_S)\,\Xi_{S,\omega}\, L_S(A,\omega,1)\,.
\end{equation}
\end{mytheorem}

\begin{proof} By definition, recalling \eqref{e:thetasi}, we have
\begin{equation} \label{e:thetaAprod}
\theta_{A,L,S}=\prod_{i=1}^g\theta_{A,L,S,i}^+(\theta_{A,L,S,i}^+)^\sharp= \prod_{i=1}^g\Theta_{L,S}(\alpha_i^{-1})^\sharp \, \Theta_{L,S}(\alpha_i^{-1})\,.
\end{equation}
For every $i$ fix $s_i\in\CC$ such that $\alpha_i=q^{s_i}$, so that $\beta_i=q^{1-s_i}$. Lemma \ref{l:evalTheta} yields
$$\omega\big(\Theta_{L,S}(\alpha_i^{-1})\big)=L_S^*(\omega,s_i)$$
and
$$\omega\big(\Theta_{L,S}(\alpha_i^{-1})^\sharp\big)=L_S^*(\omega^{-1},s_i)\,,$$
because $\omega(\lambda^\sharp)=\omega^{-1}(\lambda)$ for any $\lambda\in\La$.
It is well known that the $L$-function satisfies the functional equation
\begin{equation}\label{e:functeqL} L(\omega^{-1},s)=\tau(\omega)q^{(\frac{1}{2}-s)(2\kappa-2+d_\omega)}L(\omega,1-s) \end{equation}
(see e.g.~\cite[VII, Theorems 4 and 6]{We73}). Defining
$$L_S(\omega^{-1},s):= L(\omega^{-1},s) \prod_{v\in S-S_\omega}(1-\omega([v]_{F_\omega})^{-1}q_v^{-s})=(1-\omega(\Fr_q)^{-1}q^{-s})^{-\delta_S}L_S^*(\omega^{-1},s)\,, $$
formula \eqref{e:functeqL} implies
$$L_S(\omega^{-1},s_i)=\tau(\omega)(q^{1/2}\alpha_i^{-1})^{2\kappa-2+d_\omega} \prod_{v\in S-S_\omega}\frac{1-\omega([v]_{F_\omega})^{-1}\alpha_{i,v}^{-1}}{1-\omega([v]_{F_\omega})\beta_{i,v}^{-1}} \cdot L_S(\omega,1-s_i)\,.$$
Putting everything together, we have obtained
$$\omega(\theta_{A,L,S})=\tau(\omega)^g\big(q^{g/2}\prod_{i=1}^g\alpha_i^{-1}\big)^{2\kappa-2+d_\omega}\,\omega(\Delta_S)\,\Xi_S(\omega)\prod_{i=1}^g L_S(\omega,s_i)\,L_S(\omega,1-s_i)\,.$$
On the other hand,
\begin{eqnarray}  L_S(A,\omega,1) &=& \prod_{v\notin S}\prod_{i=1}^g \big(1-\omega([v]_L)\beta_{i,v}q_v^{-1}\big)^{-1}\big(1-\omega([v]_L)\alpha_{i,v}q_v^{-1}\big)^{-1}\\
 \label{e:decomposLA} {}&=&\prod_{i=1}^g L_S(\omega,s_i)\,L_S(\omega,1-s_i) \,,\end{eqnarray}
since, by \cite[Corollary 4.37]{mazur72},
$$P_v(A,x)=\prod_{i=1}^g (1-\alpha_{i,v}x)\big(1-\beta_{i,v}x).$$
\end{proof}

\begin{remark} {\em The function $L_S(A,\omega,s)$ cannot have a pole at $s=1$: this follows from \eqref{e:decomposLA} and \eqref{e:rational}, using the fact that the $\alpha_i$'s cannot be roots of 1 (since, as eigenvalues of the Frobenius endomorphism of $A$, the $\alpha_i$'s are Weil $q$-numbers).} \end{remark}

\subsubsection{The $p$-adic $L$-function for $L=\Kpinf$} In the case $S=\emptyset$ (that is, $L$ is the arithmetic extension $\Kpinf$), formula \eqref{e:interpo} can be improved a little. We have $\Xi_{\emptyset,\omega}=1$ and $\delta_\omega=0$ for all $\omega$, while
$$\Delta_\emptyset=\prod_{i=1}^g(1-\alpha_i^{-1}\Fr_q^{-1})(1-\alpha_i^{-1}\Fr_q)$$
can be seen as an element in $\La$, because the $\alpha_i$'s are the eigenvalues of the twist matrix $\bo u\in GL_g(\ZZ_p)$.
Define
\begin{equation} \label{e:cLtilde}
\tilde{\cL}_{A/\Kpinf}:=\frac{1}{q^{g(\kappa-1)}\Delta_\emptyset}\big(\prod_{i=1}^g\alpha_i\big)^{2\kappa-2}\Fr_q^{g(2\kappa-2)}\theta_{A,\Kpinf,\emptyset}\in Q(\La)\,, \end{equation}
where $Q(\La)$ is the fraction field of $\La$. By Theorem \ref{t:interpo} we see that
\begin{equation} \label{e:interpoLtilde} \omega(\tilde{\cL}_{A/\Kpinf})=L(A,\omega,1) \end{equation}
for all characters of $\Gamma$. We can think of $\tilde{\cL}_{A/\Kpinf}$ as the value at $s=1$ of
$$\tilde{\cL}_{A/\Kpinf}(s):=\prod_{i=1}^g\prod_v \big(1-[v]_L\alpha_{i,v}q_v^{-s}\big)^{-1} \big(1-[v]_L\beta_{i,v}q_v^{-s}\big)^{-1} \in\La[[q^{-s}]].$$

In section \ref{su:padicL} below the reader will find a different $p$-adic $L$-function $\cL_{A/\Kpinf}$, defined in the case when $L$ is the arithmetic extension $\Kpinf$ and $A/K$ has semistable reduction. Theorem \ref{GIMC2} will prove that $\tilde{\cL}_{A/L}=\cL_{A/L}$ when $A$ is a constant abelian variety.

\end{subsection}


\begin{subsection}{Proof of the Main Conjecture} \label{su:MCconstant}

Proposition \ref{p:cs2} proves the Iwasawa Main Conjecture (i.e., the analogue of {\bf (IMC3)}) for $A/K$ a constant abelian variety when $L=L_{ar}$. To deal with the general case, we apply the following theorem.
For $M'\subset M$ two Galois extensions of $K$, let $p_{M/M'}\colon\ZZ_p[[\Gal(M/K)]]\rightarrow\ZZ_p[[\Gal(M'/K)]]$ denote the natural map of Iwasawa algebras. We write $\chi_M$, $\chi_{M'}$ for characteristic ideals in $\ZZ_p[[\Gal(M/K)]]$, $\ZZ_p[[\Gal(M'/K)]]$ respectively.

\begin{mytheorem}[Tan] \label{t:descent} Let $M/K$ be a $\ZZ_p^e$-extension, ramified at finitely many places, and $M'/K$ a $\ZZ_p^{e-1}$-subextension, $e\geq2$. Define $\vartheta_{M/M'}, \varrho_{M/M'}\in\ZZ_p[[\Gal(M'/K)]]$ by
\begin{equation} \label{e:thetaM/M'} \vartheta_{M/M'}:=\prod_{v\in S_{M/M'}}\prod_{i=1}^g(1-\alpha_{i,v}^{-1}[v]_{M'})(1-\alpha_{i,v}^{-1}[v]_{M'}^{-1}) \end{equation}
{\em{(}}where $ S_{M/M'}$ is the set of places ramified in $M/K$ but not in $M'/K${\em{)}} and
\begin{equation} \label{e:rhoM/M'} \varrho_{M/M'}:=\begin{cases}
\displaystyle \prod_{i=1}^g(1-\alpha_i^{-1}\Fr_q)(1-\alpha_i^{-1}\Fr_q^{-1}) & \text{when } M'=\Kpinf \\
1 & \text{otherwise}.  \end{cases}\end{equation}
Then
\begin{equation} \label{e:descent}
\vartheta_{M/M'}\cdot\chi_{M'}(X_p(A/M'))=\varrho_{M/M'}\cdot p_{M/M'}\big(\chi_M(X_p(A/M))\big)\,.
\end{equation}
\end{mytheorem}

This is just a special case of \cite[Theorem 1]{tan10b}: since $A$ has good reduction everywhere, the only terms appearing in $\prod_v\vartheta_v$ of {\em loc.~cit.} are the ones coming from changes in the set of ramified places, i.e., $\vartheta_{M/M'}$ as defined in \eqref{e:thetaM/M'}.

\begin{proof}[{\bf Proof of Theorem \ref{t:imc3constant}}] By Proposition \ref{p:cs2}, $\chi(X_p(A/L_{ar}))=(\theta_{A,L_{ar},S'})$.
If $L=L_{ar}$ then the result is already proved. Otherwise, apply Theorem \ref{t:descent} with $M=L_{ar}$ and $M'=L$. Then obviously $S_{M/M'}=\emptyset$, and hence $\vartheta_{M/M'}=1$.
Also, $\varrho_{L_{ar}/L}=1$ since $\Kpinf$ is not contained in $L$. Therefore \eqref{e:descent} implies the desired equality
$$\chi(X_p(A/L))=(\theta_{A,L,S})$$
as $p_{L_{ar}/L}(\theta_{A,L_{ar},S})=\theta_{A,L,S}\,$.
\end{proof}

\end{subsection}
\end{section}
\end{part}


\begin{part}{The case of the arithmetic $\ZZ_p$-extension} \label{part:semistable}

We keep the setting of Part \ref{part:constantA}: $K$ is a function field and $\FF$ its constant field, with $|\FF|=q$.

\section{Syntomic cohomology of abelian varieties} \label{s:genord}

As announced in the introduction, if we take as $L$ the arithmetic $\ZZ_p$-extension we can deal with all semistable abelian varieties and not just the constant ones. Thus in this part we let $A/K$ be an abelian variety with at worst semistable reduction, but we impose the restriction that $L=\Kpinf$.

\subsubsection{The arithmetic tower} We fix the notations for Part \ref{part:semistable}. For any $n\geq 0$, let $k_n/\FF$ be the
$\ZZ/p^n\ZZ$-extension of $\FF$ and let $\kpinf:=\cup_{n\geq 0} k_n$ denote the induced $\ZZ_p$-extension of $\FF$. Thus our tower becomes $K_n:=Kk_n$ and $L=\Kpinf:=K\kpinf$.
Since they are canonically isomorphic, we identify $\Gamma$, $\Gamma_n$ and $\Gamma^{(n)}$ with $\Gal(\kpinf/\FF)$, $\Gal(k_n/\FF)$ and $\Gal(\kpinf/k_n)$ respectively. As in Part \ref{part:constantA}, we denote by $\Fr_q$ the generator of $\Gal(\kpinf/\FF)$, $x\mapsto x^q$.

Recall the smooth proper geometrically connected curve $C/\FF$ which is the model of $K$ over $\FF$. Let $C_\infty:=C\times_{\FF} \kpinf$ and $C_n:=C\times_{\FF}k_n$. Let $\pi\colon C_\infty\to C$ and $\pi_n\colon C_n\to C$ denote the \'etale covering with Galois group $\Gamma$ and $\Gamma_n$ respectively. By abuse of notation, we will also denote by $\pi$ and  $\pi_n$ the associated morphisms in the log crystalline topos (\cite{bbm82}).

\subsection{The cohomology} \label{su:cohom}

\subsubsection{The Dieudonn\'e crystal} Let $\cal A$ denote the N\'eron model of $A$ over $C$. Let $U$ be the dense open subset of $C$ where $A$ has good reduction and $Z:=C- U$ the finite (possibly empty) set of points where $A$ has bad (at worst semistable) reduction. We endow $C$ with the log structure induced by the smooth divisor $Z$ and denote $C^\#$ this log-scheme. Let $D$ be the (covariant) log Dieudonn\'e crystal over $C^\#/W(\FF)$ associated with $A/K$ as constructed in \cite[IV]{KT03}. Recall the
following theorem of \cite{KT03}:

\begin{mytheorem} \label{Lie(D)} {\em(}\cite[\S 5.4(b) and \S 5.5]{KT03}{\em)} Let $i$ be the canonical morphism of topoi of \cite{bbm82} from the topos of sheaves on $C_{{\acute{e}t}}$ to the log crystalline topos $\left(C^\#/W(\FF)\right)_{crys}$.
There exists a surjective map of sheaves $D\to i_*(Lie({\cal A}))$ in $\left(C^\#/W(\FF)\right)_{crys}$.
\end{mytheorem}


\subsubsection{A distinguished triangle} We denote by $D^0$ the kernel of $D\rightarrow i_{\ast}(Lie({\cal A}))$ in the topos $(C^\#/W(\FF))_{crys}$. Let ${\bf 1}\colon D^0\rightarrow D$ be the natural inclusion. By applying the canonical projection $u_{\ast}$ from the log crystalline topos $(C^\#/W(\FF))_{crys}$ to the topos of sheaves on $C_{\acute{e}t}$, we get a distinguished triangle:
$$Ru_{\ast} D^0 \overset{\bf 1}{\lra}  Ru_{\ast} D\lra Lie({\cal A}).$$
We can twist this triangle by the divisor $Z$ to get a triangle:
\begin{equation}\label{T1}
Ru_{\ast} D^0 (-Z) \overset{\bf 1}{\lra}  Ru_{\ast} D (-Z) \lra Lie({\cal A})(-Z).
\end{equation}
where $D(-Z)$ is the twist of the log Dieudonn\'e crystal $D$ defined in \cite[\S 5.11]{KT03}.

\subsubsection{The syntomic complex} In \cite[\S 5.8]{KT03}, a Frobenius operator
$$\varphi\colon Ru_*D^0(-Z)\lra Ru_{\ast}D(-Z)$$
is constructed. We denote by ${\cS}_D$ the mapping fiber of the map
$${\bf 1}-\varphi\colon Ru_*D^0(-Z) \lra Ru_*D(-Z).$$
This complex is an object in the derived category of complexes of sheaves over $C_{\acute{e}t}$ and we have a distinguished triangle:
\begin{equation}\label{T2} \begin{CD}
\cS_D \lra Ru_{\ast} D^0(-Z) @>{{\bf 1}-\varphi}>> Ru_{\ast}D(-Z). \end{CD}
\end{equation}

\subsubsection{The cohomology theories} \label{ss:cohomtheories}


%


We define the following modules:
\begin{enumerate}
\item Let
$$P^i_n:=\coh^i_{\mathrm{crys}}(C_n^\#/W(k_n),\pi_n^*D(-Z))\,.$$
Then for any $n$, $P^i_n$ is a finitely generated $W(k_n)$-module endowed with a $\Fr_q$-linear operator $F_{i,n}$ induced by the Frobenius operator of the Dieudonn\'e crystal. Using the (log) crystalline base change by the morphism of topoi $\pi_n\colon(C_\infty^\#/W(k_n))_{crys}\to (C^\#/W(\FF))_{crys}$
(\cite[\S2.5.2]{Ka94}) and by flatness of the extensions $W(k_n)/\ZZ_p$, we have,
for $n\geq 1$,
\begin{equation} \label{e:pi0n} P^i_n\simeq P^i_0\otimes W(k_n)\,. \end{equation}
These isomorphisms identify the $\Fr_q$-linear operator $F_{i,n}$ on the left hand side with the $\Fr_q$-linear operator $F_{i,0}\otimes\Fr_q$ on the right hand side.

\item Let $M_{1,\infty}^i$ be the $i$th cohomology group of $$\RR\Gamma_{crys}(C_\infty^\#/W(\kpinf),\pi^*D^0(-Z))\otimes^{\mbb{L}} \QQ_p/\ZZ_p.$$
\item Let $M_{2,\infty}^i$ be the $i$th cohomology group of $$\RR\Gamma_{crys}(C_\infty^\#/W(\kpinf),\pi^*D(-Z))\otimes^{\mbb{L}} \QQ_p/\ZZ_p.$$
\item Let $M_{1,n}^i$ be the $i$th cohomology group of $$\RR\Gamma_{crys}(C_n^\#/W(k_n),\pi_n^*D^0(-Z))\otimes^{\mbb{L}} \QQ_p/\ZZ_p.$$
\item Let $M_{2,n}^i$ be the $i$th cohomology group of $$\RR\Gamma_{crys}(C_n^\#/W(k_n),\pi^*_nD(-Z))\otimes^{\mbb{L}} \QQ_p/\ZZ_p.$$

Again by the base change theorem, we have, for $k=1,2$, an isomorphism of torsion $W(\kpinf)$-modules
\begin{equation} \label{e:mi0infty} M^i_{k,\infty}\simeq M^i_{k,0}\otimes W(\kpinf) \end{equation}
and for any $n\geq 0$ an isomorphism of torsion $W(k_n)$-modules $$M^i_{k,n}\simeq M^i_{k,0}\otimes W(k_n)$$ identifying the $\Fr_q$-linear operator ${\bf 1}-\varphi_{i,n}$ on the left
hand side with the $\Fr_q$-linear operator ${\bf 1}\otimes id-\varphi_{i,0}\otimes\Fr_q$ on the right hand side.
\item Let $L^i_\infty$ be the $i$th cohomology group of
$$\RR\Gamma\big(C_{\infty},\pi^*Lie(\A(-Z))\big)\otimes^{\mbb{L}}\QQ_p/\ZZ_p
=\RR\Gamma\big(C_\infty,\pi^*Lie(\A(-Z))\big)[1]\,.$$

\item Let $L^i_n$ be the $i$th cohomology group of
$$\RR\Gamma\big(C_n,\pi^*_nLie(\A(-Z))\big)\otimes^{\mbb{L}}\QQ_p/\ZZ_p=\RR\Gamma\big(C_n,\pi^*_n Lie(\A(-Z))\big)[1]\,.$$

By the Zariski base change formula (note that the cohomology of the finite locally free module $Lie({\A})(-Z)$ is the same in the \'etale or Zariski site), we have isomorphisms
$$L^i_\infty\simeq L^i_0\otimes W(\kpinf)$$
and, for any $n\geq 0$,
$$L^i_n\simeq L^i_0\otimes W(k_n).$$
In particular, since $L^i_0$ is a finite $\FF_p$-vector space with rank $d(L^i_0)$, we deduce that $L^i_\infty$ is a finite $\kpinf$-vector space while $L^i_n$ is a
finite $k_n$-vector space, both with the same rank $d(L^i_0)$.
\item Let $$N^i_\infty := H^i_{syn}( C_\infty , \pi^*\mathcal{S}_D \otimes \QQ_p/\ZZ_p )$$
be the $i$th cohomology group of $$\RR\Gamma(C_\infty,\pi^*\cS_D)\otimes^{\mbb{L}}\QQ_p/\ZZ_p.$$
\item Let $N^i_n$ be the $i$th cohomology group of $$\RR\Gamma(C_n,\pi_n^*\cS_D)\otimes^{\mbb{L}}\QQ_p/\ZZ_p.$$
\end{enumerate}

The distinguished triangles (\ref{T1}) and (\ref{T2}) induce, by passing to the cohomology, the following long exact sequences:
\begin{gather}
\begin{CD} \label{L1} ... @>>> N^i_\infty @>>> M_{1,\infty}^i @>{\bf 1}-\varphi_{i,\infty}>> M_{2,\infty}^i @>>> ... \end{CD}\\
\begin{CD} \label{L2} ... @>>> L^i_\infty @>>> M_{1,\infty}^i @>{\bf 1}>> M_{2,\infty}^i @>>> ...  \end{CD}
\end{gather}
which are inductive limits of the long exact sequences
\begin{gather}
\begin{CD} \label{L1n} ... @>>> N^i_n @>>>  M_{1,n}^i @>{{\bf 1}-\varphi_{i,n}}>> M_{2,n}^i @>>> ...  \end{CD}\\
\begin{CD} \label{L2n} ... @>>> L^i_n @>>>  M_{1,n}^i @>{\bf 1}>> M_{2,n}^i @>>> ...  \end{CD}
\end{gather}
Note that the cohomology theories $M$ and $N$ are concentrated in degrees 0,1 and 2 and the cohomology theory $L$ is concentrated in degrees 0 and 1.\\

In \cite{KT03}, the following was proved (see \cite[\S3.3.3, \S3.3.4 and \S3.3.5]{KT03}):
\begin{lemma}\label{cofinite}
For $k=1$ or $2$, there exists a map
$$f_k\colon P^i_0[p^{-1}]\lr M^i_{k,0}$$
satisfying the following conditions:
\begin{enumerate}
\item The kernel of $f_k$ is a $\ZZ_p$-lattice in $P^i_0[\frac{1}{p}]$ and the cokernel is a finite group. In particular, $M^i_{1,0}$ and
  $M^i_{2,0}$ are torsion $\ZZ_p$-modules with the same finite corank.
\item The diagrams
    \[\begin{CD}
    P^i_0[\frac{1}{p}] @>{id}>> P^i_0[\frac{1}{p}] && \hspace{60pt} && P^i_0[\frac{1}{p}] @>{p^{-1}F_{i,0}}>> P^i_0[\frac{1}{p}]\\
    @V{f_1}VV @V{f_2}VV  \text{and} && @V{f_1}VV @V{f_2}VV \\
    M^i_{1,0} @>{{\bf 1}}>> M^i_{2,0} && \hspace{60pt} && M^i_{1,0} @>{\varphi_{i,0}}>> M^i_{2,0}
    \end{CD}\]
    commute.
\end{enumerate}
\end{lemma}

\begin{lemma}\label{linalg1} Let $M$ be a torsion $\ZZ_p$-module of cofinite type. Then the Pontryagin dual of $M\otimes_{\ZZ_p} W(\kpinf)$ is a finitely generated $\La$-module. Moreover, we have an isomorphism of $\La$-modules
$$\big(M\otimes_{\ZZ_p} W(\kpinf)\big)^\vee\simeq\La^r\oplus \bigoplus_{i=1}^s \La/(p^{n_i})\,.$$
\end{lemma}

\begin{proof} Let $X_\infty$ denote the Pontryagin dual of $M\otimes_{\ZZ_p} W(\kpinf)$. Then $X_\infty$ is the limit of the projective system of Pontryagin duals $X_n:=\big(M\otimes_{\ZZ_p} W(k_n)\big)^\vee$. Note that the functor $M\rightsquigarrow(M\otimes_{\ZZ_p} W(\kpinf))^\vee$ is exact.

In the case $M=\ZZ/p\ZZ$, we have a projective system $\{X_n=(k_n)^\vee\}_n$ where the transition maps are the Pontryagin dual of the canonical inclusions
$k_n\hookrightarrow k_{n+1}$. Let $\Omega:=\FF[[\Gamma]]\simeq \FF[[T]]$. We have
$$(X_\infty)_\Gamma=X_\infty/TX_\infty=\big((\kpinf)^\Gamma\big)^\vee\simeq\FF\,.$$
Now, by Nakayama's lemma, this implies that $X_\infty$ is a cyclic $\Omega$-module and since $X_\infty$ is infinite we have $X_\infty\simeq\Omega=\La/p\La$.

If $M=\ZZ/p^j\ZZ$, then $X_\infty=(W_j(\kpinf))^\vee={\ilim} (W_j(k_n))^\vee=:Y_j$ (where $W_j$ denotes Witt vectors of length $j$) and we prove by the same argument that $Y_j$ is a cyclic $W_j(\FF)[[\Gamma]]$-module. In particular for each $j$ we have a surjective map:
$$W_j(\FF)[[\Gamma]]\,\lr Y_j\,.$$
We prove by induction on $j$ that $Y_j$ is a free $W_j(\FF)[[\Gamma]]$-module. The case $j=1$ was treated above. Now assume the assertion is true for $j-1$ and consider the commutative diagram of short exact sequences:
$$\begin{CD}
0 @>>> Y_{j-1} @>{\delta}>>  Y_j @>{\epsilon}>> Y_1 @>>> 0 \\
&& @A{\alpha}AA @A{\beta}AA  @A{\gamma}AA \\
0 @>>> W_{j-1}(\FF)[[\Gamma]] @>{\times p}>> W_j(\FF)[[\Gamma]] @>>> \Omega @>>> 0 \,.\end{CD}$$
The upper horizontal line is induced by the exact sequence $\ZZ/p\ZZ\hookrightarrow\ZZ/p^j\ZZ\twoheadrightarrow\ZZ/p^{j-1}\ZZ\,.$ The vertical maps are constructed as follows: define first $\beta$ as the map sending 1 to a generator $x$ of the cyclic $W_j(\FF)[[\Gamma]]$-module $Y_j$. Then $px\in\Ker(\epsilon)=\image(\delta)$. Let $w\in Y_{j-1}$ be such that $px=\delta(w)$ and let $y=\epsilon(x)$. Finally, define $\alpha$ to be the map sending 1 to $w$ and $\gamma$ to be the map sending 1 to $y$. The map $\gamma$ is surjective because $\beta$ and $\epsilon$ are surjective. So $\gamma$ is an isomorphism since $Y_1$ is infinite and the only proper quotients of $\Omega$ are finite. We deduce by the snake lemma that $\alpha$ is surjective and therefore, by the induction hypothesis, an isomorphism. This implies that $\beta$ is also an isomorphism, by the 5-lemma.

The case $M=\QQ_p/\ZZ_p$ is deduced from the case $M=\ZZ/p^j\ZZ$ by passing to the inductive limit in $j$ and the general case follows from the cases $M=\QQ_p/\ZZ_p$ and $M=\ZZ/p^j\ZZ$.
\end{proof}

We deduce from Lemmas \ref{linalg1} and \ref{cofinite}:
\begin{corollary}\label{M,L} \hspace{5pt}
\begin{enumerate}
\item
For any $i\geq 0$, $(M^i_{1,\infty})^\vee$ and $(M^i_{2,\infty})^\vee$ are two $\La$-modules of finite type with the same rank $r_i$ (equal to the $\ZZ_p$-rank of $P^i_0${\em )} and torsion parts isomorphic to $\oplus_{j=1}^s \La/p^{n_j}\La.$
\item For any $i$, $(L^i_\infty)^\vee$ is a finitely generated torsion $\La$-module isomorphic to $\oplus_{j=1}^{d(L^i_0)} \La/p\La\,.$
\item Passing to the Pontryagin dual, the long exact sequences \eqref{L1} and \eqref{L2} induce long exact sequences of $\La$-modules with $\La$-linear operators.
\end{enumerate}
\end{corollary}

\begin{proof} The first assertion is a consequence of Lemmas \ref{linalg1} and \ref{cofinite}. Assertion (2) is proved similarly to the case of $M^i_{k,\infty}$. For the third assertion, note that $\Gamma$ is an abelian group and therefore any $\tau\in\Gamma$ commutes with $\Fr_q$. Hence, the operator ${\bf 1}-\varphi$ is $\La$-linear.
\end{proof}

Next, we are going to prove that $(N^i_\infty)^\vee$ is a finitely generated torsion $\La$-module for $i=0,\dots,2$.

\begin{proposition} \label{0case} The $\La$-modules $(N^0_\infty)^\vee$ and $\left(\Coker({\bf 1}-\varphi_{0,\infty})\right)^\vee$ are $\La$-torsion.
\end{proposition}

\begin{proof} Reasoning as in \cite[\S2.5.2]{KT03} (see \eqref{e:alg-arithm} below), one obtains that the module $(N^0_\infty)^\vee$ is a quotient of $A(\Kpinf)[p^\infty]^\vee$. The latter is a finitely generated $\ZZ_p$-module (and so a torsion $\La$-module), hence the claim for $(N^0_\infty)^\vee$ is proven. As for $(\Coker({\bf 1}-\varphi_{0,\infty}))^\vee$, note that we have an exact sequence of $\La$-modules
$$\begin{CD} 0\lr (\Coker({\bf 1}-\varphi_{0,\infty}))^\vee\ \lr (M_{2,\infty}^0)^\vee @>{({\bf 1}-\varphi_{0,\infty})^\vee}>> (M_{1,\infty}^0)^\vee \lr (N^0_\infty)^\vee \lr 0. \end{CD}$$
Since $(M_{1,\infty}^0)^\vee$ and $(M_{2,\infty}^0)^\vee$ have the same $\La$-rank, it implies that the $\La$-rank of $(\Coker({\bf 1}-\varphi_{0,\infty}))^\vee$ is equal to that of $(N^0_\infty)^\vee$, which is zero.
\end{proof}

We now prove that $(N^1_\infty)^\vee$ is $\La$-torsion.\\

The exact sequence:
\begin{multline}\label{alpha_beta}
\cdots \lra \coh^1_{\mathrm{syn}} (C_\infty,\pi^*\mathcal{S}_{D}) \otimes \QQ_p
\buildrel{\alpha}\over\lra \coh^1_{\mathrm{syn}} (C_\infty,\pi^*\mathcal{S}_{D}
\otimes \QQ_p/\ZZ_p) \\
\buildrel{\beta}\over\lra \coh^2_{\mathrm{syn}} (C_\infty,\pi^*\mathcal{S}_{D})\buildrel{\gamma}\over\lra \coh^2_{\mathrm{syn}} (C_\infty,\pi^*\mathcal{S}_{D})\otimes \QQ_p\lra \cdots
\end{multline}
induces a short exact sequence
\begin{equation}\label{exact-section2}
0\lra \image(\alpha)\lra N^1_\infty\lra \image(\beta)\lra 0.
\end{equation}
By taking the Pontryagin dual of $(\ref{exact-section2})$, we have
\begin{equation}\label{exactbisbis-section2}
0\lra \image(\beta)^\vee \lra (N^1_\infty)^\vee \lra \image(\alpha )^\vee \lra 0,
\end{equation}
where the modules and the morphisms are naturally defined over $\La$.

\begin{lemma}\label{finitevector}
$\coh^1_{\mathrm{syn}} (C_\infty,\pi^*\mathcal{S}_{D})\otimes \QQ_p$ is a finite dimensional
$\QQ_p$-vector space.
\end{lemma}

\begin{proof} The long exact sequence
\begin{multline*}
\cdots
\lra \coh^i_{\mathrm{syn}} (C_\infty,\pi^*\mathcal{S}_{D})\otimes \QQ_p\lra \coh^i_{\mathrm{crys}}\big(C_\infty^\#/W(\kpinf),\pi^*D^0(-Z)\big)\otimes \QQ_p \\
\buildrel{1-\varphi_i}\over\lra \coh^i_{\mathrm{crys}}\big(C_\infty^\#/W(\kpinf),\pi^*D(-Z)\big)\otimes \QQ_p\lra \cdots \end{multline*}
can be rewritten
\begin{multline*}
\cdots \lra \coh^i_{\mathrm{syn}} (C_\infty,\pi^*\mathcal{S}_{D})\otimes \QQ_p\lra \coh^i_{\mathrm{crys}}(C_\infty^\#/W(\kpinf),\pi^*D(-Z))\otimes \QQ_p \\
\buildrel{1-\varphi_i}\over\lra \coh^i_{\mathrm{crys}}(C_\infty^\#/W(\kpinf),\pi^*D(-Z)) \otimes \QQ_p\lra \cdots
\end{multline*}
(because the difference between the middle terms in these sequences is $L^i_\infty$, which, thanks to Lemma \ref{M,L}, (2) and (3), is known to be $p$-torsion).\\
We deduce from this long exact sequence the following short exact sequence:
$$ 0\lra \Coker(1-\varphi_0)\lr \coh^1_{\mathrm{syn}} (C_\infty,\pi^*\mathcal{S}_{D})\otimes \QQ_p\lra \Ker(1-\varphi_1)\lra 0. $$
Since $\coh^i_{\mathrm{crys}}(C_\infty^\#/W(\kpinf),\pi^*D(-Z))\otimes\QQ_p$ is a finite dimensional $W(\kpinf)[\frac{1}{p}]$-vector space and \eqref{e:pi0n} holds, the assertion is implied by the following:
\begin{lemma}\label{sblem-section2} Let $V$ be a finite dimensional $\QQ_p$-vector space endowed with a linear operator $\varphi\colon V\rightarrow V$. Then
$1-\varphi\otimes\Fr_q\colon V\otimes_{\ZZ_p}W(\kpinf)\to V\otimes_{\ZZ_p}W(\kpinf)$ is a surjective map whose kernel is a finite dimensional $\QQ_p$-vector space.
\end{lemma}
\begin{proof} One can easily check that the proof of \cite[Lemme 6.2]{EL97} remains true if we replace $\bar{\FF}_p$ by $\kpinf$.
\end{proof}
\end{proof}

\begin{corollary}\label{lemmaN^1}
The group $\image(\alpha)^{\lor}$ of \eqref{exactbisbis-section2} is a free $\ZZ_p$-module of finite rank.
\end{corollary}

\begin{proof}
By (\ref{L1}), $N^1_\infty$ is a torsion $\ZZ_p$-module. Thus, $\image(\alpha)$ is a torsion $\ZZ_p$-module
which is a quotient of $\coh^1_{\mathrm{syn}}(C_\infty,\pi^*\mathcal{S}_{D})\otimes \QQ_p$. We deduce from Lemma \ref{finitevector} that $\image(\alpha)$
is cofree of finite corank $n\leq \dim_{\QQ_p}(\coh^1_{\mathrm{syn}} (C_\infty,\pi^*\mathcal{S}_{D})\otimes \QQ_p)$.
\end{proof}

We now study the term $\image(\beta)$:
\begin{lemma}\label{beta}
The group $\image(\beta )^{\lor}$ of \eqref{exactbisbis-section2} is a finitely generated torsion $\La$-module.
\end{lemma}

\begin{proof}
Note that \eqref{alpha_beta} yields an isomorphism $\image(\beta)\simeq \Ker(\gamma)$.
The kernel of the map
$$\gamma\colon\coh^2_{\mathrm{syn}} (C_\infty,\pi^*\mathcal{S}_{D})\lra \coh^2_{\mathrm{syn}} (C_\infty,\pi^*\mathcal{S}_{D})\otimes \QQ_p$$
is $\coh^2_{\mathrm{syn}} (C_\infty,\pi^*\mathcal{S}_{D})[p^\infty]$.
Recall that we have a short exact sequence:
$$ 0\lra \Coker({\bf 1}-\varphi_{1,\infty})\lra \coh^2_{\mathrm{syn}} (C_\infty,\pi^*\mathcal{S}_{D}) \lra \Ker({\bf 1}-\varphi_{2,\infty})\lra 0. $$
By taking the $p$-power torsion part of this sequence, we have
\begin{equation}\label{exact2-section2}
0\lra \Coker({\bf 1}-\varphi_{1,\infty})[p^\infty] \lra \image(\beta) \lra \Ker({\bf 1}-\varphi_{2,\infty})[p^\infty]\,.
\end{equation}
Taking Pontryagin duals, we have the following:
\begin{equation}\label{seq}
\Ker({\bf 1}-\varphi_{2,\infty})[p^\infty]^\vee \lra \image(\beta)^\vee  \lra \Coker({\bf 1}-\varphi_{1,\infty})[p^\infty]^\vee \lra 0,
\end{equation}
where the modules and the morphisms are defined over $\La$.
By the sequence $(\ref{seq})$, it is enough to show that $\Coker({\bf 1}-\varphi_{1,\infty})[p^\infty]^\vee$ and $\Ker({\bf 1}-\varphi_{2,\infty})[p^\infty]^\vee$ are finitely generated
torsion $\La$-modules. But these two groups are both $p$-torsion subgroups of finitely generated $W(\kpinf)$-modules, so their Pontryagin duals are $\La$-torsion by Lemma \ref{linalg1}.
\end{proof}

\begin{corollary}\label{1case} The groups
$$(N^1_\infty)^\vee \text{, } (\Ker({\bf 1}-\varphi_{1,\infty}))^\vee \text{, } (\Coker({\bf 1}-\varphi_{1,\infty}))^\vee \text{, } (\Ker({\bf 1}-\varphi_{2,\infty}))^\vee \text{ and } (N^2_\infty)^\vee$$
are $\La$-torsion. In particular, $X_p(A/\Kpinf)$ is $\La$-torsion.
\end{corollary}

\begin{proof} The module $(N^1_\infty)^\vee$ is $\La$-torsion by Lemma \ref{beta} and Corollary \ref{lemmaN^1}. By \eqref{L1}, using the short exact sequence
\begin{equation} \label{e:sesNi} 0\lr \Coker({\bf 1}-\varphi_{i-1,\infty})\lr N^i_\infty\lr \Ker({\bf 1}-\varphi_{i,\infty})\lr 0 \end{equation}
we find that $(\Ker({\bf 1}-\varphi_{1,\infty}))^\vee$ is $\La$-torsion. Then by using the exact sequence
$$0\lr \Ker({\bf 1}-\varphi_{1,\infty})\lr M^1_{1,\infty}\lr M^1_{2,\infty}\lr \Coker({\bf 1}-\varphi_{1,\infty})\lr 0$$
we deduce that $(\Coker({\bf 1}-\varphi_{1,\infty}))^\vee$ is $\La$-torsion. Finally, by the exact sequence
$$0\lr \Ker({\bf 1}-\varphi_{2,\infty})\lr M^2_{1,\infty}\lr M^2_{2,\infty}\lr 0$$
we know that $(\Ker({\bf 1}-\varphi_{2,\infty}))^\vee$ is $\La$-torsion. Hence, the short exact sequence
$$0\lr \Coker({\bf 1}-\varphi_{1,\infty})\lr N^2_\infty\lr \Ker({\bf 1}-\varphi_{2,\infty})\lr 0$$
tells us that $(N^2_\infty)^\vee$ is also $\La$-torsion. For the last assertion, note that the surjections $N^1_n\twoheadrightarrow \Sel_{p^\infty}(A/K_n)$ proved in \cite[\S2.5.2]{KT03} induce a surjection $N^1_\infty\twoheadrightarrow \Sel_p(A/\Kpinf)$ by passing to the inductive limit in $n$. But since $(N^1_\infty)^\vee$ is $\La$-torsion, the same assertion holds for $X_p(A/\Kpinf)$.
\end{proof}

\noindent This completes the proof of Theorem \ref{GIMC1}.

\subsection{The characteristic element} \label{su:charelem}
We denote $Q(\La)$ the field of fractions of $\La$. As mentioned in the introduction, in this section we find it more convenient to use characteristic elements rather than characteristic ideals. This is because we are going to take ratios of characteristic elements and we find it more suggestive to work with $Q(\La)^\times/\La^\times$ rather than with the divisor group of $\La$. So, for any finitely generated $\La$-torsion module $M$, we let $f_M\in \La$ be a characteristic element associated with $M$; $f_M$ is defined uniquely up to $\La^\times$.

We set
\begin{equation} \label{e:charelement}
f_{A/\Kpinf}:=\frac{f_{(N^1_\infty)^\vee}f_{(L^0_\infty)^\vee}}{f_{(N^0_\infty)^\vee}f_{(N^2_\infty)^\vee}f_{(L^1_\infty)^\vee}}\in Q(\La)^\times/\La^\times
\end{equation}
and call it the characteristic element associated with $A/K$ relative to the arithmetic $\ZZ_p$-extension of $K$.

\subsubsection{Arithmetic interpretation} \label{ss:bsdinfty} The characteristic element $f_{A/\Kpinf}$ is related to the arithmetic invariants of $A$ as follows. By considering the
$p$-torsion part of the exact sequence in \cite[\S2.5.2]{KT03} and using the fact that the functor ``take the $p$-primary part'' is exact in the category of finite abelian groups,
we deduce an exact sequence:
\begin{equation} \label{e:alg-arithm0} 0\lr N^0_0\lr A(K)[p^\infty]\lr (\oplus_{v\in Z}\M_v)[p^\infty]\lr N^1_0\lr \Sel_{p^\infty}(A/K)\lr 0, \end{equation}
with $\M_v:=A(K_v)/{\A}(\fm_v)$, where $O_v$ and $\fm_v$ are ring of integers and maximal ideal of $K_v$ and
$${\A}(\fm_v):=\Ker\!\big(A(K_v)= {\A}(O_v)\lr {\A}(k(v))\big).$$
Observe that, since we assumed that $A/K$ has semistable reduction, we can take the group $V_v$ defined in \cite[Proposition 5.13]{KT03} to be equal to ${\A}(\fm_v)$. Since $\A/O_v$ is
smooth we have in fact $\M_v=\A(k(v))$. The finite group $\M_v$ is controlled by the short exact sequence:
$$0\,\lr Q_v\lr \M_v\lr \Phi_v\lr 0,$$
where $\Phi_v$ is the (finite) group of components, $\Phi_v:=({\A}/{\A}^0)(k(v))$, and
$$Q_v={\A}^0(O_v)/({\A}(\fm_v)\cap{\A}^0(O_v))={\A}^0(k(v))$$
by Hensel's lemma. Moreover, since $A/K$ has semistable reduction at $v$, we have a short exact sequence:
$$0\,\lr T_v\lr {\A}^0_v\lr B_v\lr 0,$$
where $T_v$ is a torus and $B_v$ an abelian variety over $k(v)$.\\

Since the N\'eron model functor is stable by \'etale base change, we also have for any $n\geq 0$ an exact sequence:
$$ 0\to N^0_n\to A(K_n)[p^\infty]\to(\oplus_{w\in C_n, w|v\in Z}\M_w)[p^\infty]\to N^1_n\to \Sel_{p^\infty}(A/K_n)\to 0, $$
which induces, by passing to the inductive limit in $n$, an exact sequence:
\begin{equation} \label{e:alg-arithm} 0\to N^0_\infty\to A(\Kpinf)[p^\infty]\to (\oplus_{w\in C_\infty, w|v\in Z}\M_w)[p^\infty]\to N^1_\infty\to \Sel_{p^\infty}(A/\Kpinf)\to 0 \end{equation}
and then, by passing to the Pontryagin dual, an exact sequence of finitely generated torsion $\La$-modules. We set
$$\M_n:=(\oplus_{w\in C_n, w|v\in Z}\M_w)[p^\infty]$$
and $\M_\infty:=\varinjlim \M_n$. By multiplicativity of characteristic elements associated with torsion $\La$-modules, we have:
\begin{equation} \label{alg-arith1} \frac{f_{(N^1_\infty)^\vee}}{f_{(N^0_\infty)^\vee}}=\frac{f_{X_p(A/\Kpinf)}f_{\M_\infty^\vee}^{}}{f_{A(\Kpinf)[p^\infty]^\vee}}. \end{equation}

Let $\zeta_1,\dots,\zeta_l$ be the eigenvalues of the Galois actions of $\Fr_q$ on $T_pA(K^{(p)}_\infty)[p^{\infty}]$. Then by \cite[Proposition 2.3.5]{tan10b} we can write
\begin{equation}\label{e:fvee}
f_{A(K^{(p)}_\infty)[p^{\infty}]^\vee}=\prod_{i=1}^l (1-\zeta_i^{-1}\Fr_q^{-1})\,.
\end{equation}

\begin{lemma} \label{l:minfty} Denote by $g(v)$ the dimension of $B_v$ and let $\beta_{1}^{(v)},..., \beta_{2g(v)}^{(v)}\in \bar\QQ_p $ be the eigenvalues of the Frobenius endomorphism $\tF_{B_v,\,q_v}$ (notations of \S\ref{su:fa}). Then
\begin{equation}\label{e:fMinfty}
f_{\M_\infty^\vee}=\prod_{v\in Z}\prod_{i=1}^{2g(v)} ({\beta_{i}^{(v)}}-\Fr_v^{-1})\,.
\end{equation}
\end{lemma}

\begin{proof} For $v$ a place of $K$, let $\Gamma_v\subset\Gamma$ denote the decomposition group at $v$ and put $\La_v:=\ZZ_p[[\Gamma_v]]$. Thus we get $\La=\oplus_{\sigma\in \Gamma/\Gamma_v}\sigma\La_v\,.$ For each $v\in Z$, choose a $w_0\in C_\infty$ sitting over $v$. Then
$$\M_{\infty, v}:=\bigoplus_{w\in C_\infty, w|v}\M_w[p^\infty]=\bigoplus_{\sigma\in \Gamma/\Gamma_v} \sigma \M_{w_0}[p^\infty] $$
and hence $\M_{\infty, v}^\vee=\La \otimes_{\La_v} \M_{w_0}[p^\infty]^\vee$. In particular, the characteristic element $f_{\M_{\infty, v}^\vee}$ can be chosen to be that of $\M_{w_0}[p^\infty]^\vee$ over $\La_v$. Also, since $\Phi_v$ is finite and $T_v$ is a torus, $\M_{w_0}[p^\infty]^\vee$ is pseudo-isomorphic to $B_v[p^\infty](k_\infty^{(p)})^\vee$.
For each $v$, we order the eigenvalues $\beta_i^{(v)}$ so that $\beta_i^{(v)}$ is a $p$-adic unit if and only if $i\leq f(v)\leq g(v)$.

Let $\Fr_v\colon x\mapsto x^{q_v}$ denote the Frobenius substitution as an element of $\Gal(\overline{k(v)}/k(v))$ (and also, by abuse of notation, the corresponding element of $\Gal(k_\infty^{(p)}/k(v))$ and of $\Gamma_v$). The product $\beta_{1}^{(v)}\cdot...\cdot\beta_{f(v)}^{(v)}$ is a $p$-adic unit and $\prod_{i>f(v)} ({\beta_{i}^{(v)}}-\Fr_v^{-1})$ is a unit in $\La$.
We claim that $\beta_{1}^{(v)},...,\beta_{f(v)}^{(v)}$ are the eigenvalues of the action of $\Fr_v$ on the Tate module $T_pB_v\,$.
Then by \cite[Proposition 2.3.6]{tan10b} we can write
$$ f_{\M_{\infty, v}^\vee}=\prod_{i=1}^{f(v)} ({\beta_{i}^{(v)}}-\Fr_v^{-1})=\prod_{i=1}^{2g(v)} ({\beta_{i}^{(v)}}-\Fr_v^{-1})$$
and \eqref{e:fMinfty} follows from $f_{\M_\infty^\vee}=\prod_{v\in Z} f_{\M_{\infty,v}^\vee}$

We have to prove the claim. Let $\rho\colon\mathrm{End}_{k(v)}(B_v)\rightarrow \mathrm{End} (T_pB_v)$ denote the $p$-adic representation. Then for every $f\in \mathrm{End}_{k(v)}(B_v)$ the eigenvalues of $\rho(f)$, counting multiplicities, are a portion of those of $f$ (this can be seen e.g.~mimicking the argument in the proof of \cite[Theorem 12.18]{gm13}, with $\Ker(f)$ replaced by its maximal \'etale subgroup). In particular, we can rearrange the order of the $\beta_i^{(v)}$'s so that, for every positive integer $N$, the eigenvalues of $\rho(\tF_{B_v,q_v}^N)=\Fr_v^N$ equal $(\beta_1^{(v)})^N,...,(\beta_{h(v)}^{(v)})^N$ for some $h(v)\leq f(v)$. Let $k(v)_N$ denote the degree $N$ extension of $k(v)$. Letting $\equiv_p$ denote congruence modulo $p$-adic units, we have
$$\prod_{i=1}^{h(v)}(1-(\beta_{i}^{(v)})^N)\equiv_p|B_v[p^\infty](k(v)_N)|\equiv_p\prod_{i=1}^{2g(v)}(1-(\beta_{i}^{(v)})^N)\equiv_p\prod_{i=1}^{f(v)}(1-(\beta_{i}^{(v)})^N)\,,$$
where the second equality is from Weil's formula and the third is from the fact that $1-(\beta_{i}^{(v)})^N$ is a $p$-adic unit if $i>f(v)$. Taking $N$ such that $(\beta_{i}^{(v)})^N\equiv 1$ mod $p$ for all $i<f(v)$, we deduce $h(v)=f(v)$.
\end{proof}

Write $\AA_K$ for the adelic ring. Let $\mu=(\mu_v)_v$ be the Haar measure on $Lie({\A})( \AA_K)$ such that
$$\mu_v(Lie({\cal A})(O_v)):=1$$ for every $v$ and let $\alpha_v$ denote the Haar measure on $A(K_v)=\A(O_v)$ such that for $n\geq 1$
$$\alpha_v(\A(\fm_v^n)):=\mu_v (Lie({\cal A})(\fm_v^n)).$$
Then since $|\A(O_v)/\A(\fm_v)|=|\M_v|$ and $Lie({\cal A})(O_v)/Lie({\cal A})(\fm_v)\simeq (O_v/\fm_v)^g$, we have
\begin{equation}\label{e:alpgaM}
\alpha_v(\A(O_v))=|\M_v|\cdot q_v^{-g}.
\end{equation}
By \cite[p.552]{KT03} we have the relation:
\begin{equation} \label{e:tamagawa} |\M_0|\cdot |L^0_0|\cdot |L^1_0|^{-1}= \mu(Lie(\A)( \AA_K)/Lie(\A)( K))^{-1}\cdot\prod_{v\in Z}\alpha_v(\A(O_{v}))\,. \end{equation}
Thus, by \eqref{e:alpgaM} and \eqref{e:tamagawa}
\begin{equation}\label{e:tamagawa2} |L^0_0|\cdot|L^1_0|^{-1}=q^{-g\deg (Z)}\cdot \mu(Lie({\cal A})( \mathbb{A}_K)/ Lie({\cal A})( K))^{-1}. \end{equation}

Next, we choose a a basis $e_1,...,e_g$ of the $K$-vector space $Lie(A)(K)=Lie({\A})(K)$. Then for every $v$ the exterior product $e:=e_1\wedge\cdots\wedge e_g$
determines the Haar measure $\mu_v^{(e)}$ on $ Lie({\A})(K_v)$ that has measure $1$ on the compact subset $\fL(O_v):=O_ve_1+\cdots + O_ve_g$. Similarly, if we choose a a basis $f_1,...,f_g$ of $Lie({\A})(O_v)$ over $O_v$, then the exterior product $f_v:=f_{1 v}\wedge\cdots\wedge f_{g v}$ actually determines the Haar measure $\mu_v$. Define the number $\delta$ by
\begin{equation}\label{e:deltaratio}
q^{-\delta}:=\prod_{\text{all}\; v} \frac{\mu_v^{(e)}(Lie({\A})(O_v))}{\mu_v(Lie({\A})(O_v))}= \prod_{\text{all}\; v} \mu_v^{(e)}(Lie({\A})(O_v))\,, \end{equation}
so that the Haar measures $\mu$ and $\mu^{(e)}$ are related by $\mu^{(e)}=q^{-\delta}{\mu}\,.$
\footnote{If $g=1$ and $\Delta$ denote the global discriminant, then $\delta=\frac{\deg (\Delta)}{12}$ (see e.g.~ \cite[eq.~(9)]{tan95v}).}
By a well-known computation (see e.g.~\cite[VI, Corollary 1 of Theorem 1]{We73}) one finds
$$\mu^{(e)}(Lie({\cal A})( \mathbb{A}_K)/Lie({\cal A})(K))=q^{g(\kappa-1)},$$
with $\kappa$ the genus of $C/\FF$, whence, by \eqref{e:tamagawa2} and \eqref{e:deltaratio},
\begin{equation}\label{e:tamagawa3}
|L^0_0|\cdot|L^1_0|^{-1}=q^{-g(\deg (Z)+\kappa-1)-\delta}.
\end{equation}

\begin{lemma}\label{l:l0l1}
Under the above notation we can write
$$\frac{f_{(L_\infty^0)^\vee}}{f_{(L_\infty^1)^\vee}}=q^{-g(\deg (Z)+\kappa-1)-\delta}.$$
\end{lemma}

\begin{proof}
Since $L_\infty^i\simeq L_0^i\otimes_{\ZZ_p} W(k_\infty^{(p)})$, the lemma follows from \eqref{e:tamagawa3} and Lemma \ref{linalg1} (as well as its proof).
\end{proof}

Finally, by \cite[2.5.3]{KT03}, we have for any $n\geq 0$ an isomorphism:
\begin{equation} \label{e:KT253} N^2_n\simeq \Sel_{\ZZ_p}(A^t/K_n)^\vee, \end{equation}
where $\Sel_{\ZZ_p}(.):=\varprojlim \Sel_{p^n}(.)$ denotes the compact Selmer group as in \cite[\S2.3]{KT03}. These isomorphisms induce, when passing to the inductive limit, an isomorphism
\begin{equation} \label{e:N2compactSel} (N^2_\infty)^\vee\simeq \Sel_{\ZZ_p}(A^t/\Kpinf):=\ilim \Sel_{\ZZ_p}(A^t/K_n)\,. \end{equation}
Let $T_p(A(K^{(p)}_\infty)):=\varprojlim A(K^{(p)}_\infty)[p^{m}]$ denote the Tate-module of $A(K^{(p)}_\infty)$.

\begin{proposition} \label{p:N2trivial}
If $A_{p^\infty}(\Kpinf)$ is a finite group, then $N^2_\infty=0$.
In general, we can write
$$f_{(N^2_{\infty})^\vee}=f_{T_p(A^t(K^{(p)}_\infty))}.$$
\end{proposition}

\begin{proof} For simplicity denote $D_n:=A^t(K_n)[p^{\infty}]$. Also, write $V_n:=\varprojlim_m \Sel_{p^{\infty}}(A^t/K_n)[p^m]$. Since $\Sel_{p^{\infty}}(A^t/K_n)$ is cofinitely generated over $\ZZ_p\,$, actually
$$V_n=\varprojlim_m \Sel_{p^{\infty}}(A^t/K_n)_{div}[p^m].$$

For each $m$ we have the exact sequence
$$\xymatrix{0 \ar[r] & D_n/p^m D_n \ar[r] & \Sel_{p^{m}}(A^t/K_n) \ar[r] & \Sel_{p^{\infty}}(A^t/K_n)[p^m]\ar[r] & 0},$$
which induces, by taking $m\rightarrow \infty$, the exact sequence
$$\xymatrix{0 \ar[r] & D_n \ar[r] & \Sel_{\ZZ_p}(A^t/K_n) \ar[r] & V_n\ar[r] & 0}.$$
Since $\Sel_{p^\infty}(A^t/L)$ is cotorsion over $\La$, the divisible subgroup $\Sel_{div}(A^t/L)$ defined in \S\ref{se:se} must be cofinitely generated over $\ZZ_p\,$. Thus, for $n$ sufficiently large the restriction map $\Sel_{p^{\infty}}(A^t/K_n)_{div}\rightarrow \Sel_{div}(A^t/L)$ is surjective. By Lemma \ref{l:h1d}, the kernel of this map is finite. It follows that if $n$ is sufficiently large then the restriction map $\Sel_{p^{\infty}}(A^t/K_n)_{div}\rightarrow \Sel_{p^{\infty}}(A^t/K_r)_{div} $ is an isomorphism for every $r>n$, and hence $V_r$ can be identified with $V_n$. This implies $\varprojlim_n V_n=0$ as the map $V_r\rightarrow V_n$ becomes multiplication by $p^{r-n}$ on $V_r$ for sufficiently large $r$, $n$. Hence, \eqref{e:N2compactSel} and the above exact sequence yield
$$(N^2_\infty)^\vee=\varprojlim_n \Sel_{\ZZ_p}(A^t/K_n)=\varprojlim_n A^t(K_n)[p^{\infty}].$$
This proves the first assertion, while the second follows from \cite[Proposition 2.3.5]{tan10b}.
\end{proof}

\begin{remark}{\em An alternative proof of the first assertion of Proposition \ref{p:N2trivial} can be obtained following the same argument used in the proof of \cite[Lemma 5]{P87}. Also, as we already mentioned in Remark \ref{r:fintors}, the condition that $A_{p^\infty}(\Kpinf)$ is a finite group is often satisfied. }
\end{remark}

Since $A$ and $A^t$ are isogenous, the Galois actions of $\Fr_q$ on both $T_p(A(K^{(p)}_\infty))$ and $T_p(A^t(K^{(p)}_\infty))$ have the same eigenvalues.
Then, by \cite[Proposition 2.3.5]{tan10b} again, we can write
\begin{equation}\label{e:ftp}
f_{T_p(A^t(K^{(p)}_\infty))}=\prod_{i=1}^l (1-\zeta_i^{-1}\Fr_q)\,.
\end{equation}

\subsubsection{The value of $\star_{A,L}$} Summarizing the above, we can make more precise the statement of Theorem \ref{t:summary}. Recall that, in the notation of our Introduction, $c_{A/\Kpinf}=f_{X_p(A/K_{\infty}^{(p)})}\,$.

\begin{proposition}\label{p:summarize} Let notation be as above. Then we can write
$$f_{A/K_{\infty}^{(p)}}=\star_{A,K_{\infty}^{(p)}}\cdot c_{A/\Kpinf}.$$
with
$$\star_{A,K_{\infty}^{(p)}}=\frac{q^{-g(\deg(Z)+\kappa-1)-\delta}\cdot\prod_{v\in Z}\prod_{i=1}^{2g(v)} (\beta_i^{(v)}- \Fr_v^{-1})}{\prod_{i=1}^l(1-\zeta_i^{-1}\Fr_q)(1-\zeta_i^{-1}\Fr_q^{-1})}.$$
\end{proposition}

\section{The semistable case} \label{s:semistab}

In this section we give a geometric analogue of the Iwasawa Main Conjecture of abelian varieties with semistable reduction. We keep the notations and the hypotheses of Section \ref{s:genord}.

\subsection{Interpolation and the Main Conjecture}

\subsubsection{The modules $P^i_\infty$} \label{ss:Pinfty}
Let $\QQ_{p,n}$ denote the fraction field of $W(k_n)$ and $\QQ_{p,\infty}:=\cup_n \QQ_{p,n}$.

\begin{lemma}\label{induce} For any $i$, the map
$$\big({\bf 1}\otimes W(\kpinf)\big)\!^\vee\left[\frac{1}{p}\right]\colon(M^i_{2,\infty})^\vee\left[\frac{1}{p}\right]\lr (M^i_{1,\infty})^\vee\left[\frac{1}{p}\right]$$
is an isomorphism induced by the identity on $(P^i_0[\frac{1}{p}]\otimes_{\QQ_p} \QQ_{p,\infty})^\vee$.
\end{lemma}

\begin{proof} The first assertion follows from the exact sequence \eqref{L2} and from Corollary \ref{M,L} (2). To show that the map is induced by the identity on $(P^i_0[\frac{1}{p}]\otimes_{\QQ_p}
\QQ_{p,\infty})^\vee$, note that, by Lemma \ref{cofinite} (1), for all $n$ and for $k=1,2$ there exist some exact sequences:
$$0\to B^i_0\otimes_{\ZZ_p} W(k_n)\to P^i_0[p^{-1}]\otimes_{\QQ_p} \QQ_{p,n}\to M^i_{k,0}\otimes_{\ZZ_p} W(k_n)\to C^i_0\otimes_{\ZZ_p} W(k_n) \to 0$$
where $B^i_0$ is a lattice of $P^i_0[\frac{1}{p}]$ and $C^i_0$ a finite abelian $p$-group.

Passing to the inductive limit in $n$, then to the Pontryagin dual, and finally by inverting $p$, we obtain an exact sequence:
\begin{equation}\label{exact1} 0\lr (M^i_{k,\infty})^\vee[p^{-1}]\lr (P^i_0[p^{-1}]\otimes_{\QQ_p}\QQ_{p,\infty})^\vee \lr (\dlim B^i_n)^\vee[p^{-1}]\lr 0,\end{equation}
since $\big(\varinjlim C^i_0\otimes_{\ZZ_p} W(k_n)\big)^\vee$ is $p$-torsion (actually, by Lemma \ref{linalg1}, it is a finite direct sum of $(\ZZ/p^{n_j}\ZZ)[[\Gamma]]$).
In particular, $(M^i_{k,\infty})^\vee[\frac{1}{p}]$ is a submodule of $(P^i_0[\frac{1}{p}]\otimes_{\QQ_p} \QQ_{p,\infty})^\vee$ for $k=1,2$ and the map
$({\bf 1}\otimes W(\kpinf))^\vee[\frac{1}{p}]$ is compatible with the identity on $(P^i_0[\frac{1}{p}]\otimes_{\QQ_p} \QQ_{p,\infty})^\vee$ by Lemma \ref{cofinite}, (2).
\end{proof}

As a consequence of the previous lemma, we denote $P^i_\infty$ the module $(M^i_{2,\infty})^\vee[\frac{1}{p}]$, endowed with the operator $\Phi_i:=\big({\bf 1}\otimes W(\kpinf))^\vee[\frac{1}{p}]\big)^{-1}\circ\big((\varphi_{i,0}\otimes\Fr_q)^\vee[\frac{1}{p}]\big)$. First, we define the $p$-adic $L$-function.

\subsubsection{The $p$-adic $L$-function} \label{su:padicL} By Corollary \ref{M,L}, (1), we know that the $P^i_\infty$'s are free $\La[\frac{1}{p}]$-modules of finite rank. We set
$$\cL_{A/\Kpinf}:=\prod_{i=0}^2 {\det}_{\La[\frac{1}{p}]}\left(id-\Phi_i,P^i_\infty\right)^{(-1)^{i+1}}$$
and call it the $p$-adic $L$-function associated with $A/K$ relative to the arithmetic $\ZZ_p$-extension of $K$.



\begin{lemma}\label{interpol} Let $E$ be a finite extension of $\QQ_p$ and $\omega\colon\Gamma\to E^\times$ an Artin character factoring through $\Gamma_N$. Then we have
$$\omega(\cL_{A/\Kpinf})=\prod_{i=0}^2 {\det}_{E}\left(id-p^{-1}F_{i,0}\otimes m_{\omega(\Fr_q)},P^i_0[\frac{1}{p}]\otimes_{\QQ_p}E\right)^{(-1)^{i+1}}$$
where $m_{\omega(\Fr_q)}\colon E\to E$ is the multiplication by $\omega(\Fr_q)$.
\end{lemma}

\begin{proof} We have for any $i$,
$$\omega({\det}_{\La[\frac{1}{p}]}(id-\Phi_i,P^i_\infty))={\det}_E(id-\Phi_i\otimes id_E,P^i_\infty\otimes_{\La[\frac{1}{p}]}E).$$
(where, as usual, $E$ is a $\La[\frac{1}{p}]$-module via $\omega$).

Since $\omega$ factors through $\Gamma_N$, we have
\begin{eqnarray*}P^i_\infty\otimes_{\La[\frac{1}{p}]}E&=&
P^i_\infty\otimes_{\La[\frac{1}{p}]}\QQ_p[\Gamma_N]\otimes_{\QQ_p[\Gamma_N]}E\\
                                   &=&(P^i_\infty)_{\Gamma^{(N)}}\otimes_{\QQ_p[\Gamma_N]}E\\
                                   &=&\left((M^i_{k,0}\otimes W(\kpinf))^{\Gamma^{(N)}}\right)\!^\vee[p^{-1}]\otimes_{\QQ_p[\Gamma_N]}E\\
                                   &=&(M^i_{k,0}\otimes W(k_N))^\vee[p^{-1}]\otimes_{\QQ_p[\Gamma_N]}E\\
                                   &=&(P^i_0[p^{-1}]\otimes_{\QQ_p}\QQ_{p,N})^\vee\otimes_{\QQ_p[\Gamma_N]}E,
\end{eqnarray*}
where the last equality is a consequence of Lemma \ref{cofinite}. An operator and its dual share the same determinant: hence we are reduced to compute ${\det}_E((id-p^{-1}F_{i,0}\otimes\Fr_q) \otimes id_E)$ on $P^i_0[\frac{1}{p}]\otimes_{\QQ_p}\QQ_{p,N}\otimes_{\QQ_p[\Gamma_N]}E\,.$ By the normal basis theorem $\QQ_{p,N}\otimes_{\QQ_p[\Gamma_N]}E$ endowed with its $E$-endomorphism $\Fr_q\otimes id_E$
is isomorphic to $E$ endowed with the endomorphism $m_{\omega(\Fr_q)}$ and so the assertion is clear.
\end{proof}

\subsubsection{Twisted Hasse-Weil $L$-function} \label{ss:twistedHW} Let $\omega\colon\Gamma\to E_0^\times$ be any character factoring through $\Gamma_N$, with $E_0$ some totally ramified finite extension of $\QQ_{p,0}$ endowed with a Frobenius operator $\sigma$ which acts trivially on $\QQ_{p,0}$ and on $\omega(\Gamma)$. Then we can see $\omega$ as a one-dimensional $E_0$-representation of the fundamental group of $U$, having finite local monodromy. By \cite[Theorem 7.2.3]{Ts98} this representation corresponds to a unique constant unit-root overconvergent isocrystal $U(\omega)^\dagger$ over $\FF/E_0$ endowed with $m_\omega(\Fr_q)$ as Frobenius operator. Let
$$pr_1^*\colon F^a\hbox{-}iso^\dagger(U_{/\FF}/\QQ_{p,0})\lr F^a\hbox{-}iso^\dagger(U_{/\FF}/E_0)$$
and
$$pr_2^*\colon F^a\hbox{-}iso^\dagger(\FF/E_0)\lr F^a\hbox{-}iso^\dagger(U_{/\FF}/E_0),$$
denote the two restriction functors in the categories of overconvergent isocrystals endowed with Frobenius. From $I^\dagger\in F^a\hbox{-}iso^\dagger(U_{/\FF}/\QQ_{p,0})$ we obtain the $F^a$-isocrystal $pr_1^*I^\dagger\otimes pr_2^*U(\omega)^\dagger$ endowed with the natural Frobenius induced by the Frobenius on $I^\dagger$ and the one on $U(\omega)^\dagger$.

\begin{definition} Let $I^\dagger\in F^a\hbox{-}iso^\dagger(U_{/\FF}/\QQ_{p,0})$. Then we set $$L(U,I^\dagger,\omega,t):=L(U,pr_1^*I^\dagger\otimes pr_2^*U(\omega)^\dagger,t)$$ where the right hand term is the classical $L$-function associated with the $F^a$-isocrystal $pr_1^*I^\dagger\otimes pr_2^*U(\omega)^\dagger$, as defined in \cite{EL93}. We call the function $L(U,I^\dagger,\omega,t)$ the $\omega$-twisted $L$-function of $I^\dagger$.
\end{definition}

Recall (\cite[Th\'eor\`eme 6.3]{EL93}) that this is a rational function in the variable $t$ and we have
$$L(U,I^\dagger,\omega,t)=\prod_{i=0}^2\det(1-t\varphi_i,\coh^i_{rig,c}(U/E_0,pr_1^*I^\dagger\otimes pr_2^*U(\omega)^\dagger))^{(-1)^{i+1}}.$$

\subsubsection{Interpolation} Recall (\cite[IV]{KT03}) that $D$, the log Dieudonn\'e crystal associated with our semistable abelian variety $A/K$, induces an overconvergent $F^a$-isocrystal $D^\dagger$ over $U/\QQ_{p,0}$ and that we have a canonical isomorphism:
\begin{equation} \label{e:PHrig} P^i_0[p^{-1}]\simeq \coh^i_{rig,c}(U/\QQ_{p,0},D^\dagger) \end{equation}
compatible with the Frobenius operators.

Moreover, we have, by \cite[3.2.2]{KT03},
$$L(U,D^\dagger,q^{-s})=L_Z(A,s),$$
where $L_Z(A,s)$ is the Hasse-Weil $L$-function of $A$ without Euler factors outside $U$, as in \S\ref{s:interpol}. In fact, more generally, one can show that for any character $\omega\colon\Gamma\to \CC^\times$ we have
\begin{equation} \label{e:LD+=LA} L(U,D^\dagger,\omega,q^{-s})=L_Z(A,\omega,s), \end{equation}
since the Euler factors on both sides can be written as $\prod_{v\in U} (1-\omega([v])\varepsilon_{i,v}q^{-s})^{-1}$, where the $\varepsilon_{i,v}$'s are the eigenvalues of the arithmetic Frobenius at $v$ acting on $T_\ell(A)$ (or, equivalently, of the geometric Frobenius acting on the fibre at $v$ of $D^\dagger$). These eigenvalues don't depend on $\ell$, as results of \cite{KM} (for the details see e.g. the proof of \cite[Corollaire 1.4]{Tr}, where this independence is used to deduce the equality of different definitions of $L$-functions).

The K\"unneth formula for rigid  cohomology (formula (1.2.4.1) in \cite{Ke06}) implies:

\begin{lemma} Let $\omega\colon\Gamma\to E_0^\times$ be as in \S\ref{ss:twistedHW}. There is an isomorphism of $E_0$-vector spaces compatible with Frobenius operators:
$$\coh^i_{rig,c}(U/E_0,pr_1^*D^\dagger\otimes pr_2^*U(\omega)^\dagger)\simeq \coh^i_{rig,c}(U/\QQ_{p,0},D^\dagger)\otimes E_0.$$
\end{lemma}

\noindent Together with Lemma \ref{interpol}, \eqref{e:PHrig} and \eqref{e:LD+=LA}, this immediately yields the following analogue of ({\bf IMC2}):

\begin{mytheorem}\label{GIMC2} For any Artin character $\omega\colon\Gamma\to \bar\QQ_p^\times$, we have
\begin{equation} \label{e:GIMC2} \omega(\cL_{A/\Kpinf})=L_Z(A,\omega,1). \end{equation}
\end{mytheorem}

\begin{remark}{\em In order to discuss $\omega(\cL_{A/\Kpinf})$, one needs to know that the denominator of $\cL_{A/\Kpinf}$ is not killed by $\omega$. Actually, we are going to see (formula \eqref{e:detfr} below) that $\cL_{A/\Kpinf}$ is an alternating product of terms $1-\alpha_{ij}\Fr_q$. By \cite[Theorem 5.4.1]{Ke06b}, the coefficients $\alpha_{ij}$ are Weil numbers of weight respectively $-1$ (for $i=0$) and 1 (for $i=2$): in particular their complex absolute values do not include 1. Hence the left-hand side of \eqref{e:GIMC2} is well defined.}\end{remark}

\subsubsection{The Main Conjecture} Finally, we prove the analogue of ({\bf IMC3}) in this setting.

\begin{proof}[{\bf Proof of Theorem \ref{GIMC3}}]
For any morphism $g\colon M\to N$ of $\La$-modules whose kernel and cokernel are both torsion $\La$-modules, we denote by $char(g)$ the element $$f_{\Coker(g)}\cdot f_{\Ker(g)}^{-1}\in Q(\La)^\times/\La^\times.$$
Dualizing the exact sequence \eqref{e:sesNi} we get
$$0\lr (\Ker({\bf 1}-\varphi_{i,\infty}))^\vee\lr (N^i_\infty)^\vee \lr (\Coker({\bf 1}-\varphi_{i-1,\infty}))^\vee \lr 0\,$$
which, remembering that $\Ker(h^\vee)=(\Coker(h))^\vee$, implies
$$f_{(N^i_\infty)^\vee}=f_{\Coker(({\bf 1}-\varphi_{i,\infty})^\vee)}f_{\Ker(({\bf 1}-\varphi_{i-1,\infty})^\vee)}\,.$$
Similarly, \eqref{L2} yields $f_{(L^i_\infty)^\vee}=f_{\Coker({\bf 1}_i^\vee)}f_{\Ker({\bf 1}_{i-1}^\vee)}$. Replacing in \eqref{e:charelement}, we obtain
$$f_{A/\Kpinf}=\frac{f_{\Coker(({\bf 1}-\varphi_{1,\infty})^\vee)}f_{\Ker(({\bf 1}-\varphi_{0,\infty})^\vee)}}{f_{\Coker(({\bf 1}-\varphi_{0,\infty})^\vee)}f_{\Coker(({\bf 1}-\varphi_{2,\infty})^\vee)}f_{\Ker(({\bf 1}-\varphi_{1,\infty})^\vee)}}\cdot\frac{f_{\Coker({\bf 1}_0^\vee)}}{f_{\Coker({\bf 1}_1^\vee)}f_{\Ker({\bf 1}_{0}^\vee)}}\,.$$
Since $\Ker(({\bf 1}-\varphi_{i,\infty})^\vee)$ and $\Coker(({\bf 1}-\varphi_{i,\infty})^\vee)$ are $\La$-torsion modules by Proposition \ref{0case} and Corollary \ref{1case}, this can be rewritten as
$$f_{A/\Kpinf}=\frac{char(({\bf 1}-\varphi_{1,\infty})^\vee)\cdot char({\bf 1}_1^\vee)^{-1}}{char(({\bf 1}-\varphi_{0,\infty})^\vee)\cdot char({\bf 1}_0^\vee)^{-1}\cdot char(({\bf 1}-\varphi_{2,\infty})^\vee)\cdot char({\bf 1}_2^\vee)^{-1}}$$
On the other hand, $\cL_{A/\Kpinf}$ is defined as an alternating product of determinants of
$$id-\Phi_i=({\bf 1}_i^\vee)^{-1}\circ({\bf 1}^\vee-\varphi_{i,\infty})\,.$$
Thus Theorem \ref{GIMC3} becomes an immediate consequence of the following lemma (whose proof is an easy exercise which we omit):
\begin{lemma} Let $g,h\colon M\rightarrow N$ be two homomorphisms of finitely generated $\La$-modules with torsion kernel and cokernel: then
$${\det}_{Q(\La)}(g_{Q(\La)}h_{Q(\La)}^{-1})=char(g)char(h)^{-1}\,.$$
\end{lemma}
\end{proof}

\subsection{Euler characteristic} \label{su:Eulerchar}
After identifying $\La$ with $\ZZ_p[[T]]$, the characteristic element of $M$ can be written $f_M(T)=T^rf(T)$, with $f(T)\in \ZZ_p[[T]]$ such that $T$ does not divide $f(T)$. We call $r$ the order of $f_M$ and $f(0)$ the leading term of $f_M$ (note that $f(0)$ is defined up to $\ZZ_p^\times$, since the choice of a different isomorphism $\La\simeq\ZZ_p[[T]]$ changes $T$ by a unit in $\La$).
In the following, we are going to compute order and leading term of $\cL_{A/\Kpinf}$ and compare them with those of the classical $L$-function.

\subsubsection{Generalized Euler characteristic} We recall the definition of the generalized $\Gamma$-Euler characteristic. Let $M$ be a finitely generated torsion $\La$-module and let $g_M\colon M^\Gamma\to M_\Gamma$ denote the composed map $M^\Gamma\hookrightarrow M\to M_\Gamma$, where the first map is the canonical inclusion and the second map the canonical projection. Then we say that $M$ has finite generalized $\Gamma$-Euler characteristic, denoted $char(\Gamma,M)$, if $\Ker(g_M)$ and $\Coker(g_M)$ are finite groups and in this case we set
$$char(\Gamma,M):=\frac{|\Coker(g_M)|}{|\Ker(g_M)|}.$$
By the identifications $(M^\vee)^\Gamma=(M_\Gamma)^\vee$ and  $(M^\vee)_\Gamma=(M^\Gamma)^\vee$, we see that $g_{M^\vee}$ is the dual of $g_M$ and hence
\begin{equation} \label{e:EulerMdual} char(\Gamma,M^\vee)=char(\Gamma,M)^{-1} \end{equation}
if one of them is defined.

\subsubsection{Twisted Euler characteristic} \label{ss:twistEuler}
Let $\omega\colon\Gamma\to \cO^\times$ be an Artin character, with $\cO$ the ring of integers of some finite extension of $\QQ_p$, and $M$ be a finitely generated torsion $\La$-module. Let $\omega^*\colon \La_\cO\rightarrow \La_\cO$, $\gamma\mapsto \omega(\gamma)^{-1}\gamma$
be the automorphism defined in \S\ref{ss:iwH1}(H2), let $\cO(\omega)$ be the module defined in \S\ref{su:twistmodule}, and denote $M_\cO(\omega):=\cO(\omega)\otimes_{\ZZ_p}M$.
Then $M_\cO(\omega)$ has again a structure of finitely generated torsion $\La$-module.

Assuming that $M_\cO(\omega)$ has finite generalized $\Gamma$-Euler characteristic, we denote $\pounds_\omega(f_M)$ the leading term of $f_{M_\cO(\omega)}$ and $ord_\omega(f_M)$ the order of $f_{M_\cO(\omega)}$. We have the following result (compare \cite[Lemma 2.11]{zer09} and also \cite[Prop.~3.19]{BV06}):

\begin{lemma}\label{recall} Let $M$ be a finitely generated torsion $\La$-module with characteristic element $f_M\in \La/\La^\times$ and let $\omega\colon\Gamma\to \cO^\times$ be a character. Let $d_\cO:=\left[\cO:\ZZ_p\right]$. Then
$$\rank_{\ZZ_p}\left(M_\cO(\omega)^\Gamma\right)=\rank_{\ZZ_p}\left(M_\cO(\omega)_\Gamma\right)\leq ord_\omega\left(f_M\right)=ord\left(\omega^*(f_M)\right),$$
with equality if and only if $M_\cO(\omega)$ has finite generalized $\Gamma$-characteristic and in this case we have
$$char\left(\Gamma,M_\cO(\omega)\right)=|\pounds_\omega(f_M)|_p^{-d_\cO}=|\omega^*(f_M)(0)|_p^{-d_\cO}.$$
\end{lemma}

\begin{proof} By Lemma \ref{l:phi[]chi}, if $M\sim\La_\cO/f\La_\cO$ then $M_\cO(\omega)\sim\La_\cO/\omega^*(f)\La_\cO$. It is an easy exercise to check that if $M$ and $N$ are pseudo-isomorphic $\La_\cO$-modules, then they have the same Euler characteristic (for a hint, see \cite[Lemma 3.5]{css03}). Besides, the Euler characteristic is multiplicative: hence we are reduced to compute it for the case $M=\La_\cO/f\La_\cO$, with $f$ a power of some prime $\xi\in\La_\cO\simeq\cO[[T]]$, and in the rest of the proof we will assume we are in this situation. Then if $f=T$ the map $g_M$ is the identity, while if $f=T^i$ for some $i>1$ we have $g_M=0$ and $M^\Gamma\simeq\cO\simeq M_\Gamma$. Finally, if $f$ is coprime with $T$ we get $M^\Gamma=0$ and
$$M_\Gamma\simeq\La_\cO/(f,T)\simeq\cO/f(0)\cO.$$
Now just remember that, by basic number theory, $|\cO/x\cO|=|x|_p^{-d_\cO}$ for any $x\in\cO$.
\end{proof}

\begin{lemma}\label{L(chi)} Let $\omega\colon\Gamma\to \cO^\times$ and $d_\cO$ be as in the previous lemma. Then, for $j=0,1$, we have
$$char\left(\Gamma,(L^j_\infty)^\vee(\omega)\right)=p^{d_\cO d(L^j_0)}$$
and
$$\rank_{\ZZ_p}\left(\left((L^j_\infty)^\vee(\omega)\right)^\Gamma\right)=0.$$
\end{lemma}

\begin{proof} By Corollary \ref{M,L}, (2), for $j=0,1$, $(L^j_\infty)^\vee$ is a finite direct sum of copies of $\La/p\La$ and therefore $(L^j_\infty)^\vee(\omega)$ is a finite direct sum of copies of $\La_{\cO}/p\La_{\cO}$ and the assertion is clear.
\end{proof}

We deduce now from Lemma \ref{recall}, Lemma \ref{L(chi)}, Theorem \ref{GIMC3} and \eqref{e:tamagawa3} the following result:

\begin{mytheorem} \label{chi-Iwasawa-BSD} Let $\omega\colon\Gamma\to \cO^\times$ be a character.
Assume that, for $i=0,1,2$, the $\La$-modules $(N^i_\infty)^\vee(\omega)$ have finite generalized $\Gamma$-Euler characteristic. Then
$$ord_\omega\big(\cL_{A/\Kpinf}\big)=\sum_{i=0}^2 (-1)^{i+1}\rank_{\ZZ_p}\left((N^i_\infty)^\vee(\omega)\right)^\Gamma$$
and
$$|\pounds_\omega(\cL_{A/\Kpinf})|_p^{-d_\cO}= q^{-d_\cO(\delta+g(\deg(Z)+\kappa-1))} \prod_{i=0}^2 char\left(\Gamma,(N^i_\infty)^\vee(\omega)\right)^{(-1)^{i+1}}.$$
\end{mytheorem}

If $\omega$ is the trivial character, we can obtain more precise results. We consider first the problem of finiteness of the generalized $\Gamma$-Euler characteristic.

\subsubsection{Hochschild-Serre spectral sequence} \label{specseq}  Since $\Gamma$ has cohomological dimension one, the natural Hochschild-Serre spectral sequence
\begin{equation} \label{e:HochSerre} \coh^i(\Gamma,N^j_\infty)\Rightarrow N_0^{i+j} \end{equation}
induces (\cite[Appendix B]{mil80}) the following two exact sequences,
\begin{equation}\label{ses0}
0\lr (N^0_\infty)_\Gamma\lr N^1_0 \buildrel{\beta}\over\lr (N^1_\infty)^\Gamma\lr 0
\end{equation}
and
\begin{equation}\label{ses1}
0\lr (N^1_\infty)_\Gamma \buildrel{\gamma}\over\lr N^2_0\lr (N^2_\infty)^\Gamma\lr 0.
\end{equation}

\begin{lemma}\label{N02} The groups $(N^i_\infty)_\Gamma$ and $(N^i_\infty)^\Gamma$ are finite for $i=0$ and $i=2$. In particular, we have
\begin{equation} \label{e:ranks} \rank_{\ZZ_p}(N^1_0)^\vee = \rank_{\ZZ_p}\big((N^1_\infty)^\vee\big)_\Gamma = \rank_{\ZZ_p}(X_p(A/K)). \end{equation}
\end{lemma}

\begin{proof} Recall that invariants and coinvariants of a $\La$-module have the same rank. For $i=0$, observe that we have, by \eqref{e:alg-arithm},
$$\rank\big((N^0_\infty)^\vee\big)_\Gamma=\rank\big((N^0_\infty)^\Gamma\big)^\vee\leq \rank\big(A_{p^\infty}(\Kpinf)^\Gamma\big)^\vee =\rank\big(A(K)[p^\infty]\big)^\vee=0\,.$$
From \eqref{e:alg-arithm} we get the exact sequence
$$0 \lr X_p(A/\Kpinf) \lr (N^1_\infty)^\vee \lr \M_\infty[p^\infty]^\vee.$$
Keeping the notations of \S\ref{ss:bsdinfty}, we have
$$\rank_{\ZZ_p}\big((\M_\infty)^\vee\big)^\Gamma =\rank_{\ZZ_p}\big(\oplus_{w|v}\big(B_w(k(w))[p^\infty]^\vee\big)^\Gamma\big)=0$$
where the first equality results from the facts that $\Phi_w$ is a finite group and $(T_w)[p^\infty]$ is trivial. Therefore we have
$$\rank_{\ZZ_p}\left((N^1_\infty)^\vee\right)^\Gamma = \rank_{\ZZ_p}X_p(A/\Kpinf)^\Gamma = \rank_{\ZZ_p} X_p(A/K) \,,$$
where the last equality is a consequence of the control theorem \cite[Theorem 4]{tan10a}. This yields \eqref{e:ranks}.
On the other hand, for the group $(N^2_\infty)^\Gamma$, we have
$$\rank_{\ZZ_p}(N^2_0)^\vee=\rank_{\ZZ_p}\Sel_{\ZZ_p}(A^t/K)\,,$$
by \eqref{e:N2compactSel}.
Comparing  \eqref{e:selmersha} with the Pontryagin dual of the exact sequence \eqref{e:selmersha}, we get
$$\rank_{\ZZ_p}\Sel_{\ZZ_p}(A^t/K)=\rank A^t(K)+\rank_{\ZZ_p}\Sha_{p^\infty}(A^t/K)^\vee=\rank_{\ZZ_p}X_p(A^t/K),$$
In particular,
$$\rank_{\ZZ_p}(N^2_0)^\vee = \rank_{\ZZ_p}X_p(A^t/K) = \rank_{\ZZ_p}X_p(A/K)\,,$$
where the second equality results from the existence of an isogeny between $A$ and $A^t$. Hence, from the short exact sequence \eqref{ses1} we deduce $\rank_{\ZZ_p}((N^2_\infty)^\Gamma)^\vee=0$.
\end{proof}

Let $\tau\in \coh^1_{\mathrm{\acute et}}(\FF_p,\ZZ_p)=\Hom_{cont}(\Gal(\bar{\FF}_p/\FF_p),\ZZ_p)$ be the element which sends the arithmetic Frobenius to 1. By \cite[Lemma 6.9]{KT03} the composed map
$$\cup'\colon N_0^1\lr M^1_{1,0}\buildrel{{\bf 1}}\over\lr M^1_{2,0}\lr N_0^2$$
coincides up to sign with $\tau\cup$, the cup product by the image of $\tau$ in $\coh^1_{\mathrm{fl}}(X, \ZZ_p)$. Using \cite[Proposition 6.5]{Mil86b}, we can prove the following result.

\begin{lemma}\label{cup}
The composed map
$$\begin{CD} \cup\colon N^1_0\buildrel{\beta}\over\lr (N^1_\infty)^\Gamma @>{g_{N^1_\infty}}>> (N^1_\infty)_\Gamma \buildrel{\gamma}\over\lr N^2_0 \end{CD}$$
coincides (up to the sign) with the map
$$\tau\cup=\cup'\colon N^1_0\, \lr N^2_0.$$
\end{lemma}

\subsubsection{The height pairing}\label{iwasawapairing} Recall the N\'eron-Tate height pairing \eqref{e:nthp} of \S\ref{su:d}. Let $e_1,...e_r$ be elements of $A(K)$ which form a $\ZZ$-basis of $A(K)/A(K)_{tor}$ and let $e_1^*,...,e_r^*$ be elements of $A^t(K)$ which form a $\ZZ$-basis of $A^t(K)/A^t(K)_{tor}$. Then
$$Disc(\tilde h_{A/K}):=|\det(\tilde h_{A/K}(e_i,e_j^*)_{i,j})|\in\RR$$
is independent of the choices of basis and we call it the discriminant of the height pairing. It is known that $Disc(\tilde h_{A/K})\neq 0$. We write
$$Disc(h_{A/K}) = \log(p)^{-r} Disc(\tilde h_{A/K}),$$
with $r=\rank(A(K))$.

Consider the quotient category $(ab)/(fab)$, where $(ab)$ is the category of abelian groups and $(fab)$ the category of finite abelian groups. Let $\theta$ be the composed map in $(ab)/(fab)$ defined by:
\begin{equation} \label{e:relhp} \begin{CD}  A(K) \otimes \QQ_p/\ZZ_p @>{\alpha}>>  N^1_0 \buildrel{\beta}\over\lr (N^1_\infty)^\Gamma @>{g_{N^1_\infty}}>> (N^1_\infty)_\Gamma \\
@V{\theta}VV &&    @V{\gamma}VV\\
\Hom\big(A^t(K)/A^t_{p^\infty}(K), \QQ_p/\ZZ_p\big) @<{v^{-1}}<< \Hom(A^t(K), \QQ_p/\ZZ_p) @<{u}<< N^2_0
\end{CD}\end{equation}
where: \begin{enumerate}
\item the map $\alpha\colon A(K)\otimes \QQ_p/\ZZ_p \rightarrow N_0^1$ is the canonical morphism in $(ab)/(fab)$ coming from $A(K)\otimes \QQ_p/\ZZ_p \to \Sel_{p^\infty}(A/K)$, by \eqref{e:alg-arithm0};
\item the map $u\colon N^2_0\to \Hom(A^t(K), \QQ_p/\ZZ_p)$ is the map constructed by using the isomorphism \eqref{e:KT253} and the natural map $A^t(K)\otimes \ZZ_p\to \Sel_{\ZZ_p}(A^t/K)$;
\item the map $v$ is induced by the quotient map (which is an isomorphism, since we are in the quotient category $(ab)/(fab)$).
\end{enumerate}
Thanks to Lemma \ref{cup} and \cite[3.3.6.2 and \S6.8]{KT03}, $\theta$ coincides (up to sign) with the map induced by $h_{A/K}$ in $(ab)/(fab)$. In particular, since the N\'eron-Tate height pairing is non-degenerate, $\theta$ is a quasi-isomorphism (i.e., an isomorphism in the quotient category).

\subsubsection{Computation of the Euler characteristic} If $f$ is a quasi-isomorphism of abelian groups, we denote $char(f):=|\Ker(f)|/|\Coker(f)|$. Since this characteristic is multiplicative, \eqref{e:relhp} gives
$$char(\alpha)\cdot char(\beta)\cdot char(g_{N^1_\infty})\cdot char(\gamma)\cdot char(u)\cdot char(v)^{-1} = char(\theta) \equiv_p Disc(h_{A/K})$$
where $ \equiv_p$ means ``$\equiv$ mod $\ZZ_p^{\times}$''. If we assume that $A/K$ has semistable reduction and the Tate-Shafarevich group of $A/K$ is finite, then we have: \label{euler}
$$char(\alpha)=\frac{|A_{p^\infty}(K)|}{|\Sha_{p^\infty}(A/K)|\cdot |\M_0[p^\infty]|\cdot|N^0_0|}$$
by \eqref{e:alg-arithm0};
$$char(v)\equiv_p|A_{p^\infty}^t(K)|^{-1}$$
because $v\colon \Hom\big(A^t(K)/A^t_{p^\infty}(K), \QQ_p/\ZZ_p\big) \rightarrow \Hom(A^t(K), \QQ_p/\ZZ_p)$ is injective with cokernel $A^t_{p^\infty}(K)^\vee$;
$$char(\beta)=|(N^0_\infty)_\Gamma|$$
and
$$char(\gamma)=|(N^2_\infty)^\Gamma|^{-1}\,,$$
by \eqref{ses0} and \eqref{ses1}, using Lemma \ref{N02}; and $char(u)=1$ since the Tate-Shafarevich group is assumed to be finite.

\begin{mytheorem}\label{euchar}  Assume that $A/K$ has semistable reduction and the Tate-Shafarevich group of $A/K$ is finite. Then we have
$$ord_{\bf 1}(\cL_{A/\Kpinf})=\rank_\ZZ(A(K))=ord_{s=1}L_Z(A,s)$$
and
$$|\pounds_{\bf 1}(\cL_{A/\Kpinf})|_p^{-1}\equiv_p c_{BSD}\cdot|(N^2_\infty)_\Gamma|\,,$$
where $c_{BSD}$ is the leading coefficient at $s=1$ of $L_Z(A,s)$.
\end{mytheorem}

\begin{proof} By Lemma \ref{N02}, $(N^0_\infty)^\vee$ and $(N^2_\infty)^\vee$ have finite generalized $\Gamma$-characteristic.
For $(N^1_\infty)^\vee$, remark that under our assumption $\alpha$, $\beta$ and $v$ are quasi-isomorphisms while $u$ and $\gamma$ are isomorphisms. Therefore, the map $g_{N^1_\infty}$ is a quasi-isomorphism and so is its Pontryagin dual, $g_{(N^1_\infty)^\vee}$. By Theorem \ref{chi-Iwasawa-BSD} and Lemma \ref{N02}, we have
$$ord_{\bf 1}(\cL_{A/\Kpinf})=\rank_{\ZZ_p}X_p(A/K)=\rank_\ZZ A(K)\,,$$
since we have assumed that the Tate-Shafarevich group of $A/K$ is finite. The last equality $\rank_\ZZ A(K)=ord_{s=1}L_Z(A,s)$ follows from the main theorem of \cite{KT03}. The same theorem also proves that if $\alpha_v$ and $\mu_v$ are the Haar measure defined in \S\ref{ss:bsdinfty}, then
$$c_{BSD}=\frac{|\Sha(A/K)|\cdot Disc(h_{A/K})}{|A(K)_{tor}|\cdot|A^t(K)_{tor}|}\cdot\mu(Lie({\cal A})( \mathbb{A}_K)/Lie({\cal A})( K))^{-1}\cdot \prod_{v\in Z}\alpha_v(A(K_v))\,.$$
Replacing the values above and applying \eqref{e:tamagawa} and \eqref{e:tamagawa3} one gets
$$ c_{BSD} \equiv_p q^{-g(\deg(Z)+\kappa-1)-\delta}\cdot\frac{|(N^0_\infty)_\Gamma|\cdot char(g_{N^1_\infty})}{|N^0_0|\cdot|(N^2_\infty)^\Gamma|}\,. $$

On the other hand, by Theorem \ref{chi-Iwasawa-BSD},
$$|\pounds_{\bf 1}(\cL_{A/\Kpinf})|_p^{-1}=q^{-g(\deg(Z)+\kappa-1)-\delta}\cdot\frac{char(\Gamma,(N^1_\infty)^\vee)}{char(\Gamma,(N^0_\infty)^\vee)\,char(\Gamma,(N^2_\infty)^\vee)}$$
and so the result follows remembering \eqref{e:EulerMdual} and observing that if the $\La$-module $M^\Gamma$ has finite cardinality then $char(\Gamma,M)=|M_\Gamma|/|M^\Gamma|$ (note also that $(N^0_\infty)^\Gamma=N^0_0$).
\end{proof}

Without assuming the finiteness of the Tate-Shafarevich group, we have:

\begin{mytheorem} \label{t:ranks}
$$ord_{\bf 1}(\cL_{A/\Kpinf})=ord_{s=1}\big(L(A,s)\big)\geq \rank_{\ZZ_p} X_p(A/K).$$
\end{mytheorem}

\begin{proof} The last inequality has been proved in \cite[\S3.5]{KT03}. We show that the analytic rank is equal to the rank of our $p$-adic $L$ function $\cL_{A/\Kpinf}$. First, note that the operator $(\varphi_{i,0}\otimes\Fr_q)^\vee$ on $P^i_\infty$ is induced by the operator  $(p^{-1}F_{i,0}\otimes\Fr_q)^\vee$ on $(P^i_0[\frac{1}{p}]\otimes \QQ_{p,\infty})^\vee$. Moreover, using \eqref{exact1}, we observe that we have an injection of $\QQ_p[[\Gamma]]$-modules
$$P^i_\infty\otimes_{\La[\frac{1}{p}]}\QQ_p[[\Gamma]]\hookrightarrow (P^i_0[p^{-1}]\otimes \QQ_{p,\infty})^\vee \simeq\QQ_p[[\Gamma]]^{r_i}$$
with $r_i:=\dim_{\QQ_p}(P^i_0[\frac{1}{p}])$ and where the operator $(p^{-1}F_{i,0}\otimes\Fr_q)^\vee$ on the left-hand side corresponds to the operator $\Fr_q\cdot p^{-1}F_{i,0}$ on the right-hand side (as shown in Lemma \ref{induce}).
Also, Lemma \ref{linalg1} shows that $P^i_\infty$ is a free $\La[\frac{1}{p}]$-module of rank equal to the $\ZZ_p$-corank of $M^i_{2,0}$, which, by Lemma \ref{cofinite}, is precisely $r_i$.
Hence, the $p$-adic $L$-function $\cL_{A/\Kpinf}$ can be written
\begin{equation} \label{e:detfr}
\prod_{i=0}^2 {\det}_{\QQ_p[[\Gamma]]}\left(id-\Fr_q\cdot p^{-1}F_{i,0},P^i_\infty\otimes\QQ_p[[\Gamma]]\right)^{(-1)^{i+1}}=\prod_{i=0}^2(\prod_{j=1}^{r_i}(1-\alpha_{ij}\Fr_q))^{(-1)^{i+1}},
\end{equation}
where the $\alpha_{ij}$'s are the eigenvalues (in ${\bar \QQ}_p$) of $p^{-1}F_{i,0}$. In particular, $ord_{\bf 1}(\prod_{j=1}^{r_i}(1-\alpha_{ij}\Fr_q))$ is the number of $\alpha_{ij}$ equal to $1$ (note that $1-\lambda\Fr_q$ has order 0 if $\lambda\neq 1$ and order 1 else), that is, the multiplicity of the eigenvalue 1 of the operator $p^{-1}F_{i,0}$ and the assertion follows from \cite[3.5.2]{KT03}.
\end{proof}

\subsection{Comparison with the constant ordinary case} \label{su:2dreams}
Now we assume that $A/K$ is a constant abelian variety and compare the Main Conjecture of Part \ref{part:semistable} (i.e., Theorem \ref{GIMC3}) with the Main Conjecture of Part \ref{part:constantA} (that is, Theorem \ref{t:imc3constant}).
Let $\mathbf{A}/\FF$ be such that $A=\mathbf{A}\times_\FF \Spec K$ and $\A=\mathbf{A}\times_\FF C$.

We have $Z=\emptyset$. Thus, \eqref{e:interpoLtilde} and Theorem \ref{GIMC2} imply that for all characters of $\Gamma$
$$\omega(\tilde{\cL}_{A/\Kpinf}) = L(A,\omega,1) = \omega(\cL_{A/\Kpinf}).$$
Therefore,
\begin{equation}\label{e:LL}
\tilde{\cL}_{A/\Kpinf}=\cL_{A/\Kpinf}.
\end{equation}

\begin{mytheorem} Let $A/K$ be a constant ordinary abelian variety. Then {\em{Theorem \ref{t:imc3constant}}} is equivalent to {\em{Theorem \ref{GIMC3}}}.
\end{mytheorem}

\begin{proof}
First, by \eqref{e:cLtilde} and \eqref{e:LL}
$$ \theta_{A,\Kpinf,\emptyset}= q^{g(\kappa-1)}\Fr_q^{g(2-2\kappa)}(\prod_{i=1}^g\alpha_i^{2-2\kappa}\big) \cdot \big(\prod_{i=1}^g(1-\alpha_i^{-1}\Fr_q)(1-\alpha_i^{-1}\Fr_q^{-1})\big) \cdot \cL_{A/\Kpinf} \,. $$
Note that $\Fr_q^{g(2-2\kappa)}$ and $\prod\alpha_i^{2-2\kappa}$ are units in $\La$. Thus, it is sufficient to show that
$$\star_{A,K_\infty^{(p)}}= q^{-g(\kappa-1)}\cdot \big(\prod_{i=1}^g(1-\alpha_i^{-1}\Fr_q)(1-\alpha_i^{-1}\Fr_q^{-1})\big)^{-1}.$$
On the other hand Proposition \ref{p:summarize} yields
$$\star_{A,K_\infty^{(p)}}= q^{-g(\kappa-1)}\cdot \big(\prod_{i=1}^l(1-\zeta_i^{-1}\Fr_q)(1-\zeta_i^{-1}\Fr_q^{-1})\big)^{-1}$$
as we have $Z=\emptyset$ and $\delta=0$ because we can choose $e_1,...,e_g$ to be a basis of $Lie( \mathbf{A})(\FF)$ so that $f_v$ can be taken to be $e$. Besides $\{\zeta_1,...,\zeta_l\}\subset\{\alpha_1,...,\alpha_g\}\,,$ since they are the eigenvalues of the Galois action respectively on $T_pA(\Kpinf)[p^\infty]$ and $T_pA$. Let $\tilde\Gamma$ and $H$ as in \eqref{e:prod2}. The maps $\Fr_q\mapsto\alpha_i$ determine homomorphisms of $\tilde\Gamma$ onto some group of $p$-adic units and  the equality $T_pA(\Kpinf)[p^\infty]=(T_pA)^H$ yields that the image of $H$ is non-trivial for $i>l$. Hence $i>l$ implies $\alpha_i\not\equiv1$ mod $p$ and thus $(1-\alpha_i\Fr_q^{\pm1})\in\La^\times$.
\end{proof}

\section{Iwasawa theory for general coefficients and open questions}

Let $D$ be a log Dieudonn\'e crystal over $C^\#/W(\FF)$ and define the cohomology groups $N_{D,\infty}^i$, $L^j_{D,\infty}$ and $P^i_{D,\infty}$ similarly to those in \S\ref{ss:cohomtheories} and \S\ref{ss:Pinfty}. Let $D^\dagger$ be the overconvergent $F$ isocrystal over $U$ associated with $D$ (remember that $U$ is the dense open subset of $C$ where the log structure is trivial).

We can prove (exactly as in Subsection \ref{su:cohom}) the following:

\begin{mytheorem} Assume that $(N^0_{D,\infty})^\vee$ is $\La$-torsion. Then, $(N_{D,\infty}^i)^\vee$ ($i=0,1,2$) and $(L^j_{D,\infty})^\vee$ ($j=0,1$) are finitely generated torsion $\La$-modules.
\end{mytheorem}

Hence, if $(N^0_{D,\infty})^\vee$ is $\La$-torsion (for example, this happens when $N^0_{D,0}$ is a finite group), we define
$$f_{D}:=\frac{f_{(N^1_{D,\infty})^\vee}f_{(L^0_{D,\infty})^\vee}}{f_{(N^0_{D,\infty})^\vee}f_{(N^2_{D,\infty})^\vee}f_{(L^1_{D,\infty})^\vee}}\in Q(\La)^\times/\La^\times$$
and call it the characteristic element associated with $D$ relative to the arithmetic $\ZZ_p$-extension of $K$.

\subsubsection{Iwasawa main conjecture for log Dieudonn\'e crystals} As in Section \ref{s:semistab}, the element
$$\cL_D:=\prod_{i=0}^2 det_{\La[\frac{1}{p}]}\left(id-\Phi_i,P^i_{D,\infty}\right)^{(-1)^{i+1}}$$
is well defined and we call it the $p$-adic $L$-function associated with $D$ relative to the arithmetic $\ZZ_p$-extension of $K$.

We get as in Theorems \ref{GIMC2}, \ref{GIMC3} and \ref{chi-Iwasawa-BSD}:

\begin{mytheorem}\label{GGIMC2} For any Artin character $\omega\colon\Gamma\to \bar\QQ_p^\times$ and for any log Dieudonn\'e crystal $D$, $$\omega(\cL_D)=L(U,D^\dagger,\omega,1).$$
\end{mytheorem}

\begin{mytheorem}\label{GGIMC3} Let $D$ be a log Dieudonn\'e crystal over $C^\#/W(\FF)$ and assume that $(N^0_{D,\infty})^\vee$ is $\La$-torsion. We have the following equality in $Q(\La)^\times/\La^\times$:
$$\cL_{D}=f_D$$
\end{mytheorem}

\begin{mytheorem}\label{Gchi-Iwasawa-BSD} Let $\omega\colon\Gamma\to O^\times$ be a character. Let $D$ be a log Dieudonn\'e crystal over $C^\#/W(\FF)$ and assume that $(N^0_{D,\infty})^\vee$ is $\La$-torsion. Assume furthermore that, for $i=0,1,2$, the $\La$-modules $(N^i_{D,\infty})^\vee(\omega)$ have finite generalized $\Gamma$-characteristic. Then
$$ord_\omega\left(\cL_D\right)=\sum_{i=0}^2 (-1)^{i+1}rank_{\ZZ_p}\left(\left((N^i_{D,\infty})^\vee(\omega)\right)^\Gamma\right)$$
and
$$|\pounds_\omega(\cL_{D})|_p^{-1}=\prod_{i=0}^2 char\left(\Gamma,(N^i_{D,\infty})^\vee(\omega)\right)^{(-1)^{i+1}}p^{d_\omega(d(L^0_{D,0})-d(L^1_{D,0}))}.$$
\end{mytheorem}

\subsubsection{Finiteness of the crystalline cohomology in a ramified tower?} If $L/K$ is a $\ZZ_p$-extension different from $\Kpinf$, then it will be ramified at some places (and possibly even at infinitely many). The method above does not work any more because the modules over $L$ are not obtained by base extension. It might be interesting to prove analogues of ({\bf IMC2}) and ({\bf IMC3}) in the following cases:
\begin{enumerate}
\item Assume that $A/K$ has at worst semistable reduction. Consider a $\ZZ_p$-extension $L/K$ such that $X_p(A/L)$ is $\La$-torsion. The $\La$-modules $(L^i_\infty)^\vee$ and $(N^i_\infty)^\vee$ are also $\La$-torsion using \S\ref{ss:bsdinfty} and Lemma \ref{N02}. Therefore, under these assumptions, we can define $f_{A/L}$ as in \eqref{e:charelement}. The problem is to define the $p$-adic $L$-function.\\

{\bf Question:} {\em is $(M^i_{k,\infty})^\vee$ a finitely generated $\La$-module of finite type?}\\

\noindent If the answer to this question is yes, then the $p$-adic $L$-function as defined before is well-defined and Theorem \ref{GIMC3} holds.

\item For a non-isotrivial elliptic curve $E/K$ with ordinary generic fibre, we know that the supersingular locus is a finite set of points $\Sigma_{ss}$ (this is not true in general for an abelian variety). Let $\Sigma_{bad}$ denote the finite set of places of $C$ where $E$ has bad reduction, $\Sigma_{ram}$ denote the places of $K$ ramifying in $L$ and $\Sigma:=\Sigma_{ss}\cup\Sigma_{bad}\cup\Sigma_{ram}$. We assume that $\Sigma_{ram}$ is finite and $\Sigma_{ram}\cap \Sigma_{ss}=\emptyset$. Put $U:=C-\Sigma$. Then if $D(E)$ denotes the (overconvergent) covariant Dieudonn\'e crystal associated with the N\'eron model of $E$ over $U$, we have (see \cite{crw92}) a short exact sequence of Dieudonn\'e crystals
$$0\lr M^*(1)\lr D(E)\lr M\lr 0,$$
where $M$ is a unit-root $F$-isocrystal (not necessarily with finite monodromy) and $M^*(1)$ is the dual Dieudonn\'e crystal. The geometric Iwasawa Main Conjecture for $E/K$ relative to the $\ZZ_p$-extension $L/K$ should reduce in this case to the conjecture for $M$ and $M^*(1)$.
\end{enumerate}

\end{part}

\appendix

\section{An example of non-cotorsion Selmer group}
In this appendix we provide an example where $X_p(A/L)$ is a non-torsion $\La$-module, with $K$ a global field of characteristic $p>0$ and $L/K$ a $\ZZ_p$-extension. The case $p=2$ is allowed.\\

Let $k\subset K$ be a subfield such that $K/k$ is a separable quadratic extension with $\Gal(K/k)=\{1,\tau\}$. We say that an abelian extension $M/K$ is anticyclotomic with respect to $k$ if $M/k$ is Galois and $\Gal(M/k)$ is dihedral: that is, we have a decomposition
$$\Gal(M/k)\simeq\Gal(M/K)\rtimes\Gal(K/k)$$
so that for any $\sigma\in\Gal(M/K)$ we have $\tau\sigma\tau^{-1}=\sigma^{-1}$ (where $\tau$ is abusively identified with its lift to $\Gal(M/k)$).

In this appendix we specialize $A$ to be a non-isotrivial semistable elliptic curve of analytic rank 0 with split multiplicative reduction at a place $v_0$ of $K$ and take as $L/K$ a $\ZZ_p$-extension, totally ramified above $v_0$, unramified elsewhere and anticyclotomic with respect to $k$ (class field theory provides plenty of such $L$). Furthermore, $A$ is assumed to be already defined over $k$ and to have split multiplicative reduction at the restriction of $v_0$ to $k$. We are going to show that under these conditions $X_p(A/L)$ cannot be a torsion $\Lambda$-module.\\

In the following, with a slight abuse of notation we shall often use the same symbol to denote places in different fields: e.g., $L_{v_0}$ will be the completion of $L$ at the only place above $v_0$. Also, for $v$ a place of $K$ put $\Gamma_v:=\Gal(L_v/K_v)$. Since $\Gamma_{v_0}$ can be identified with $\Gamma$, we won't distinguish between the two.

\subsubsection{} By \cite{Ta66} (and \cite{mil75} for the $p$-part of $\Sha$ - see also the comments on Milne's webpage http://www.jmilne.org/math/articles/index.html for $p=2$) it is known that analytic rank 0 implies that the full Birch and Swinnerton-Dyer conjecture holds for $A/K$. Therefore the groups $A(K)$, $\Sha(A/K)$ and $\Sel_{p^\infty}(A/K)$ are all finite.

Since $A$ has split multiplicative reduction at $v_0$, it is a Tate curve on the completion $K_{v_0}$: we denote the local Tate period by $Q$.

\begin{remark} \label{r:stabletor} {\em We claim that $A_{tor}(K)=A_{tor}(L)$. To see it, first notice that the constant field does not grow in $L_{v_0}/K_{v_0}$, since it is a totally ramified extension. Besides, if a root $Q^{1/n}$ of $Q$ is not already in $K_{v_0}$ then it cannot  belong to $L_{v_0}$, either because then $Q^{1/n}$ is not separable over $K_{v_0}$ (for $n$ a power of $p$) or because $L/K$ is a $p$-extension (for $(p,n)=1$). The explicit description of the torsion points of $A(L_{v_0})$ in terms of roots of unity and of $Q$ then implies $A_{tor}(K_{v_0})=A_{tor}(L_{v_0})$. To conclude, it suffices to observe that $L\cap K_{v_0}=K$, exploiting once again the totally ramified hypothesis.} \end{remark}

\subsubsection{} There are many known instances of elliptic curves satisfying the hypotheses assumed in this appendix. For example, let $\FF$ be a finite field of characteristic $p>2$ and consider the function fields $k=\FF(s)$ and $K=\FF(t)$, with $s=t^2$.
Let $A$ be defined by the Weierstrass equation
\begin{equation} \label{e:appbeau} A: y^2=x(x+1)(x+s)=x(x+1)(x+t^2). \end{equation}
This is an elliptic curve having split multiplicative reduction at $t=0$, $t=\infty$ and (if $-1$ is a square in $\FF$) $t=\pm 1$, and good reduction at all other places. Therefore, as a divisor of $K$ the conductor $\fn$ of $A/K$  is the sum of these four places. The elliptic curve $A/K$ is well known to have Hasse-Weil function $L(A/K,s)\equiv 1$: one way to prove it is to observe that $A$ is associated with a modular form and apply the corollary of \cite[Proposition 3]{tan93}; another approach (closer to the methods of \cite{Ta66}) is explained in \cite{shi92}.

Equation \eqref{e:appbeau} is an instance of Beauville surface: they exist in all characteristics and always satisfy the hypotheses of this appendix. See \cite{lan91} for a full discussion and classification.

\subsection{The algebraic side} In the following, we fix a topological generator $\gamma$ of $\Gamma$ and put $T:=\gamma-1\in\Lambda$; then $T$ is a generator of the augmentation ideal $I$. By abuse of notation, we also identify $\gamma$ with a generator of $\Gamma_n$ for all $n$.

\subsubsection{} Since $L/k$ is Galois, our ramification hypotheses imply that the $\Gal(K/k)$-orbit of $v_0$ contains no other place of $K$. Besides, we are assuming that $A$ is already a Tate curve over $k_{v_0}$ and thus we have $Q\in k_{v_0}^\times$.

Let $\N\subseteq K_{v_0}^\times$ denote the group of universal norms of the local extension $L_{v_0}/K_{v_0}$. Write ${\rm rec}$ for the reciprocity map of local class field theory: then we have ${\rm rec}(\tau (x))=\tau {\rm rec}(x)\tau^{-1}$ for any $x\in K_{v_0}^\times$, and hence $\tau(x)\equiv x^{-1}\mod\N$. In particular $Q\equiv Q^{-1}\mod\N$, so that $Q^2\in\N$. As $K_{v_0}^\times/\N\simeq\Gamma$ is torsion free, we deduce that $Q\in\N$.

\begin{lemma} \label{l:apploccoh} Let $v$ be a place of $K$. Then
$$\coh^1(\Gamma_v,A(L_v))\simeq \begin{cases}
0 & \text{if } v \text{ is a place of good reduction;}  \\
\text{a finite group}& \text{if } v \text{ is an unramified place of bad reduction;}\\
\QQ_p/\ZZ_p & \text{if }  v=v_0.  \end{cases}$$
\end{lemma}

\begin{proof} For unramified places, this is a consequence of \cite[I, Proposition 3.8]{mil86a}. As for $v_0$, observe that for any $n$ we have $A(K_{n,v_0})\simeq K_{n,v_0}^\times/Q^{\ZZ}.$ We deduce the exact sequence
$$\coh^1(\Gamma_n,K_{n,v_0}^\times)\lr \coh^1(\Gamma_n,A(K_{n,v_0}))\lr \coh^2(\Gamma_n,Q^{\ZZ})\lr \coh^2(\Gamma_n,K_{n,v_0}^\times)\,,$$
that we can rewrite
\begin{equation} \label{e:seqcohloc}  0\lr \coh^1(\Gamma_n,A(K_{n,v_0}))\lr Q^{\ZZ}/Q^{p^n\ZZ}\lr K_{v_0}^\times/N_{K_{n,v_0}/K_{v_0}}(K_{n,v_0}^\times)\,. \end{equation}
Since $Q\in\N$, the map $Q^{\ZZ}\to K_{v_0}^\times/N_{K_{n,v_0}/K_{v_0}}(K_{n,v_0}^\times)$ is trivial. Thus we obtain
\begin{equation} \label{e:H1loc} \coh^1(\Gamma_n,A(K_{n,v_0}))\simeq p^{-n}\ZZ/\ZZ\,. \end{equation}
As $n$ varies, this isomorphism is compatible with the inflation maps on the left and the canonical inclusions on the right, thereby proving the assertion.
\end{proof}

\begin{corollary} \label{c:appcorank} The group $\Sel_{p^\infty}(A/L)^\Gamma$ has $\ZZ_p$-corank at most 1.
\end{corollary}

\begin{proof} Applying the snake lemma to the diagram
$$\begin{CD} 0 @>>> \Sel_{p^\infty}(A/K) @>>> \coh^1_{\mathrm{fl}}(K,A_{p^\infty}) @>{\cL_K}>> \bigoplus_v \coh^1(K_v,A) \\
&& @VVV  @VVV  @VVV \\
0 @>>> \Sel_{p^\infty}(A/L)^\Gamma@>>> \coh^1_{\mathrm{fl}}(L,A_{p^\infty})^\Gamma @>{\cL_L}>> \bigoplus_v \coh^1(L_v,A)^{\Gamma_v} \end{CD}$$
we get an exact sequence
$$\coh^1(\Gamma,A_{p^\infty}(L)) \lr H \lr \Sel_{p^\infty}(A/L)^\Gamma/\Sel_{p^\infty}(A/K) \lr \coh^2(\Gamma,A_{p^\infty}(L))=0,$$
where $H$ is a subgroup of $\oplus_v\coh^1(\Gamma_v,A(L_v))$ and the equality on the right comes from the fact that $\Gamma$ has cohomological dimension 1. The claim follows from Lemma \ref{l:apploccoh}, since $A_{p^\infty}(L)$ and $\Sel_{p^\infty}(A/K)$ are finite groups.
\end{proof}

\begin{remark} \label{r:H1Ginvariant}
{\em The group $G:=\Gal(K/k)$ acts on $\coh^1(K,A)$ and the other cohomology groups appearing in the above discussion. In particular, since $v_0$ does not split in $K/k$, $G$ acts on $\coh^1(\Gamma,A(L_{v_0}))$. The connecting homomorphism in \eqref{e:seqcohloc} is a map of $G$-modules and so is the isomorphism \eqref{e:H1loc}, because $Q\in k_v$. Therefore we have
$$\coh^1(\Gamma,A(L_{v_0})) = \coh^1(\Gamma,A(L_{v_0}))^G\,. $$}
\end{remark}

\subsubsection{}\label{ss:E} Consider the localization map:
$$loc_K\colon \coh^1(K,A)\lr \bigoplus_v \coh^1(K_v,A).$$
The generalized Cassels-Tate dual exact sequence of \cite[Main Theorem]{gt07} identifies, for any integer $m$, the cokernel of $loc_K$ with the Pontryagin dual of
$$T_m\Sel(A^t/K):=\varprojlim\Sel_{m^n}(A^t/K)\,.$$
In our case $A$ is an elliptic curve, so $A=A^t$. Taking the inverse limit of the sequence
$$0\lr A(K)/m^nA(K) \lr \Sel_{m^n}(A/K) \lr \Sha(A/K)[m^n] \lr 0$$
we see that $T_m\Sel(A/K)$ is always finite, and trivial for almost every $m$, because so are $A(K)$ and $\Sha(A/K)$. Hence $loc_K$ has finite cokernel and, by the inclusions
$$\coh^1(\Gamma,A(L_{v_0}))\subseteq \coh^1(K_{v_0},A)\subseteq \bigoplus_v \coh^1(K_v,A)\,,$$
it follows that
$$\cH:=\coh^1(\Gamma,A(L_{v_0}))\cap loc_K\big(\coh^1(K,A)\big)$$
has finite index in $\coh^1(\Gamma,A(L_{v_0}))$. Since the latter is $p$-divisible, we must have
\begin{equation} \label{e:imlocK} \cH=\coh^1(\Gamma,A(L_{v_0}))\subseteq loc_K\big(\coh^1(K,A)\big). \end{equation}
Write $E:=loc_K^{-1}\big(\coh^1(\Gamma,A(L_{v_0}))\big)$ and let $res_{L/K}\colon\coh^1(K,A)\rightarrow\coh^1(L,A)$ denote the restriction map. We have an exact sequence
\begin{equation} \label{e:E} \begin{CD} 0\,\lr \Sha(A/K)\,\lr E @>{loc_K}>> \coh^1(\Gamma,A(L_{v_0}))\lr 0. \end{CD}\end{equation}
It follows that $loc_L(res_{L/K}(E))=0$, because restriction and localization commute and $\coh^1(\Gamma,A(L_{v_0}))$ has trivial image in $\oplus_v \coh^1(L_v,A)$.

Let $D$ be the divisible part of $E$: then $D\simeq\QQ_p/\ZZ_p$, since $\Sha(A/K)$ is a finite group. By construction $D$ is a subgroup of $\coh^1(K,A)$, killed by $loc_L\circ res_{L/K}$: that is,
\begin{equation} \label{e:Dinsha} res_{L/K}(D)\subset \Sha_{p^\infty}(A/L)^\Gamma. \end{equation}

\subsubsection{} For the proof that $X_p(A/L)$ is not torsion, we are going to reason by contradiction. Thus in the following we assume that $X_p(A/L)$ is torsion.

Let $\fA$ be the Cassels-Tate $\Gamma$-system defined in Section \ref{s:al}. Since $A=A^t$,
$$\fa_n=\fb_n=\Sha_{p^\infty}(A/K_n)/\Sha_{p^\infty}(A/K_n)_{div}=\Sel_{p^\infty}(A/K_n)/\Sel_{p^\infty}(A/K_n)_{div}.$$
By \eqref{e:axy}, our hypothesis on $X_p(A/L)$ implies that $\fa=\fb$ is actually torsion. The Cassels-Tate pairing induces, for every $n$, a perfect alternating pairing:
$$\langle\,,\rangle_n\colon \fb_n\times \fb_n\lr \QQ_p/\ZZ_p\,.$$
Then the module $\fb_{\infty}:=\varinjlim \fb_m$ is identified with the Pontryagin dual of $\fb$.

\begin{proposition} \label{p:botimesLaI} The $\ZZ_p$-rank of $\fb/T\fb$ equals $1$.
\end{proposition}

\begin{proof} First, note that, by Lemma \ref{l:zerocoker}, the finiteness of $\Sel_{p^\infty}(A/K)$ implies $\Sel_{p^\infty}(A/K_n)^\Gamma$ is finite for every $n\geq 0$.
Then Lemma \ref{l:zerocoker} and Lemma \ref{l:fingen} imply that for $n\gg 0$,
$$\Sel_{div}(A/L)^\Gamma=(\Sel_{div}(A/L)^{\Gamma^{(n)}})^\Gamma=(\Sel_{p^\infty}(A/K_n)_{div})^\Gamma\subset \Sel_{p^\infty}(A/K_n)^\Gamma$$
is also finite. By duality, the $\ZZ_p$-module $Y_p(A/L)/T Y_p(A/L)$ is finite. This also implies that $Y_p(A/L)[T]$ is finite, because $Y_p(A/L)$, as a submodule of the $\ZZ_p$-free part of $X_p(A/L)$, is a finitely generated $\ZZ_p$-module. Since \eqref{e:axy} and the snake lemma yield the exact sequence
$$ Y_p(A/L)[T] \lr \fb/T\fb \lr X_p(A/L)/T X_p(A/L) \lr Y_p(A/L)/T Y_p(A/L)\,,$$
it follows from Corollary \ref{c:appcorank} that the $\ZZ_p$-rank of $\fb/T\fb$ is at most $1$.

For the other inequality, consider the composition
\begin{equation} \label{e:piD} \pi\colon\xymatrix{D\ar[rr]^-{{res_{L/K}}} & & \Sha_{p^{\infty}}(A/L)^\Gamma \ar[r] &  \fb_\infty^{\Gamma}}\,. \end{equation}
Let $\AE$ denote the preimage of $D$ under the natural surjection from $\coh^1(K,A_{p^\infty})$ to the $p$-primary part of $\coh^1(K,A)$. Because $A(K)$ is finite, the exact sequence
$$0\lr \QQ_p/\ZZ_p\otimes A(K)\lr\AE \lr D\lr 0$$
implies that $\AE$ is of corank $1$ over $\ZZ_p$. By Lemma \ref{l:h1d}, the kernel of the restriction map $\coh^1(K,A_{p^\infty})\rightarrow  \coh^1(L,A_{p^\infty})$ is of finite order, and hence
$res_{L/K}(\AE)\subset \Sel_{p^\infty}(A/L)^{\Gamma}$ is also of corank $1$ over $\ZZ_p$.
If the image of $\pi$ were finite (and thus trivial, as $D$ is $p$-divisible) then $res_{L/K}(\AE)$ would be contained in $\Sel_{div}(A/L)^{\Gamma}$, which has just been shown to be finite. This is absurd. Therefore, the corank of $\fb_{\infty}^\Gamma$ is at least $1$, and by the duality the rank of $\fb/T\fb$ is at least $1$.
\end{proof}

Proposition \ref{p:botimesLaI} implies that there exist some $r\geq 1$ and $\xi_i\in\La$, $i=1,...,s$, coprime to $T$ so that
\begin{equation}\label{e:fbTfb}
[\fb]\sim \La/T^r\La\oplus\bigoplus_i^s\La/\xi_i\La.
\end{equation}

\begin{lemma} \label{l:bntorsion} Let $r$ be as in \eqref{e:fbTfb}. Then the indexes of the subgroups $\fk_n(\fb)\subset \fb_n$ and $\fk_n(\fb[T^r])\subset \fb_n[T^r]$ are bounded as $n$ varies.
\end{lemma}

\begin{proof} The same argument already used in the proof of Proposition \ref{p:N2trivial} shows that $\Sel_{p^\infty}(A/K_n)_{div}$ stabilizes for $n\gg0$. Thus, by \eqref{e:bnabvar} and the inequality \eqref{e:ddd}, the kernel of the map $\fb_n\rightarrow \fb_\infty$ is of bounded order. By duality the cokernel of $\fk_n\colon\fb\rightarrow\fb_n$ is also of bounded order.
Denote $\fd:=\fb/\fb[T^r]$ and $\fd_n:=\fk_n(\fd)$. By \eqref{e:fbTfb}, $\fd$ is annihilated by some $g\in\La$ coprime to $T$. Then we have $p^s\in (T^r,g)\subset \La$ for some $s\geq 1$.
Therefore, if $a_1,...,a_l$ are generators of $\fd$ over $\ZZ_p$, then the order of  $\fd_n[T^r]$ is bounded by $p^{sl}$.
\end{proof}

\subsubsection{} In order to exploit the anticyclotomic assumption, we now consider the action of $G:=\Gal(K/k)=\langle\tau\rangle$ on $\fb$.
We lift $G$ to a subgroup of $\Gal(L/k)$. Note that the maps $loc_K$ of \eqref{e:E} and $\pi$ of \eqref{e:piD} are both compatible with the action of $G$.

\begin{lemma} \label{l:bGinv} There exists $x\in\fb[T^r]^G$ such that $[\fb[T^r]:\La x]<\infty$. \end{lemma}

\begin{proof} Remark \ref{r:H1Ginvariant} yields that $(1-\tau)D$ is contained in $\Ker(loc_K)$ and hence is trivial (because $D$ is divisible). Thus we have $\pi(D)=\pi(D_G)=\pi(D)_G$ and by duality (see the proof of Proposition \ref{p:botimesLaI}) it follows that $(\fb/T\fb)^G$ has rank 1 over $\ZZ_p$. Also the image of $\fb[T^r]$ in $\fb/T\fb$ is of $\ZZ_p$-rank $1$, by \eqref{e:fbTfb}, so there must be $y\in\fb[T^r]$ such that $y$ mod $T\fb$ is $G$-invariant and has infinite order. Choose $x:=(1+\tau)y$. Then $x\in\fb[T^r]^G$ and $\La x$ has finite index in $\fb[T^r]\sim\La/T^r\La$ since $x\equiv 2y\pmod{T\fb}$ generates a free $\ZZ_p$-module in $\fb/T\fb$.
\end{proof}

\noindent Let $\fc:=\La x$ and $\fc_n:=\fk_n(\fc)$.

\begin{lemma} \label{l:Trbndual} As $n$ varies, the orders of the cokernels of the maps $\fc_n\rightarrow \fb_n[T^r]$ and $\fc_n\rightarrow \fb_n\rightarrow \fb_n/T^r\fb_n$ are bounded. \end{lemma}

\begin{proof} By \eqref{e:fbTfb} the cokernel of $\fb[T^r]\rightarrow \fb\rightarrow \fb/T^r\fb$ is finite. Then apply Lemma \ref{l:bntorsion} and Lemma \ref{l:bGinv}.
\end{proof}

\begin{lemma}\label{l:keyhere} We have
$$\langle\fc_n,\fc_n\rangle_n\subset (\QQ_p/\ZZ_p)[2].$$
\end{lemma}

\noindent So $\langle\fc_n,\fc_n\rangle_n$ has at most 2 elements and it is trivial if $p\not=2$.

\begin{proof} Write $x_n:=\fk_n(x)$. For any $\lambda\in\La$ we have
\begin{equation} \label{e:appc+} \langle x_n,\lambda x_n\rangle_n = - \langle \lambda x_n,x_n\rangle_n = - \langle x_n,\lambda^\# x_n\rangle_n \,, \end{equation}
using first the fact that $\langle \;,\;\rangle_n$ is alternating and then its $\Gamma$-equivariance. On the other hand, the pairing is also $G$-invariant and $x=\tau x$ implies
\begin{equation} \label{e:appc-} \langle x_n,\lambda x_n\rangle_n = \langle \tau x_n,\tau(\lambda x_n)\rangle_n = \langle x_n,\lambda^\# x_n\rangle_n \,, \end{equation}
because $\tau(\lambda x)=(\tau\lambda\tau^{-1})\tau x$ and the action of $\tau$ on $\La$ is precisely $\lambda\mapsto\lambda^\#$, by the anticyclotomic hypothesis. Equalities \eqref{e:appc+} and \eqref{e:appc-} together prove
$$2\langle x_n,\lambda^\# x_n\rangle_n=0$$
and this suffices, since $\fc_n=\La x_n$.
\end{proof}

\subsubsection{} Now we can finally obtain a contradiction. The Cassels-Tate pairing makes $\fb_n[T^r]$ and $\fb_n/T^r\fb_n$ dual to each other . Lemma \ref{l:Trbndual} implies that the subgroup
$$\langle \fc_n,\fc_n\rangle_n\subseteq \langle \fb_n[T^r],\fb_n/T^r\fb_n\rangle_n$$
has bounded index as $n$ varies. Since $\pi(D)\subseteq \fb_\infty[T]\subseteq \fb_\infty[T^r]=\varinjlim_n \fb_n[T^r]$ is infinite, we must have
$$\bigcup_n \langle \fc_n,\fc_n\rangle_n=\QQ_p/\ZZ_p\,,$$
a contradiction to Lemma \ref{l:keyhere}.

\subsection{The analytic side} In this example $X_p(A/L)$ is non-torsion, whence its characteristic ideal is trivial.
In the spirit of the Iwasawa main conjecture one expects that the corresponding $p$-adic $L$-function should be 0. Here we verify this.\\

In the function field setting, non-isotrivial elliptic curves are known to be modular: that is, there is a cuspidal automorphic function $f$ associated with $A$; its level is $\fn$, the conductor of $A/K$ (as a divisor of $K$). For any $n\geq 0$, let $\fd_n$ denote the divisor $nv_0$ and let $K(\fd_n)/K$ be the corresponding ray class field. It is shown in \cite{tan93} how to construct a modular element $\Theta_{\fd_n,f}\in\ZZ_p[\Gal(K(\fd_n)/K)]$, such that for each $\omega\in\Gal(K(\fd_n)/K)^\vee$ one has
\begin{equation}  \label{e:apptheta}  \omega(\Theta_{\fd_n,f})=\tau_{\omega}\cdot L(A,\omega,1)\,, \end{equation}
where $\tau_{\omega}$ is a Gauss sum.

By \cite[Proposition 2, 2.(d)]{tan93}, the maps
$$\ZZ_p[\Gal(K(\fd_{n+1})/K)]\lr \ZZ_p[\Gal(K(\fd_n)/K)]$$
send the modular elements $\Theta_{\fd_n,f}$ into each other, so that one can take their limit $\tilde\Theta$. Any abelian extension of $K$ totally ramified above $v_0$ and unramified elsewhere is contained in $\cup K(\fd_n)$: in particular this holds for our $L$. Let $\Theta$ be the image of $\tilde\Theta$ under the projection
$$\varprojlim\ZZ_p[\Gal(K(\fd_n)/K)]\lr \La.$$
Equation \eqref{e:apptheta} shows that $\Theta$ satisfies the interpolation property required for the $p$-adic $L$-function.

Observe that $\Theta_{\fd_n, f}$ is invariant under the action of $\Gal(K/k)$, since the modular elements are already defined above $k$. On the other hand, by \cite[Proposition 3]{tan93} we have
$$\Theta_{\fd_n,f}= - \Theta_{\fd_n,f}^{\#} \cdot \eta,$$
where $\eta\in\Gal(K(\fd_n)/K)$ corresponds to the divisor $\fn'=\fn-v_0$. This implies that $\Theta=0$.


\end{document}